\documentclass[a4paper,twoside]{article}

\usepackage[utf8]{inputenc}
\usepackage[T1]{fontenc}
\usepackage{hyperref}
\hypersetup{colorlinks,linkcolor={red},citecolor={blue},urlcolor={green!70!black}}
\usepackage{amsfonts,amssymb,bm,stmaryrd}
\usepackage{amsmath}
\numberwithin{equation}{section}
\usepackage{tikz-cd}
\usepackage{amsthm}
\newtheoremstyle{mytheoremstyle}{7pt}{7pt}{\normalfont}{}{\normalfont\bfseries}{:}{.5em}{}
\theoremstyle{mytheoremstyle}
\newtheorem{definition}{Definition}[section]
\newtheorem{lemma}[definition]{Lemma}
\newtheorem{proposition}[definition]{Proposition}
\newtheorem{theorem}[definition]{Theorem}

\newtheorem{example}[definition]{Example}
\newtheorem{remark}[definition]{Remark}

\newcommand{\SetFont}[1]{\mathbb{#1}}
\newcommand{\setN}{\SetFont{N}}
\newcommand{\setZ}{\SetFont{Z}}
\newcommand{\setR}{\SetFont{R}}
\newcommand{\Ker}{\operatorname{Ker}}
\newcommand{\Ima}{\operatorname{Im}}
\newcommand{\smooth}{\mathcal{C}^\infty}
\newcommand{\anchor}{\mathsf{a}}
\newcommand{\dif}{\mathrm{d}}
\newcommand{\difDeRham}{\bm{\dif}}
\newcommand{\opeLie}{\mathcal{L}}
\newcommand{\opeLieDeRham}{\bm{\opeLie}}
\newcommand{\opeIns}{\iota}
\newcommand{\opeInsDeRham}{\bm{\opeIns}}
\newcommand{\End}{\operatorname{End}}
\newcommand{\Hom}{\operatorname{Hom}}
\newcommand{\Der}{\operatorname{Der}}
\newcommand{\ad}{\operatorname{ad}}
\newcommand{\Ad}{\operatorname{Ad}}
\newcommand{\opeD}{\mathsf{D}}
\newcommand{\Ann}{\operatorname{Ann}}
\newcommand{\opeId}{\mathrm{Id}}
\newcommand{\adjoint}{\dagger}
\newcommand{\GroupFont}[1]{\mathbf{#1}}
\newcommand{\Dif}{\GroupFont{Dif}}
\newcommand{\Aut}{\GroupFont{Aut}}
\newcommand{\Gau}{\GroupFont{Gau}}
\newcommand{\groO}{\GroupFont{O}}
\newcommand{\groSO}{\GroupFont{SO}}
\newcommand{\LieAlgebraFont}[1]{\mathfrak{#1}}

\newcommand{\aut}{\LieAlgebraFont{aut}}
\newcommand{\gau}{\LieAlgebraFont{gau}}

\makeatletter
\newsavebox{\@brx}
\newcommand{\llangle}[1][]{\savebox{\@brx}{\(\m@th{#1\langle}\)}%
	\mathopen{\copy\@brx\mkern2mu\kern-0.9\wd\@brx\usebox{\@brx}}}
\newcommand{\rrangle}[1][]{\savebox{\@brx}{\(\m@th{#1\rangle}\)}%
	\mathclose{\copy\@brx\mkern2mu\kern-0.9\wd\@brx\usebox{\@brx}}}
\makeatother

\begin{document}

\title{\bfseries Dissections and automorphisms\\ of regular Courant algebroids}
\author{Benjamin \textsc{Couéraud}\\ LAREMA -- UMR CNRS 6093 -- FR CNRS 2962\\ Université d'Angers -- UBL\\ 2 boulevard Lavoisier, 49045 Angers Cedex 01, France\\ \href{mailto:coueraud@math.univ-angers.fr}{\texttt{coueraud@math.univ-angers.fr}}}
\date{}
\maketitle

\begin{abstract}
	In this note, given a regular Courant algebroid, we compute its group of automorphisms relative to a dissection. We also propose an infinitesimal version and recover examples of the literature.
\end{abstract}

\tableofcontents

\addcontentsline{toc}{section}{Introduction}
\section*{Introduction}

Since the work of Courant on the integrability of Dirac structures \cite{MR998124}, and the work of Liu, Weinstein and Xu on Lie bialgebroids \cite{MR1472888}, Courant algebroids have become important objects in Differential Geometry and Theoretical Physics (see \cite{MR3033556} for an historical survey). Courant algebroids are objects in differential geometry that generalize both Lie algebroids and quadratic Lie algebras in some sense. They are a particular case of left Leibniz algebroids, in which skew-symmetry of the bracket is controlled by an inner product. Courant algebroids that are exact are of particular interest, and play a fundamental role in generalized complex geometry. In this new geometry, in addition to diffeomorphisms of the underlying manifold, new symmetries appear, the so called $B$-fields which are 2-forms on the manifold. Such a result is obtained through the use of a particular splitting of the short exact sequence defining exact Courant algebroids. In an important article of Chen, Stiénon and Xu (see \cite{MR3022918}), a generalization of such a splitting is described for regular Courant algebroids, and is called a dissection. We will review dissections in section \ref{sec:dissection}. Using a dissection of a regular Courant algebroid, we can describe explicitly both global and infinitesimal automorphisms of such an algebroid. This will be done in sections \ref{sec:global-automorphism} and \ref{sec:infinitesimal-automorphism}. New symmetries appear, the so-called $A$-fields, which are 1-forms taking values in some quadratic Lie algebra bundle naturally related to the algebroid (see theorems \ref{thm:automorphism-group} and \ref{thm:infinitesimal-automorphism-algebra}); these new symmetries are of interest in both Mathematics and Theoretical Physics. This result encompasses previous ones from the literature on $D_n$-geometry, $B_n$-geometry and heterotic Courant algebroids (see for instance \cite{MR2534281}, \cite{MR2479266}, \cite{MR3090107} and \cite{GFRYT}).

\section{Algebroids}

All manifolds that we consider are \emph{smooth} in the sense of \cite[chapter 1]{MR2954043}, and all vector bundles are \emph{real} and of \emph{constant rank} (see \cite[definition 1.1, chapter 3]{MR1249482}). The usual interior product, De Rham differential, and Lie derivative, operating on differential forms of a manifold, will be typed in bold type ($\opeInsDeRham_X$, $\difDeRham$ and $\opeLieDeRham_X$, where $X\in\mathfrak{X}(M)$ is a vector field on the manifold $M$), as well as De Rham cohomology groups. Sections of a Whitney sum of vector bundles $E\oplus F\to M$ will be denoted by $u\oplus v$ instead of $(u,v)$ or $u+v$, for any $u\in\Gamma(E)$ and $v\in\Gamma(F)$.

\subsection{Lie algebroids}

Lie algebroids are objects in differential geometry that generalize both Lie algebras and tangent bundles to manifolds (see the definition \ref{def:Lie-algebroid}), and that appear in many situations, ranging from classical geometric classification problems (see \cite{MR2503824} for instance) to Poisson geometry (see \cite{MR2455155}) and foliations (see \cite{MR2012261}), but not exclusively. In this subsection, we give the material necessary to study automorphisms of regular Courant algebroids in the second section of this note. We work exclusively with the usual differential geometry framework; for the definition of Lie algebroid in terms of \emph{graded differential geometry}, see \cite{Vaintrob}.

\begin{definition}
	Let $M$ be a manifold. An \emph{anchor} on a vector bundle $A\to M$ is vector bundle morphism $\anchor:A\to TM$ covering the identity, where $TM\to M$ denotes the tangent bundle of $M$.
\end{definition}

\begin{definition}[{\cite[definition 3.3.1]{MR2157566}}]\label{def:Lie-algebroid}
	Let $M$ be a manifold. A \emph{Lie algebroid} $\mathcal{A}$ is a triple $(A\to M,\anchor,[\cdot,\cdot])$, where $A\to M$ is a vector bundle on $M$, $\anchor$ is an anchor on $A\to M$, and $[\cdot,\cdot]$ is a $\setR$-bilinear operation on the $\smooth(M)$-module $\Gamma(A)$ of sections of $A\to M$ called the \emph{bracket}, such that $(\Gamma(A),[\cdot,\cdot])$ is a Lie algebra and a \emph{right Leibniz rule} is satisfied:
	\begin{equation}
		[u,fv]=f[u,v]+(\anchor(u)\cdot f)v,
	\end{equation}
	for all $u,v\in\Gamma(A)$ and $f\in\smooth(M)$.
\end{definition}

\begin{remark}
	From the definition above, a \emph{left} Leibniz rule can be obtained using the skew-symmetry of the bracket $[\cdot,\cdot]$, that is,
	\begin{equation}
		[fu,v]=f[u,v]-(\anchor(v)\cdot f)u,
	\end{equation}
	for all $u,v\in\Gamma(A)$ and $f\in\smooth(M)$.
\end{remark}

\begin{proposition}[{\cite[section 6.1]{Kosmann1990}}]
	Let $\mathcal{A}=(A\to M,\anchor,[\cdot,\cdot])$ be a Lie algebroid. The anchor $\anchor$ induces a morphism of Lie algebras $\Gamma(A)\to\mathfrak{X}(M)$, still denoted by $\anchor$, where $\mathfrak{X}(M)$ is equipped with the Lie bracket of vector fields on the manifold $M$.
\end{proposition}

\begin{definition}\label{def:regular-Lie-algebroid}
	Let $\mathcal{A}=(A\to M,\anchor,[\cdot,\cdot])$ be a Lie algebroid. $\mathcal{A}$ is \emph{regular} if and only if $\anchor:A\to M$ is a vector bundle map of constant rank. Note that in this case, the kernel and the image of $\anchor:A\to M$ are vector bundles of constant rank (see \cite[theorem 8.3, chapter 3]{MR1249482}).
\end{definition}

We now give some examples of Lie algebroids.

\begin{example}\label{ex:Lie-algebra}
	Lie algebras are in correspondence with Lie algebroids over a point.
\end{example}

\begin{example}
	Let $M$ be a manifold and $A\to M$ be a Lie algebra bundle. Then $A\to M$ can be equipped with a Lie algebroid structure as follows. Let $(\mathfrak{g},[\cdot,\cdot]_\mathfrak{g})$ be a typical fiber of $A\to M$, that is, a Lie algebra. Then, the Lie algebra bracket $[\cdot,\cdot]_\mathfrak{g}$ induces a bracket on the space of sections $\Gamma(A)$, defined by $[u,v]_x=[u_x,v_x]_\mathfrak{g}$ for all $u$, $v\in\Gamma(A)$ and $x\in M$. This bracket is $\smooth(M)$-bilinear since
	\begin{equation*}
		[fu,v]_x=\big[f(x)u_x,v_x\big]_\mathfrak{g}=f(x)[u_x,v_x]_\mathfrak{g}=\big(f[u,v]\big)_x,
	\end{equation*}
	for all $u$, $v\in\Gamma(A)$, $f\in\smooth(M)$ and $x\in M$. Therefore, the null anchor and the bracket we just defined turn $A\to M$ into a Lie algebroid. Lie algebra bundles are studied in detail in \cite[section 3.3]{MR2157566} and \cite{Gundogan}.
\end{example}

\begin{example}\label{ex:canonical}
	Let $M$ be a manifold. On the tangent bundle $TM\to M$ of $M$, there is a natural Lie algebroid structure given as follows. The anchor is taken to be the identity vector bundle morphism $TM\to TM$, and the bracket on $\Gamma(TM)=\mathfrak{X}(M)$ is the Lie bracket of vector fields on $M$ (see \cite[chapter 8]{MR2954043}). This Lie algebroid will be called the \emph{canonical} Lie algebroid on $M$, and will be denoted by $\mathcal{T}_M$.
\end{example}

\begin{example}\label{ex:foliation}
	Let $M$ be a manifold and $A\to M$ an involutive distribution of $TM\to M$ (see \cite[chapter 19]{MR2954043}). Note that according to the global Frobenius theorem \cite[theorem 19.21]{MR2954043}, this distribution comes from a foliation $\mathcal{F}$ of $M$. The vector bundle $A\to M$ is equipped with a Lie algebroid structure as follows. The anchor is the inclusion map $A\hookrightarrow TM$ and the bracket is the Lie bracket of vector fields on $M$ restricted to $\Gamma(F)$, which is well-defined thanks to the involutivity of the distribution. We will denote this Lie algebroid by $\mathcal{F}$ again.
\end{example}

\begin{example}\label{ex:Poisson}
	Let $M$ be a manifold and $\pi\in\Gamma(\Lambda^2 TM)$ a bivector field such that $[\pi,\pi]_\textsf{SN}=0$ for the Schouten-Nijenhuis bracket \cite[section 3.3]{MR2906391}. Then $(M,\pi)$ is a Poisson manifold \cite[section 1.3.2]{MR2906391} and the cotangent bundle $T^*M\to M$ of M can be given a Lie algebroid structure as follows (see \cite[section 6.2]{MR2455155}). The anchor $\anchor$ is defined by
	\begin{equation*}
		\anchor(\alpha)\cdot f=\pi(\alpha,\difDeRham f),
	\end{equation*}
	for all $\alpha\in\Omega^1(M)=\Gamma(T^*M)$ and $f\in\smooth(M)$. Writing $\pi^\sharp:\Omega^1(M)\to\mathfrak{X}(M)$ for the map defined by $\pi^\sharp(\alpha)(\beta)=\pi(\alpha,\beta)$, we get $\anchor=\pi^\sharp$. The bracket is defined by
	\begin{equation}
		[\alpha,\beta]=\opeLieDeRham_{{\pi^\sharp}(\alpha)}\beta-\opeLieDeRham_{{\pi^\sharp}(\beta)}\alpha-\difDeRham\big(\pi(\alpha,\beta)\big),\label{eq:Poisson-algebroid-bracket}
	\end{equation}
	for all $\alpha$, $\beta\in\Omega^1(M)$. The Jacobi identity for this bracket holds thanks to the identity $[\pi,\pi]_\textsf{SN}=0$. We will denote by $\mathcal{P}_M[\pi]$ this Lie algebroid.
\end{example}

\begin{example}\label{ex:action}
	Let $M$ be a manifold and let $(\mathfrak{g},[\cdot,\cdot]_\mathfrak{g})$ be a Lie algebra acting (infinitesimally) on $M$ through a Lie algebra homomorphism $\rho:\mathfrak{g}\to\mathfrak{X}(M)$. The trivial vector bundle $M\times\mathfrak{g}$ can be equipped with a Lie algebroid structure (called an \emph{action Lie algebroid}, see \cite[example 3.3.7]{MR2157566}). The anchor is defined by $\anchor((x,\xi))=\rho(\xi)_x$, for any $x\in M$ and $\xi\in\mathfrak{g}$. The bracket is defined on sections $u$, $v\in\Gamma(M\times\mathfrak{g})\cong\smooth(M,\mathfrak{g})$ by
	\begin{equation*}
		[u,v]_x=[u_x,v_x]_{\mathfrak{g}}+\big[\rho(u_x)\cdot v\big]_x-\big[\rho(v_x)\cdot u\big]_x,
	\end{equation*}
	where the dot $\cdot$ in the right hand side denotes the component-wise action of a vector field.
\end{example}

\begin{example}\label{ex:Atiyah}
	Let $G$ be a Lie group and $\pi:P\to M$ be a principal $G$-bundle. Let $\mathfrak{g}$ denote the Lie algebra of $G$. According to \cite[proposition 3.2.3]{MR2157566}, there is an exact sequence of vector bundles over $M$
	\begin{equation*}
	0\longrightarrow P\times_{\Ad}\mathfrak{g}\longrightarrow TP/G\xrightarrow[]{\overline{\difDeRham\pi}}TM\longrightarrow 0,
	\end{equation*}
	where $P\times_{\Ad}\mathfrak{g}$ is the adjoint bundle of $P\to M$ (see \cite[proposition 5.1.6]{Neeb}), which is a vector bundle associated to $P\to M$ by means of the adjoint action of $G$ on $\mathfrak{g}$. The sections of the quotient vector bundle can be identified with vector fields on $P$ that are $G$-invariant \cite[proposition 3.1.4]{MR2157566}. With the anchor $\overline{\difDeRham\pi}$ and the Lie bracket of vector fields on $P$, the vector bundle $TP/G\to M$ is a Lie algebroid known as the \emph{Atiyah Lie algebroid} of $\pi:P\to M$ (see \cite{MR0086359}, \cite[section 3.2]{MR2157566}).
\end{example}

A whole class of examples has been omitted, we briefly describe it hereafter. Similarly to Lie algebras that can be obtained by differentiation from a Lie group, starting with a Lie groupoid \cite[section 1.1]{MR2157566}, one can obtain a Lie algebroid by differentiation \cite[section 3.5]{MR2157566}. Nevertheless, there are obstructions to integrate a given Lie algebroid into a Lie groupoid \cite{Crainic2003}. The next to last example is actually the Lie algebroid of a Lie groupoid, the so-called \emph{action Lie groupoid} \cite[example 1.1.9]{MR2157566}, which gives the corresponding action Lie algebroid in the case the infinitesimal action comes from a global one (see \cite[example 3.5.14]{MR2157566}). The last example is also a Lie algebroid of a Lie groupoid, it comes by differentiation from the \emph{Ehresmann gauge groupoid} associated to the principal $G$-bundle (see \cite{MR1129261} and the references within).

We will need an enriched version of a Lie algebroid which is called a \emph{quadratic Lie algebroid}, and that we present below.

\begin{definition}\label{def:quadratic-Lie-algebra}
	Let $(\mathfrak{g},[\cdot,\cdot])$ be a Lie algebra equipped with a non-degenerate symmetric bilinear form $\langle\cdot,\cdot\rangle$, invariant for the adjoint action, namely satisfying the property
	\begin{equation}
		\big\langle[x,y],z\big\rangle+\big\langle y,[x,z]\big\rangle=0,\label{def:quadratic-condition}
	\end{equation}
	for all $x$, $y$ and $z\in\mathfrak{g}$.
\end{definition}

\begin{example}\label{ex:semisimple}
	According to Cartan's criterion (see \cite[theorem 2.1, part A]{MR2179691}), any semisimple Lie algebra $\mathfrak{g}$ equipped with its Killing form, defined by $\langle x,y\rangle=\operatorname{Tr}(\ad_x\circ\ad_y)$ for all $x$ and $y\in\mathfrak{g}$, is a quadratic Lie algebra.
\end{example}

\begin{definition}
	Let $\mathcal{A}=(A\to M,\anchor,[\cdot,\cdot])$ be a regular Lie algebroid (see definition \ref{def:regular-Lie-algebroid}). $\mathcal{A}$ is a \emph{quadratic Lie algebroid}, if and only if the vector subbundle $\Ker\anchor\to M$ of $A\to M$ is equipped with a non-degenerate symmetric bilinear form $\langle\cdot,\cdot\rangle$ such that
	\begin{equation*}
		\anchor(u)\cdot\langle a,b\rangle=\big\langle[u,a],b\big\rangle+\big\langle a,[u,b]\big\rangle,
	\end{equation*}
	for all $u\in\Gamma(A)$ and $a$, $b\in\Gamma(\Ker\anchor)$. In such a case, $\Ker\anchor\to M$ is a quadratic Lie algebra bundle.
\end{definition}

\begin{example}
	A bundle of quadratic Lie algebras is a quadratic Lie algebroid, with null anchor.
\end{example}

\begin{example}
	Let $G$ be a Lie group and $\pi:P\to M$ be a principal $G$-bundle, let $\mathfrak{g}$ denote the Lie algebra of $G$, and consider the associated Atiyah Lie algebroid (example \ref{ex:Atiyah}). Suppose further that $\mathfrak{g}$ is quadratic (definition \ref{def:quadratic-Lie-algebra}), and denote by $\langle\cdot,\cdot\rangle_\mathfrak{g}$ the associated bilinear form. In that case the Atiyah Lie algebroid of $\pi:P\to M$ is quadratic. Indeed, recalling that (see, for instance, \cite[proposition 1.6.3]{Neeb}) $\Gamma(P\times_{\Ad}\mathfrak{g})\cong\smooth(P,\mathfrak{g})^G$, we obtain that $\Gamma(P\times_{\Ad}\mathfrak{g})$ is equipped with a non-degenerate symmetric bilinear form $\langle\cdot,\cdot\rangle$ defined by $\langle u,v\rangle_x=\langle u_x,v_x\rangle_\mathfrak{g}$ for all $u$, $v\in\smooth(P,\mathfrak{g})^G$ and $x\in P$. For any $X\in\Gamma(TP/G)$ and $Y$, $Z\in\Gamma(P\times_{\Ad}\mathfrak{g})$ we have
	\begin{equation*}
		X\cdot\langle Y,Z\rangle=\langle X\cdot Y,Z\rangle+\langle Y,X\cdot Z\rangle=\big\langle[X,Y],Z\big\rangle+\big\langle Y,[X,Z]\big\rangle,
	\end{equation*}
	where in the last step we used \cite[lemma 5.1.7]{Neeb}. Thus under the assumption that $\mathfrak{g}$ is quadratic, the Atiyah Lie algebroid of $\pi:P\to M$ is quadratic (see \cite[section 3.1]{BarHek}).
\end{example}

We finish this section with a brief description of the Gerstenhaber algebra structure (\cite[section 1]{MR1675117}) on $\Gamma(\Lambda^\bullet A)$, where $A\to M$ is a vector bundle equipped with a Lie algebroid structure, and that in a certain sense, extends the bracket of the Lie algebroid to the exterior algebra bundle. This notion will be exclusively used to describe \emph{doubles of Lie bialgebroids} as examples of Courant algebroids in section \ref{sec:Courant-algebroid}.

\begin{definition}[{\cite[section 5.4]{MR2455155}, \cite[section 2.6]{MR2605322}}]\label{def:Schouten-Nijenhuis}
	Let $\mathcal{A}=(A\to M,\anchor,[\cdot,\cdot])$ be a Lie algebroid. The \emph{Schouten-Nijenhuis} bracket $[\cdot,\cdot]:\Gamma(\Lambda^p A)\times\Gamma(\Lambda^q A)\rightarrow\Gamma(\Lambda^{p+q-1}A)$ on $\mathcal{A}$ is defined by
	\begin{multline*}
		[u_1\wedge\dots\wedge u_p,v_1\wedge\dots\wedge v_q]=\\
		\sum_{\substack{1\leq i\leq p\\ 1\leq j\leq q}}(-1)^{i+j}[u_i,v_j]\wedge u_1\wedge\dots\wedge\widehat{u_i}\wedge\dots\wedge u_p\wedge v_1\wedge\dots\wedge\widehat{v_j}\wedge\dots\wedge v_q,
	\end{multline*}
	for any $u_1,\dots,u_p,v_1,\dots,v_q\in\Gamma(A)$. On functions $f\in\smooth(M)$, the Schouten-Nijenhuis bracket is defined by $[f,u_1\wedge\dots\wedge u_p]=-(u_1\wedge\dots\wedge u_p)(\dif f,\cdot)$, for all $u_1,\dots,u_p\in\Gamma(A)$.
\end{definition}

\subsection{Differential forms on Lie algebroids}


\begin{definition}\label{def:differential-forms}
	Let $\mathcal{A}=(A\to M,\anchor,[\cdot,\cdot])$ be a Lie algebroid. We will denote by $\Omega^\bullet(\mathcal{A})$ the $\setZ$-graded commutative $\setR$-algebra $(\Gamma(\Lambda^\bullet A^*),\wedge)$. Elements of $\Omega^\bullet(\mathcal{A})$ will be called \emph{differential forms on $\mathcal{A}$}.
\end{definition}

Before introducing the \emph{exterior derivative} of a Lie algebroid, we introduce an \emph{interior product} and a \emph{Lie derivative}. These two operators will be used later in this note. Together with the exterior derivative, they define a what we will informally call a \emph{Cartan triple} for the Lie algebroid (see theorem \ref{thm:Cartan-triple}).

\begin{definition}
	Let $\mathcal{A}=(A\to M,\anchor,[\cdot,\cdot])$ be a Lie algebroid. For $u\in\Gamma(A)$, the \emph{interior product} of $\mathcal{A}$ with respect to $u$ is the operator $\opeIns_u:\Omega^p(\mathcal{A})\to\Omega^{p-1}(\mathcal{A})$ defined by
	\begin{equation*}
		\opeIns_u\alpha(u_2,\dots,u_p)=\alpha(u,u_2,\dots,u_p),
	\end{equation*}
	for all $u$, $u_2,\dots,u_p\in\Gamma(A)$ and $\alpha\in\Omega^p(\mathcal{A})$. On functions $f\in\smooth(M)$ it is defined by $\opeIns_u f=0$, for all $u\in\Gamma(A)$.
\end{definition}

\begin{definition}[{\cite[section 5.1]{MR2455155}}]\label{def:Lie-derivative}
	Let $\mathcal{A}=(A\to M,\anchor,[\cdot,\cdot])$ be a Lie algebroid. For $u\in\Gamma(A)$, the \emph{Lie derivative} of $\mathcal{A}$ with respect to $u$ is the operator $\opeLie_u:\Omega^p(\mathcal{A})\to\Omega^p(\mathcal{A})$ defined by
	\begin{equation*}
		\opeLie_u\alpha(v_1,\dots,v_p)=\anchor(u)\cdot\alpha(v_1,\dots,v_p)-\sum_{i=1}^p\alpha\big(v_1,\dots,[u,v_i],\dots,v_p\big),
	\end{equation*}
	for any $u$, $v_1,\dots,v_p\in\Gamma(A)$ and $\alpha\in\Omega^p(\mathcal{A})$. On functions $f\in\smooth(M)$, it is defined by $\opeLie_u f=\anchor(u)\cdot f$, for all $u\in\Gamma(A)$.
\end{definition}

\begin{definition}[{\cite[section 2]{MR1726784}, \cite[section 5.2]{MR2455155}}]\label{def:exterior-derivative}
	Let $\mathcal{A}=(A\to M,\anchor,[\cdot,\cdot])$ be a Lie algebroid. The \emph{exterior derivative} of $\mathcal{A}$ is the operator $\dif:\Omega^p(\mathcal{A})\to\Omega^{p+1}(\mathcal{A})$ defined by
	\begin{multline}
		\dif\alpha(u_0,\dots,u_p)=\sum_{i=0}^p(-1)^i\anchor(u_i)\cdot\alpha(u_0,\dots,\widehat{u_i},\dots,u_p)\\
		+\sum_{0\leq i<j\leq p}(-1)^{i+j}\alpha\big([u_i,u_j],u_0,\dots,\widehat{u_i},\dots,\widehat{u_j},\dots,u_p\big),
	\end{multline}
	for all $u_0,\dots,u_p\in\Gamma(A)$ and any $\alpha\in\Omega^p(\mathcal{A})$; where $\anchor(u_i)\cdot\alpha(u_0,\dots,\widehat{u_i},\dots,u_p)$ denotes the action of the vector field $\anchor(u_i)$ on the function $\alpha(u_0,\dots,\widehat{u_i},\dots,u_p)$. On functions $f\in\smooth(M)$, $\dif$ is defined by $(\dif f)(u)=\anchor(u)\cdot f$, for any $u\in\Gamma(A)$.
\end{definition}

The following proposition is a particular case of \cite[proposition 5.2.3, point 2]{MR2455155}.

\begin{proposition}[{\cite[proposition 5.2.3]{MR2455155}}]\label{pr:exterior-derivative-derivation}
	Let $\mathcal{A}=(A\to M,\anchor,[\cdot,\cdot])$ be a Lie algebroid. The exterior derivative $\dif$ of $\mathcal{A}$ is a derivation of degree 1 of $\Omega^\bullet(\mathcal{A})$ that squares to zero, that is, $\dif$ is a differential and $\Omega^\bullet(\mathcal{A})$ endowed with this differential becomes a differential graded commutative algebra (see \cite[part 1, chapter 3]{MR1802847} for the definition).
\end{proposition}

\begin{theorem}[{\cite[section 5]{MR2455155}}]\label{thm:Cartan-triple}
	Let $\mathcal{A}=(A\to M,\anchor_A,[\cdot,\cdot]_A)$ be a Lie algebroid, and let $u\in\Gamma(A)$. The interior product with respect to $u$ is a derivation of degree $-1$ of $\Omega^\bullet(\mathcal{A})$ and the Lie derivative with respect to $u$ is a derivation of degree $0$ of $\Omega^\bullet(\mathcal{A})$. These derivations satisfy
	\begin{gather*}
		\big[\opeLie_u,\opeLie_v\big]=\opeLie_{[u,v]_A},\quad\big[\opeLie_u,\opeIns_v\big]=\opeIns_{[u,v]_A},\quad\big[\opeIns_u,\dif\big]=\opeLie_u,\\
		\big[\opeLie_u,\dif\big]=0,\quad\big[\dif,\dif\big]=0,\quad\big[\opeIns_u,\opeIns_v\big]=0,
	\end{gather*}
	where on the left-hand sides $[\cdot,\cdot]$ denotes the graded commutator on $\operatorname{Der}\Omega^\bullet(\mathcal{A})$. In view of these identities, the triple of operators $(\opeIns_u,\opeLie_u,\dif)$ is called a \emph{Cartan triple} for $\mathcal{A}$.
\end{theorem}

We now review the notion of representation of a Lie algebroid, which allows to consider \emph{coefficients} in Lie algebroid cohomology.

\begin{definition}
	Let $\mathcal{A}=(A\to M,\anchor,[\cdot,\cdot])$ be a Lie algebroid and $V\to M$ be a vector bundle. We will denote by $\Omega^\bullet(\mathcal{A},V)$ the $\setZ$-graded $\setR$-vector space $\Gamma(\Lambda^\bullet A^*\otimes V)$. Elements of $\Omega^\bullet(\mathcal{A},V)$ will be called \emph{differential forms on $\mathcal{A}$ with values in $V\to M$}. In degree 0, we recover sections of $V\to M$.
\end{definition}

The following result is immediate.

\begin{proposition}\label{pr:graded-module-structure}
	Let $\mathcal{A}$ be a Lie algebroid on a manifold $M$, and $V\to M$ be a vector bundle. $\Omega^\bullet(\mathcal{A},V)$ is a graded $\Omega^\bullet(\mathcal{A})$-module for the multiplication defined by $\alpha\wedge(\beta\otimes s)=(\alpha\wedge\beta)\otimes s$, for any $\alpha$, $\beta\in\Omega^\bullet(\mathcal{A})$ and $s\in\Gamma(V)$.
\end{proposition}

\begin{definition}[{\cite[definition 8.4.7]{MR2795151}, \cite[section 0]{MR1929305}, \cite[section 2]{MR1726784}}]
	Let $\mathcal{A}=(A\to M,\anchor,[\cdot,\cdot])$ be a Lie algebroid, and let $V\to M$ be a vector bundle. An \emph{$\mathcal{A}$-connection} on $V\to M$ is a $\setR$-linear map $\nabla:\Gamma(V)\to\Gamma(A^*\otimes V)$, satisfying the relations
	\begin{gather*}
		\nabla_u(fs)=f\nabla_u s+(\anchor(u)\cdot f)s,\\
		\nabla_{fu}s=f\nabla_u s,
	\end{gather*}
	for any function $f\in\smooth(M)$ and sections $u\in\Gamma(A)$, $s\in\Gamma(V)$, where $\nabla_u$ denotes the map $\Gamma(V)\to\Gamma(V)$, $s\mapsto(\nabla s)(u)$, for any $u\in\Gamma(A)$.
\end{definition}

Similarly to the covariant exterior derivative of a manifold equipped with a (linear) connection (see \cite[theorem 12.57]{MR2572292}), there is a similar operator for Lie algebroids equipped with a representation, whose existence is guaranteed by the following proposition.

\begin{proposition}[{\cite[section 5.2]{MR2455155}}]
	Let $\mathcal{A}=(A\to M,\anchor,[\cdot,\cdot])$ be a Lie algebroid, and let $\nabla$ be an $\mathcal{A}$-connection on a vector bundle $V\to M$. There exists a unique operator $\dif_\nabla:\Omega^p(\mathcal{A},V)\to\Omega^{p+1}(\mathcal{A},V)$ called the \emph{covariant exterior derivative} of $\mathcal{A}$, defined by $\dif_\nabla s=\nabla s$ for all $s\in\Gamma(V)$, and
	\begin{equation*}
		\dif_\nabla(\alpha\otimes s)=\dif\alpha\otimes s+(-1)^p\alpha\wedge\nabla s,
	\end{equation*}
	for any $\alpha\in\Omega^p(\mathcal{A})$ and $s\in\Gamma(V)$; it is then extended to the whole $\Omega^\bullet(\mathcal{A},V)$ by the rules
	\begin{gather}
		\dif_\nabla(\alpha\wedge\omega)=\dif\alpha\wedge\omega+(-1)^p\alpha\wedge\dif_\nabla\omega,
		\label{eq:cov-ext-der-left}\\
		\dif_\nabla(\omega\wedge\alpha)=\dif_\nabla\omega\wedge\alpha+(-1)^q\omega\wedge\dif\alpha,
		\label{eq:cov-ext-der-right}
	\end{gather}
	for all $\alpha\in\Omega^p(\mathcal{A})$ and $\omega\in\Omega^q(\mathcal{A},V)$.
\end{proposition}

\begin{remark}[{\cite[definition 7.1.1]{MR1262213}}]
	The operator $\dif_\nabla$ introduced in the previous proposition admits an intrinsic formula given by
	\begin{multline*}
		\dif_\nabla\omega(u_0,\dots,u_p)=\sum_{i=0}^p(-1)^i\nabla_{u_i}\omega(u_0,\dots,\widehat{u_i},\dots,u_p)\\
		+\sum_{0\leq i<j\leq p}(-1)^{i+j}\omega\big([u_i,u_j],u_0,\dots,\widehat{u_i},\dots,\widehat{u_j},\dots,u_p\big),
	\end{multline*}
	for all $\omega\in\Omega^p(\mathcal{A},V)$ and all $u_0,\dots,u_p\in\Gamma(A)$.
\end{remark}

The following result comes from the definition of a Lie algebroid connection.

\begin{lemma}\label{le:curvature-trilinear}
	Let $\mathcal{A}=(A\to M,\anchor,[\cdot,\cdot])$ be a Lie algebroid, and let $\nabla$ be an $\mathcal{A}$-connection on a vector bundle $V\to M$. The map $\Gamma(A)\times\Gamma(A)\times\Gamma(V)\to\Gamma(V)$, $(u,v,s)\mapsto(\nabla_u\circ\nabla_v-\nabla_v\circ\nabla_u-\nabla_{[u,v]})(s)$ is $\smooth(M)$-trilinear. Thus there exists a $2$-differential form on $\mathcal{A}$ with values in $\End V$, $F_\nabla\in\Omega^2(\mathcal{A},\End V)$, such that $F_\nabla(u,v)=\nabla_u\circ\nabla_v-\nabla_v\circ\nabla_u-\nabla_{[u,v]}$, for all $u$, $v\in\Gamma(A)$. 
\end{lemma}

\begin{definition}
	Let $\mathcal{A}$ be a Lie algebroid on a manifold $M$, and let $\nabla$ be an $\mathcal{A}$-connection on a vector bundle $V\to M$. The $2$-differential form $F_\nabla$ introduced in the previous lemma will be called the \emph{curvature} of the $\mathcal{A}$-connection $\nabla$, and we will say that $\nabla$ is \emph{flat} if $F_\nabla=0$.
\end{definition}

The following lemma is immediate.

\begin{lemma}
	Let $\mathcal{A}$ be a Lie algebroid on a manifold $M$, and let $V\to M$ be a vector bundle. $\Omega^\bullet(\mathcal{A},\End V)$ is a $\setZ$-graded $\setR$-algebra for the multiplication defined by
	\begin{equation*}
		(\alpha\otimes\Phi)\wedge(\beta\otimes\Psi)=(\alpha\wedge\beta)\otimes(\Phi\circ\Psi),
	\end{equation*}
	for any $\alpha$, $\beta\in\Omega^\bullet(\mathcal{A})$ and $\Phi$, $\Psi\in\Gamma(\End V)$. Moreover, $\Omega^\bullet(\mathcal{A},V)$ is a left $\Omega^\bullet(\mathcal{A},\End V)$-module for the multiplication defined by
	\begin{equation*}
		(\alpha\otimes\Phi)\wedge(\beta\otimes s)=(\alpha\wedge\beta)\otimes\Phi(s),
	\end{equation*}
	for any $\alpha$, $\beta\in\Omega^\bullet(\mathcal{A})$, $\Phi\in\Gamma(\End V)$ and $s\in\Gamma(V)$.
\end{lemma}

\begin{proposition}
	Let $\mathcal{A}=(A\to M,\anchor,[\cdot,\cdot])$ be a Lie algebroid, and let $\nabla$ be an $\mathcal{A}$-connection on a vector bundle $V\to M$. We have the formula
	\begin{equation*}
		\dif_\nabla^2\alpha=F_\nabla\wedge\alpha,
	\end{equation*}
	for all $\alpha\in\Omega^\bullet(\mathcal{A},V)$. Therefore, if $\nabla$ is flat, $\dif_\nabla$ is a differential on $\Omega^\bullet(\mathcal{A},V)$.
\end{proposition}

\begin{proof}
	To begin with, we have for any $\omega\in\Omega^1(\mathcal{A},V)$ that
	\begin{equation*}
		(\dif_\nabla\omega)(u,v)=\nabla_u\omega(v)-\nabla_v\omega(u)-\omega([u,v]),
	\end{equation*}
	for all $u$, $v\in\Gamma(A)$. Then, let $s\in\Gamma(V)$, the above formula for $\omega=\nabla s\in\Omega^1(\mathcal{A},V)$ yields
	\begin{equation*}
		\big(\dif_\nabla^2 s\big)(u,v)=\nabla_u\nabla_vs-\nabla_v\nabla_us-\nabla_{[u,v]}s=F_\nabla(u,v)s=(F_\nabla\wedge s)(u,v),
	\end{equation*}
	for all $u$, $v\in\Gamma(A)$.
	Now we prove the formula for any simple tensor  $\alpha\otimes s\in\Omega^p(\mathcal{A},V)$. According to \eqref{eq:cov-ext-der-left}, we have 
	\begin{align*}
		\dif_\nabla^2(\alpha\otimes s)&=\dif_\nabla^2(\alpha\wedge s)\\
		&=\dif_\nabla\big(\dif\alpha\wedge s+(-1)^{p}\alpha\wedge\dif_\nabla s\big)\\
		&=\dif^2\alpha\wedge s+(-1)^{p+1}\dif\alpha\wedge\dif_\nabla s+(-1)^{p}\dif\alpha\wedge\dif_\nabla s+\alpha\wedge\dif_\nabla^2 s\\
		&=\alpha\wedge(F_\nabla\wedge s)\\
		&=F_\nabla\wedge(\alpha\otimes s),
	\end{align*}
	where in the last step we used the fact that $F_\nabla$ is a $2$-differential form on $\mathcal{A}$. The general formula holds for any element of $\Omega^\bullet(\mathcal{A},V)$ using $\setR$-linear combinations of simple tensors.
\end{proof}

\begin{definition}[{\cite[section 8.4]{MR2795151}, \cite[section 1]{MR1726784}}]\label{def:representation}
	Let $\mathcal{A}$ be a Lie algebroid on a manifold $M$. A \emph{representation} of $\mathcal{A}$ (or a \emph{$\mathcal{A}$-module}) is a vector bundle $V\to M$ equipped with a flat $\mathcal{A}$-connection $\nabla$. Such a representation will be denoted by $(V\to M,\nabla)$.
\end{definition}

\begin{definition}
	Let $\mathcal{A}$ be a Lie algebroid on a manifold $M$ and $(V\to M,\nabla)$ be a $\mathcal{A}$-module. We define $H^\bullet(\mathcal{A};V,\nabla)$ as the cohomology of the differential graded vector space $(\Omega^\bullet(\mathcal{A},V),\dif_\nabla)$. We will call $H^\bullet(\mathcal{A};V,\nabla)$ the cohomology of $\mathcal{A}$ with \emph{coefficients} in $(V\to M,\nabla)$.
\end{definition}

Recall that the cohomology $H^\bullet(\mathcal{A})$ of a Lie algebroid $\mathcal{A}$ is defined as the cohomology of the differential graded commutative algebra $(\Omega^\bullet(\mathcal{A}),\dif)$.

\begin{lemma}\label{pr:cohomology-graded-module-structure}
	Let $\mathcal{A}$ be a Lie algebroid on a manifold $M$ and $(V\to M,\nabla)$ be a $\mathcal{A}$-module. Then $H^\bullet(\mathcal{A};V,\nabla)$ is a graded $H^\bullet(\mathcal{A})$-module.
\end{lemma}

\begin{proof}
	The multiplication is directly induced by the one defined in the proposition \ref{pr:graded-module-structure}: $[\alpha]\wedge[\omega]=[\alpha\wedge\omega]$, for any $\alpha\in\Omega^p(\mathcal{A})$ and $\omega\in\Omega^q(\mathcal{A},V)$. According to \eqref{eq:cov-ext-der-left}, it is clear that $\alpha\wedge\omega$ is again $\dif_\nabla$-closed. Also, the multiplication is well-defined, since, using \eqref{eq:cov-ext-der-left} again, we have
	\begin{equation*}
		(\alpha+\dif\beta)\wedge(\omega+\dif_\nabla\eta)=\alpha\wedge\omega+\dif_\nabla\big(\beta\wedge\omega+(-1)^p\alpha\wedge\eta+\beta\wedge\dif_\nabla\eta\big),
	\end{equation*}
	for any $\beta\in\Omega^\bullet(\mathcal{A})$ and $\eta\in\Omega^\bullet(\mathcal{A},V)$.
\end{proof}

\begin{example}\label{ex:trivial-representation}
	Let $\mathcal{A}=(A\to M,\anchor,[\cdot,\cdot])$ be a Lie algebroid. The \emph{trivial representation} of $\mathcal{A}$ is given by the trivial vector bundle $V=M\times\setR\to M$ of rank 1, together with the connection defined by $\nabla_u\lambda=\anchor(u)\cdot\lambda$, for any $u\in\Gamma(A)$ and $\lambda\in\Gamma(V)\cong\smooth(M)$. In this case, $H^\bullet(\mathcal{A};V,\nabla)\cong H^\bullet(\mathcal{A})$.
\end{example}

\begin{example}
	Let $\mathcal{A}$ be a Lie algebroid over a point $M$, that is, a Lie algebra $\mathfrak{g}$ (see example \ref{ex:Lie-algebra}), that we will assume to be of finite dimension. Let $(V\to M,\nabla)$ be a $\mathcal{A}$-module. Since $M$ is a point, $V\to M$ is actually a vector space that we will denote again by $V$. Then, the $\mathcal{A}$-connection $\nabla$ is a linear map $V\to\mathfrak{g}^*\otimes V$. Using the linear isomorphisms
	\begin{equation*}
		\Hom(V,\mathfrak{g}^*\otimes V)\cong\Hom(V,\Hom(\mathfrak{g},V))\cong\Hom(\mathfrak{g}\otimes V,V)\cong\Hom(\mathfrak{g},\End(V)),
	\end{equation*}
	we deduce that, to consider a linear map $V\to\mathfrak{g}^*\otimes V$ is equivalent to consider a linear map $\rho:\mathfrak{g}\to V$. Moreover, this map $\rho$ is also a Lie algebra homomorphism thanks to lemma \ref{le:curvature-trilinear}. Therefore, in this case, a $\mathcal{A}$-module $(V\to M,\nabla)$ corresponds to a $\mathfrak{g}$-module $(V,\rho)$. The covariant exterior derivative is the Chevalley-Eilenberg differential on $\mathfrak{g}$ associated to the $\mathfrak{g}$-module $V$ (see \cite[section 23]{Chevalley1948}) and $H^\bullet(\mathcal{A};V,\nabla)$ is the Chevalley-Eilenberg cohomology $H_\textsf{CE}^\bullet(\mathfrak{g},V)$ of $\mathfrak{g}$ with coefficients in the $\mathfrak{g}$-module $V$.
\end{example}

\begin{example}
	Let $M$ be a manifold and let $\mathcal{T}_M$ denote the canonical Lie algebroid associated to $M$ (see example \ref{ex:canonical}). A representation of $\mathcal{T}_M$ is a vector bundle $V\to M$ together with a flat linear connection $\nabla$ on $V\to M$. The covariant exterior derivative is the differential $\difDeRham_\nabla$ associated to $\nabla$ (see \cite[theorem 12.57]{MR2572292} and \cite[chapter 7, section 4]{MR0336651}), and $H^\bullet(\mathcal{T}_M;V,\nabla)$ is the cohomology of the chain complex $(\Gamma(\Lambda^\bullet T^*M\otimes V),\difDeRham_\nabla)$.
\end{example}

\begin{example}
	Let $(M,\pi)$ be a Poisson manifold and let $\mathcal{P}_M[\pi]$ be the associated Lie algebroid (see example \ref{ex:Poisson}). A $\mathcal{P}_M[\pi]$-connection on a vector bundle $V\to M$ corresponds to the notion of a \emph{contravariant connection} on $(M,\pi)$ (see \cite[proposition 2.1.2]{MR1818181}). However, it seems that there is no mention of the associated cohomology in the literature.
\end{example}

\begin{example}
	Let $\mathcal{A}$ be the action Lie algebroid of example \ref{ex:action}. We have a representation of $\mathcal{A}$ given by the trivial vector bundle $M\times\setR\to\setR$ of rank 1, together with the connection $\nabla:\smooth(M)\to\smooth(M,\mathfrak{g}^*)\otimes\smooth(M)$ defined by $\nabla_\xi\lambda=\rho(\xi)\cdot\lambda$, for any $\xi\in\mathfrak{g}$ and $\lambda\in\smooth(M)$. Note that this connection is flat because $\rho:\mathfrak{g}\to\Gamma({TM})$ is a Lie algebra homomorphism. We have $\Omega^\bullet(\mathcal{A},M\times\setR)\cong\smooth(M,\Lambda^\bullet\mathfrak{g}^*)$, and we also have an isomorphism of chain complexes
	\begin{equation*}
		\begin{tikzcd}[row sep=large]
			\Lambda^k\mathfrak{g}^*\otimes\smooth(M)\arrow{d}[swap]{\Phi^k}\arrow{r}{\dif_\mathsf{CE}} &
			\Lambda^{k+1}\mathfrak{g}^*\otimes\smooth(M)\arrow{d}{\Phi^{k+1}}\\
			\smooth(M,\Lambda^k\mathfrak{g}^*)\arrow{r}{\dif} & \smooth(M,\Lambda^{k+1}\mathfrak{g}^*)
		\end{tikzcd},
	\end{equation*}
	where $\dif_\textsf{CE}$ denotes the Chevalley-Eilenberg differential \cite[section 23]{Chevalley1948}, and where $\Phi^k$ denotes the natural map $\Lambda^k\mathfrak{g}^*\otimes\smooth(M)\to\smooth(M,\Lambda^k\mathfrak{g}^*)$, $\omega\otimes f\mapsto f\omega$. This map induces an isomorphism of graded modules between the cohomology $H^\bullet(\mathcal{A};M\times\setR,\nabla)$ on one side, and the Chevalley-Eilenberg cohomology \linebreak $H_\textsf{CE}^\bullet(\mathfrak{g},\smooth(M))$ of $\mathfrak{g}$ with coefficients in the $\mathfrak{g}$-module $(\smooth(M),\rho)$ on the other side.
\end{example}

\subsection{Left Leibniz algebroids}

Informally, Leibniz algebroids are obtained by replacing the structure of Lie algebra that equips the space of sections of a Lie algebroid by a structure of Leibniz algebra. We first recall the definition of (left) Leibniz algebras, which are also known as (left) \emph{Loday algebras} and give some examples. As for Lie algebroids, we work with the usual differential geometry framework; for definitions in terms of graded differential geometry, see \cite{MR3078681}.

\begin{definition}[{\cite[section 1]{MR1252069}}]
	A \emph{left Leibniz algebra} is a (real) vector space $\mathfrak{l}$ equipped with a bilinear map $[\cdot,\cdot]:\mathfrak{l}\times\mathfrak{l}\to\mathfrak{l}$ called the \emph{bracket} and satisfying the left Leibniz identity 
	\begin{equation}
		\big[u,[v,w]\big]=\big[[u,v],w\big]+\big[v,[u,w]\big],
	\end{equation}
	for all $u$, $v$ and $w\in\mathfrak{l}$.
\end{definition}

\begin{example}
	Lie algebras are particular examples of (left) Leibniz algebras, for which the bracket is skew-symmetric.
\end{example}

\begin{example}
	Let $\mathfrak{l}$ be a vector space of dimension $2$, and let $(u,v)$ be a basis of $\mathfrak{l}$. We define a bracket on $\mathfrak{l}$ by setting $[u,u]=v$, $[u,v]=v$, $[v,u]=0$, $[v,v]=0$ and extending linearly. Then $(\mathfrak{l},[\cdot,\cdot])$ is a left Leibniz algebra.
\end{example}

\begin{example}\label{ex:derived-bracket}
	Let $(\mathfrak{g},[\cdot,\cdot]_\mathfrak{g},\dif)$ be a differential Lie algebra, that is $\dif:\mathfrak{g}\to\mathfrak{g}$ is a derivation whose square is zero. Then define a new bracket on $\mathfrak{g}$ by $[\xi,\eta]=[\dif\xi,\eta]_\mathfrak{g}$, for all $\xi$, $\eta\in\mathfrak{g}$; $(\mathfrak{g},[\cdot,\cdot])$ is a left Leibniz algebra (see \cite[example 2.2]{MR1252069}). This example can indeed be considered as a non-graded version of a \emph{derived bracket} (see \cite[proposition 2.1]{MR1427124}).
\end{example}

\begin{definition}[{\cite[section 1]{MR1252069}}]
	Let $(\mathfrak{l},[\cdot,\cdot]_\mathfrak{l})$ and $(\mathfrak{m},\cdot,\cdot]_\mathfrak{m})$ be two left Leibniz algebras. A \emph{morphism of left Leibniz algebras} between $\mathfrak{l}$ and $\mathfrak{m}$ is a linear map $\Phi:\mathfrak{l}\to\mathfrak{m}$ such that for all $\xi$, $\eta\in\mathfrak{l}$, we have $\Phi([u,v]_\mathfrak{l})=[\Phi(\xi),\Phi(\eta)]_\mathfrak{m}$.
\end{definition}

\begin{example}
	Morphisms between Lie algebras are a particular case of morphisms between (left) Leibniz algebras.
\end{example}

\begin{definition}
	Let $M$ be a manifold. A \emph{left Leibniz algebroid} $\mathcal{A}$ is a triple $(A\to M,\anchor,[\cdot,\cdot])$, where $A\to M$ is a vector bundle on $M$, $\anchor$ is an anchor on $A\to M$, and $[\cdot,\cdot]$ is a $\setR$-bilinear operation on the $\smooth(M)$-module $\Gamma(A)$ of sections of $A\to M$ called the \emph{bracket}, such that $(\Gamma(A),[\cdot,\cdot])$ is a left Leibniz algebra and a \emph{right Leibniz rule} is satisfied:
	\begin{equation}
		[u,fv]=f[u,v]+(\anchor(u)\cdot f)v,
	\end{equation}
	for all $u,v\in\Gamma(A)$ and $f\in\smooth(M)$.
\end{definition}

\begin{proposition}[{\cite[lemma 2.5]{MR2681592}}]
	Let $\mathcal{A}=(A\to M,\anchor,[\cdot,\cdot])$ be a Leibniz algebroid. The anchor $\anchor$ induces a morphism of (left) Leibniz algebras $\Gamma(A)\to\mathfrak{X}(M)$, still denoted by $\anchor$, where $\mathfrak{X}(M)$ is equipped with the Lie bracket of vector fields on the manifold $M$.
\end{proposition}

\begin{example}
	A left Leibniz algebroid over a point is a left Leibniz algebra.
\end{example}

\begin{example}
	Let $M$ be a manifold and $k\in\setN^*$. The vector bundle $TM\oplus\Lambda^k T^*M\to M$ is a left Leibniz algebroid for the anchor given by the projection on the first factor and the bracket defined by
	\begin{equation*}
		[X\oplus\omega,Y\oplus\eta]={[X,Y]}\oplus{\opeLieDeRham_X\eta-\opeInsDeRham_Y\difDeRham\omega},
	\end{equation*}
	for all vector fields $X$ and $Y\in\mathfrak{X}(M)$ and differential forms $\omega$ and $\eta\in\Omega^k(M)$ (see \cite[section 2]{MR2775421} and \cite[section 2.2]{MR2901838}).
\end{example}

\subsection{Courant algebroids}\label{sec:Courant-algebroid}

Informally, a Courant algebroid is a left Leibniz algebroid for which the skew-symmetry of the bracket is controlled by an inner product. As for Lie and Leibniz algebroids, we work with the usual differential geometry framework; for definitions in terms of graded differential geometry, see \cite{MR1958835}.

\begin{definition}
	Let $M$ be a manifold and $E\to M$ be a vector bundle on $M$. Let $\langle\cdot,\cdot\rangle$ be a section of the second symmetric power of the dual vector bundle $E^*\to M$, namely, a collection of symmetric bilinear forms $\langle\cdot,\cdot\rangle_x:E_x\times E_x\to\setR$ indexed by points $x\in M$, such that the map $x\mapsto\langle\cdot,\cdot\rangle_x$ is smooth. In this case, the maps $\langle\cdot,\cdot\rangle_x$ assemble into a $\setR$-bilinear map $\langle\cdot,\cdot\rangle:E\times E\to\setR$ that induces a $\smooth(M)$-bilinear map $\langle\cdot,\cdot\rangle:\Gamma(E)\times\Gamma(E)\to\smooth(M)$. In the case where $\langle\cdot,\cdot\rangle:E\times E\to\setR$ is non-degenerate, which means that each form $\langle\cdot,\cdot\rangle_x$ is non-degenerate, $\langle\cdot,\cdot\rangle$ is called an \emph{inner product} on $E\to M$ (see \cite[section 4, chapter 2]{MR0336650} or \cite[definition 9.2, chapter 3]{MR1249482} for an equivalent definition).
\end{definition}

Let $\langle\cdot,\cdot\rangle$ be an inner product on a vector bundle $E\to M$. We will make use of the usual \emph{musical} isomorphisms. Indeed, we will use a \emph{flat} symbol in exponent to denote the isomorphism $\Upsilon:E\to E^*$ induced by the inner product, and a \emph{sharp} symbol in exponent for the inverse isomorphism $\Upsilon^{-1}:E^*\to E$. Let $\langle\cdot|\cdot\rangle:E^*\times E\to\smooth(M)$ be the duality bracket between $E^*\to M$ and $E\to M$ defined by $\langle\varphi|u\rangle=\varphi(u)$, for all $u\in\Gamma(E)$ and $\varphi\in\Gamma(E^*)$. According to our notations we have $\langle u,v\rangle=\langle u^\flat|v\rangle$ and $\langle\varphi|u\rangle=\langle\varphi^\sharp,u\rangle$, for all $u$, $v\in\Gamma(E)$ and $\varphi\in\Gamma(E^*)$. Note that since $\langle\cdot,\cdot\rangle$ is not required to be positive definite, given a vector subbundle of $E\to M$, there is no associated orthogonal decomposition in general.

\begin{definition}\label{def:Courant-algebroid}
	Let $M$ be manifold. A \emph{Courant algebroid} is a quadruple $\mathcal{E}=(E\to M,\anchor,[\cdot,\cdot],\langle\cdot,\cdot\rangle)$, where $E\to M$ is a vector bundle on $M$, $\anchor$ is an anchor on $E\to M$, $\langle\cdot,\cdot\rangle$ is an inner product on $E\to M$, and $[\cdot,\cdot]$ is a $\setR$-bilinear operation on the $\smooth(M)$-module $\Gamma(E)$ of sections of $E\to M$ called the \emph{bracket}, such that the following relations are satisfied:
	\begin{gather}
		\anchor(u)\cdot\langle v,w\rangle=\big\langle[u,v],w\big\rangle+\big\langle v,[u,w]\big\rangle,\label{pr:anchor-preserves-metric}\\
		\big[u,[v,w]\big]=\big[[u,v],w\big]+\big[v,[u,w]\big],\label{pr:Loday-identity}\\
		[u,v]+[v,u]=\opeD\langle u,v\rangle,\label{pr:non-skew-symmetric}
	\end{gather}
	for all $u$, $v$, and $w\in\Gamma(E)$, where $\opeD:\smooth(M)\to\Gamma(E)$ is the derivation $\Upsilon^{-1}\circ\anchor^*\circ\difDeRham$ and $\anchor^*:T^*M\to E^*$ is the dual map of $\anchor:E\to TM$ defined by
	\begin{equation*}
		\big\langle\anchor^*(\alpha)\big|u\big\rangle=\big\langle\alpha\big|\anchor(u)\big\rangle,
	\end{equation*}
	for all $\alpha\in\Omega^1(M)$ and $u\in\Gamma(E)$.
\end{definition}

\begin{proposition}\label{pr:Courant-algebroid-special-properties}
	Let $\mathcal{E}=(E\to M,\anchor,[\cdot,\cdot],\langle\cdot,\cdot\rangle)$ be a Courant algebroid. For any function $f\in\smooth(M)$ and any sections $u$, $v\in\Gamma(E)$, we have the following properties:
	\begin{gather}
		\langle\opeD f,u\rangle=\anchor(u)\cdot f,\label{pr:D-anchor}\\
		[u,fv]=f[u,v]+(\anchor(u)\cdot f)v,\label{pr:right-Leibniz-identity}\\
		[fu,v]=\langle u,v\rangle\opeD f-(\anchor(v)\cdot f)u+f[u,v],\label{pr:left-Leibniz-identity}\\
		[\opeD f,u]=0,\label{pr:bracket-D-section}\\
		[u,\opeD f]=\opeD\langle\opeD f,u\rangle=\opeD\big(\anchor(u)\cdot f\big),\label{pr:bracket-section-D}\\
		\anchor([u,v])=[\anchor(u),\anchor(v)],\label{pr:anchor-bracket}\\
		\anchor\circ\opeD =0,\label{pr:anchor-circ-D}\\
		\Gamma((\Ker\anchor)^\bot)\text{ is generated by }\Ima\opeD\text{ as a $\smooth(M)$-module},\label{pr:orthogonal-ker-anchor}\\
		(\Ker\anchor)^\bot\subset\Ker\anchor\label{pr:coisotropy-ker-anchor}\\
		\anchor\circ\Upsilon^{-1}\circ\anchor^*=0.\label{pr:anchor-circ-dual-anchor}
	\end{gather}
\end{proposition}

\begin{proof}
	Let $f\in\smooth(M)$, $u$, $v$, and $w\in\Gamma(E)$.
	\begin{enumerate}
		\item[\eqref{pr:D-anchor}] We have successively
		\begin{equation*}
			\langle\opeD f,u\rangle=\big\langle\opeD f^\flat\big|u\big\rangle=\big\langle\anchor^*(\dif f)\big|u\big\rangle=\big\langle\dif f\big|\anchor(u)\big\rangle=\dif f(\anchor(u))=\anchor(u)\cdot f.
		\end{equation*}
		\item[\eqref{pr:right-Leibniz-identity}] According to \eqref{pr:anchor-preserves-metric} we have
		\begin{equation*}
			\anchor(u)\cdot\langle fv,w\rangle=\big\langle[u,fv],w\big\rangle+f\big\langle v,[u,w]\big\rangle,
		\end{equation*}
		and since vector fields on $M$ are derivations of $\smooth(M)$, we obtain that
		\begin{align*}
			\anchor(u)\cdot\big(f\langle v,w\rangle\big)&=(\anchor(u)\cdot f)\langle v,w\rangle+f\anchor(u)\cdot\langle v,w\rangle\\
			&=\big\langle(\anchor(u)\cdot f)v+f[u,v],w\big\rangle+f\big\langle v,[u,w]\big\rangle.
		\end{align*}
		Combining both relations, we get the result.
		\item[\eqref{pr:left-Leibniz-identity}] Using \eqref{pr:right-Leibniz-identity} and \eqref{pr:non-skew-symmetric} we have
		\begin{align*}
			[fu,v]&=\opeD \langle fu,v\rangle-[v,fu]\\
			&=\opeD \big(f\langle u,v\rangle\big)-(\anchor(v)\cdot f)u-f[u,v]\\
			&=\langle u,v\rangle\opeD f+f\opeD \langle u,v\rangle-(\anchor(v)\cdot f)u-f[v,u]\\
			&=\langle u,v\rangle\opeD f+f[u,v]-(\anchor(v)\cdot f)u.
		\end{align*}
		\item[\eqref{pr:bracket-D-section}] Define $\textsf{Lod}(u,v,w)=\big[u,[v,w]\big]-\big[[u,v],w\big]-\big[v,[u,w]\big]$. Then
		\begin{equation*}
			\textsf{Lod}(u,v,w)+\textsf{Lod}(v,u,w)=-\big[[u,v]+[v,u],w\big]=-\big[\opeD \langle u,v\rangle,w\big],
		\end{equation*}
		and the left-hand side is zero according to \eqref{pr:Loday-identity}. But for any function $f$ we can write
		\begin{equation}
			f=\left\langle\varphi,\frac{f}{\langle\varphi,\varphi\rangle}\varphi\right\rangle,
		\end{equation}
		where $\varphi$ is any non-zero section of $E\to M$. Therefore, we obtain the result for any function $f\in\smooth(M)$.
		\item[\eqref{pr:bracket-section-D}] The relation \eqref{pr:non-skew-symmetric} yields $[u,\opeD f]+[\opeD f,u]=\opeD \langle\opeD f,u\rangle$; then the result holds thanks to \eqref{pr:bracket-D-section}.
		\item[\eqref{pr:anchor-bracket}] We have successively 
		\begin{align*}
			\anchor(u)\cdot(\anchor(v)\cdot f)&=\anchor(u)\cdot \langle\opeD f,v\rangle\\
			&=\big\langle[u,\opeD f],v\big\rangle+\big\langle\opeD f,[u,v]\big\rangle\\
			&=\big\langle[u,\opeD f],v\big\rangle+\anchor([u,v])\cdot f\\
			&=\big\langle \opeD (\anchor(u)\cdot f),v\big\rangle+\anchor([u,v])\cdot f\\
			&=\anchor(v)\cdot(\anchor(u)\cdot f)+\anchor([u,v])\cdot f.
		\end{align*}
		\item[\eqref{pr:anchor-circ-D}] Using \eqref{pr:left-Leibniz-identity} we have
		\begin{equation*}
			\anchor([fu,v])=f\anchor([u,v])-(\anchor(v)\cdot f)\anchor(u)+\langle u,v\rangle\anchor(\opeD f),
		\end{equation*}
		then applying \eqref{pr:anchor-bracket} on both sides we obtain
		\begin{equation*}
			[f\anchor(u),\anchor(v)]=f[\anchor(u),\anchor(v)]-(\anchor(v)\cdot f)\anchor(u)+\langle u,v\rangle\anchor(\opeD f).
		\end{equation*}
		But the left-hand side is equal to $f[\anchor(u),\anchor(v)]-(\anchor(v)\cdot f)\anchor(u)$, hence the result.
		\item[\eqref{pr:orthogonal-ker-anchor}] According to \cite[chapter 2, section 5, proposition 3]{MR0369382} we have
		\begin{equation*}
			(\Ker\anchor)^\bot\cong\Upsilon^{-1}\Ann(\Ker\anchor)\cong\Upsilon^{-1}(\Ima\anchor^*),
		\end{equation*}
		and $\Gamma(T^*M)$ is generated by $\Ima\difDeRham$ as a $\smooth(M)$-module (since any $1$-differential form can be written locally as $\smooth(M)$-linear combination  of $\difDeRham x^1,\dots,\difDeRham x^n$, where $(x^1,\dots,x^n)$ denote local coordinates on $M$).
		\item[\eqref{pr:coisotropy-ker-anchor}] According to \eqref{pr:orthogonal-ker-anchor} and \eqref{pr:D-anchor}, we have that the vector bundle $(\Ker\anchor)^\bot\to M$ is a vector subbundle of $\Ker\anchor\to M$, that is, $\Ker\anchor\to M$ is \emph{coisotropic} with respect to the inner product $\langle\cdot,\cdot\rangle$.
		\item[\eqref{pr:anchor-circ-dual-anchor}] For any $\alpha\in\Gamma(E^*)$ we have
		\begin{equation*}
			\left\langle\Upsilon^{-1}\anchor^*(\alpha),\opeD f\right\rangle=\big\langle\alpha\,\big|\,\anchor(\opeD f)\big\rangle=0,
		\end{equation*}
		so $\Upsilon^{-1}\anchor^*(\alpha)\in(\Ker\anchor)^\bot$. But according to \eqref{pr:D-anchor} and \eqref{pr:anchor-circ-D}, we know that $(\Ker\anchor)^\bot\subset\Ker\anchor$. Therefore $\anchor\circ\Upsilon^{-1}\circ\anchor^*=0$.
	\end{enumerate}
\end{proof}

The proof of \eqref{pr:anchor-bracket} has been given in \cite{MR1916664} for the first time. One can find parts of the previous proposition in \cite[proposition 1.2]{MR2178250}, see also \cite{MR3033556} for a historical survey of the notion of Courant algebroid and the proof of the properties above.

\begin{remark}
	A Courant algebroid is a left Leibniz algebroid. Indeed, the identity \eqref{pr:Loday-identity} means that $\Gamma(E)$ is a left Leibniz algebra, we also have the right Leibniz rule thanks to \eqref{pr:right-Leibniz-identity}.
\end{remark}

\begin{remark}
	We can also equivalently define a Courant algebroid by means of a truly \emph{skew-symmetric} bracket, see \cite[definition 3.2]{MR1768267} and \cite[definition 1.6]{MR2178250}. However, we will only use the definition \ref{def:Courant-algebroid}.
\end{remark}

\begin{remark}
	It is possible to adapt the example \ref{ex:derived-bracket} to Courant algebroids, in order to present the bracket as a \emph{derived bracket}: see \cite{Alekseev-Xu} for the definition in the context of non-graded differential geometry, and \cite{MR1958835} for the definition in the context of graded differential geometry.
\end{remark}

We now review some examples of Courant algebroids.

\begin{example}
	Quadratic Lie algebras are in correspondence with Courant algebroids over a point.
\end{example}

\begin{example}
	Let $M$ be a manifold and $E\to M$ be a quadratic Lie algebra bundle. Then $E\to M$ can be equipped with a Courant algebroid structure as follows. Let $(\mathfrak{g},[\cdot,\cdot]_\mathfrak{g},\langle\cdot,\cdot\rangle_\mathfrak{g})$ be a typical fiber of $E\to M$, that is, a quadratic Lie algebra. Then, the Lie algebra bracket $[\cdot,\cdot]_\mathfrak{g}$ induces a bracket on the space of sections $\Gamma(E)$, defined by $[u,v]_x=[u_x,v_x]_\mathfrak{g}$ for all $u$, $v\in\Gamma(E)$ and $x\in M$. This bracket is $\smooth(M)$-bilinear since
	\begin{equation*}
		[fu,v]_x=\big[f(x)u_x,v_x\big]_\mathfrak{g}=f(x)[u_x,v_x]_\mathfrak{g}=\big(f[u,v]\big)_x,
	\end{equation*}
	for all $u$, $v\in\Gamma(E)$, $f\in\smooth(M)$ and $x\in M$. The non-degenerate symmetric bilinear form $\langle\cdot,\cdot\rangle_\mathfrak{g}$ induces an inner product on the space of sections $\Gamma(E)$, defined by $\langle u,v\rangle_x=\langle u_x,v_x\rangle_\mathfrak{g}$ for all $u$, $v\in\Gamma(E)$ and $x\in M$. Since $\mathfrak{g}$ is quadratic, the relation \eqref{pr:anchor-preserves-metric} is satisfied. Therefore, the null anchor, the bracket and the inner product we just defined turn $E\to M$ into a Courant algebroid.
\end{example}

\begin{example}[{\cite[example 2.14 and section 4]{MR2507112}}]
	Let $\mathfrak{g}$ be a quadratic Lie algebra acting on a manifold $M$ through a Lie algebra morphism $\xi\in\mathfrak{g}\mapsto X_\xi\in\mathfrak{X}(M)$. Assume that the \emph{stabilizing algebras} $\big\{\xi\in\mathfrak{g}:\opeLieDeRham_{X_\xi}Y=0,\text{ for all }Y\in\mathfrak{X}(M)\big\}$ are coisotropic relatively to the inner product on $\mathfrak{g}$. Then the trivial vector bundle $M\times\mathfrak{g}\to M$ can be equipped with a Courant algebroid structure.
\end{example}

\begin{definition}[{\cite{LettresSevera}}]\label{def:exact-Courant-algebroid}
	Let $\mathcal{E}=(E\to M,\anchor,[\cdot,\cdot],\langle\cdot,\cdot\rangle)$ be a Courant algebroid. We will say that $\mathcal{E}$ is \emph{exact} if and only if we have the short exact sequence of vector bundles over $M$
	\begin{equation}
		0\longrightarrow T^*M\overset{\anchor^*}{\longrightarrow}E^*\cong E\overset{\anchor}{\longrightarrow}TM\longrightarrow 0,\label{ex:exact-Courant-algebroid}
	\end{equation}
	where the isomorphism in the middle is the vector bundle map $\Upsilon^{-1}$ induced by the inner product.
\end{definition}

The structure of exact Courant algebroids has been studied in \cite{LettresSevera}. Below we state a theorem that describes the structure of such Courant algebroids, the main ingredient being the existence of a splitting of \eqref{ex:exact-Courant-algebroid} with nice properties, which is a particular case of a \emph{dissection} of $\mathcal{E}$ as we will see in section \ref{sec:dissection}.

\begin{theorem}[{\cite{LettresSevera}}]\label{thm:exact-Courant-algebroid}
	Let $\mathcal{E}=(E\to M,\anchor,[\cdot,\cdot],\langle\cdot,\cdot\rangle)$ be an exact Courant algebroid. There exists a splitting of \eqref{ex:exact-Courant-algebroid} giving a vector bundle isomorphism between $E\to M$ and $TM\oplus T^*M\to M$, and a $\difDeRham$-closed $H\in\Omega^3(M)$ such that the Courant algebroid structure on $E\to M$ is transported onto $TM\oplus T^*M\to M$ by the previous isomorphism into
	\begin{gather*}
		\anchor(X\oplus\alpha)=X,\\
		\langle X\oplus\alpha,Y\oplus\beta\rangle=\alpha(Y)+\beta(X),\\
		[X\oplus\alpha,Y\oplus\beta]={[X,Y]}\oplus{\opeLieDeRham_X\beta-\opeInsDeRham_Y(\difDeRham\alpha)+\opeInsDeRham_Y\opeInsDeRham_X H},
	\end{gather*}
	for all $X\in\mathfrak{X}(M)$ and $\alpha\in\Omega^1(M)$; with $\opeD=0\oplus\difDeRham$. We will call this Courant algebroid the \emph{exact Courant algebroid} on $M$ \emph{associated to} $H$ and will denote it by $\mathcal{E}_M[H]$.
\end{theorem}

In the previous theorem, the $3$-differential form $H$ depends explicitly on the splitting. A change of splitting is equivalent to change $H$ into $H-\difDeRham B$, for some $B\in\Omega^2(M)$. Therefore, the De Rham cohomology class $[H]\in \mathbf{H}^3(M)$ does not depend on the splitting, it is called the \emph{\v{S}evera class} of $\mathcal{E}$ (see \cite{LettresSevera}).

\begin{example}\label{ex:Courant-Lie}
	Let $\mathcal{A}=(A\to M,\anchor_A,[\cdot,\cdot]_A)$ be a Lie algebroid. There is a Courant algebroid structure on the vector bundle $A\oplus A^*\to M$ given by
	\begin{gather*}
		\anchor(u\oplus\alpha)=\anchor_A(u),\\
		\langle u\oplus\alpha,v\oplus\beta\rangle=\alpha(v)+\beta(u),\\
		[u\oplus\alpha,v\oplus\beta]={[u,v]_A}\oplus{\opeLie_u\beta-\opeIns_v(\dif\alpha)},
	\end{gather*}
	for all $u$, $v\in\Gamma(A)$ and $\alpha$, $\beta\in\Omega^1(\mathcal{A})$; with $\opeD=0\oplus\dif$. We can also twist this bracket with a $\dif$-closed $H\in\Omega^3(\mathcal{A})$, where $\dif$ denotes the exterior derivative on $\mathcal{A}$ (see \ref{def:exterior-derivative}). In this case the bracket reads
	\begin{equation*}
		[u\oplus\alpha,v\oplus\beta]={[u,v]_A}\oplus{\opeLie_u\beta-\opeIns_v(\dif\alpha)+\opeIns_v\opeIns_u H}.
	\end{equation*}
	In particular when $\mathcal{A}$ is the canonical Lie algebroid $\mathcal{T}_M$ over $M$ (example \ref{ex:canonical}), we recover the exact Courant algebroid on $M$ associated to a $\difDeRham$-closed $H\in\Omega^3(M)$.
\end{example}

We now arrive at \emph{doubles} of Lie bialgebroids, whose definition is first recalled.

\begin{definition}[{\cite[definition 2.4]{MR1472888}}]
	Let $\mathcal{A}=(A\to M,\anchor_A,[\cdot,\cdot]_A)$ and $\mathcal{B}=(B\to M,\anchor_B,[\cdot,\cdot]_B)$ be Lie algebroids over the same manifold $M$, and such that the vector bundles $A\to M$ and $B\to M$ are dual to each other with respect to a non-degenerate bilinear form $\llangle\cdot,\cdot\rrangle:\Gamma(A)\times\Gamma(B)\to\smooth(M)$. Therefore we obtain two vector bundle isomorphisms that on sections are $\Upsilon_A:\Gamma(A)\to\Gamma(B^*),u\mapsto\llangle u,\cdot\rrangle$ and $\Upsilon_B:\Gamma(B)\to\Gamma(A^*),v\mapsto\llangle\cdot,v\rrangle$. Denote by $\dif_A$ (respectively $\dif_B$) the exterior derivative of $\mathcal{A}$ (respectively $\mathcal{B}$) and $\delta_A=\Lambda^\bullet\Upsilon_B^{-1}\circ\dif_A\circ\Lambda^\bullet\Upsilon_B$. The triple $(\mathcal{A},\mathcal{B},\llangle\cdot,\cdot\rrangle)$ is a \emph{Lie bialgebroid} if and only if
	\begin{equation}\label{ex:bialgebroid-condition}
		\delta_A[\xi,\eta]_B=\big[\delta_A\xi,\eta\big]_B+(-1)^{p-1}\big[\xi,\delta_A\eta\big]_B,
	\end{equation}
	for any $\xi\in\Gamma(\Lambda^p B)$ and $\eta\in\Gamma(\Lambda^q B)$, namely, $\delta_A$ is a derivation of the algebra $\Omega^\bullet(\mathcal{B})$ relatively to the Schouten-Nijenhuis bracket (see the definition \ref{def:Schouten-Nijenhuis}).
\end{definition}

In the definition above, the condition \eqref{ex:bialgebroid-condition} is symmetric: we could have required the map $\delta_B=\Lambda^\bullet\Upsilon_A^{-1}\circ\dif_B\circ\Lambda^\bullet\Upsilon_A$ to be a derivation of $\Omega^\bullet(\mathcal{A})$ (see \cite[theorem 3.10]{MR1262213} and \cite[proposition 3.3]{MR1362125}).

\begin{example}[{\cite[theorem 2.5]{MR1472888}}]\label{ex:double}
	Let $(\mathcal{A},\mathcal{B},\llangle\cdot,\cdot\rrangle)$ be a Lie bialgebroid, with $\mathcal{A}=(A\to M,\anchor_A,[\cdot,\cdot]_A)$ and $\mathcal{B}=(B\to M,\anchor_B,[\cdot,\cdot]_B)$ being the underlying Lie algebroids. The vector bundle $A\oplus B\to M$ admits a Courant algebroid structure given by
	\begin{gather*}
		\anchor(u\oplus\alpha)=\anchor_A(u)+\anchor_B(\alpha),\\
		\langle u\oplus\alpha,v\oplus\beta\rangle=\llangle\alpha,v\rrangle+\llangle\beta,u\rrangle,\label{ex:bialgebroid-inner-product}\\
		[u\oplus\alpha,v\oplus\beta]={[u,v]_A+\opeLie^B_\alpha v-\opeIns^B_\beta(\dif_B u)}\oplus{[\alpha,\beta]_B+\opeLie^A_u\beta-\opeIns^A_v(\dif_A\alpha)},\label{ex:bialgebroid-bracket}
	\end{gather*}
	for all $u$, $v\in\Gamma(A)$ and $\alpha$, $\beta\in\Gamma(B)$; with $\opeD=\dif_A\oplus\dif_B$. This Courant algebroid is called the \emph{double} of the Lie bialgebroid $(\mathcal{A},\mathcal{B},\llangle\cdot,\cdot\rrangle)$.
\end{example}

We now describe some examples of Courant algebroids that arise as doubles of Lie bialgebroids.

\begin{example}
	Let $(\mathfrak{g},[\cdot,\cdot]_\mathfrak{g})$ be a Lie bialgebra (see \cite[definition 11.17]{MR2906391}). Therefore, $\mathfrak{g}^*$ is also a Lie algebra, with bracket denoted by $[\cdot,\cdot]_{\mathfrak{g}^*}$. The condition \eqref{ex:bialgebroid-condition} is satisfied (see \cite[relation 11.24]{MR2906391}). The Courant algebroid structure that we obtain on $\mathfrak{g}\oplus\mathfrak{g}^*$ is nothing but the Manin triple associated to $\mathfrak{g}$ (see \cite[proposition 11.28]{MR2906391}). That is, $\mathfrak{g}\oplus\mathfrak{g}^*$ is a quadratic Lie algebra with bracket given by
	\begin{equation*}
		[x\oplus\xi,y\oplus\eta]={[x,y]_\mathfrak{g}-\ad^*_\xi(y)+\ad^*_\eta(x)}\oplus{[\xi,\eta]_{\mathfrak{g}^*}-\ad^*_x(\eta)+\ad^*_y(\xi)},
	\end{equation*}
	and with inner product given by $\langle x\oplus\xi,y\oplus\eta\rangle=\xi(y)+\eta(x)$, for any $x$, $y\in\mathfrak{g}$ and $\xi$, $\eta\in\mathfrak{g}^*$. See also \cite[section 2.1]{Roy99}.
\end{example}

\begin{example}
	Let $\mathcal{A}=(A\to M,\anchor,[\cdot,\cdot])$ be a Lie algebroid, and consider the dual vector bundle $A^*\to M$, endowed with the trivial Lie algebroid structure, that is, with the null anchor and the null bracket. This pair of Lie algebroids gives a Lie bialgebroid and example \ref{ex:double} yields example \ref{ex:Courant-Lie}, with $H=0$.
\end{example}

\begin{example}
	Let $(M,\pi)$ be a Poisson manifold and consider the associated Lie algebroid $\mathcal{P}_M[\pi]$ (see example \ref{ex:Poisson}). The vector bundles underlying the Lie algebroid $\mathcal{P}_M[\pi]$ and the canonical Lie algebroid $\mathcal{T}_M$ (see example \ref{ex:canonical}) are in duality with respect to the bilinear form given by $\llangle\omega,X\rrangle=\omega(X)$ for all $\omega\in\Omega^1(M)$ and $X\in\mathfrak{X}(M)$. Moreover, the condition \eqref{ex:bialgebroid-condition} is satisfied since
	\begin{align*}
		\dif_\pi[\xi,\eta]_{\textsf{SN}}&=-\big[\pi,[\xi,\eta]_{\textsf{SN}}\big]_{\textsf{SN}}\\
		&=-\big[[\pi,\xi]_{\textsf{SN}},\eta\big]_{\textsf{SN}}-(-1)^{p-1}\big[\xi,[\pi,\eta]_{\textsf{SN}}\big]_{\textsf{SN}}\\
		&=[\dif_\pi\xi,\eta]_{\textsf{SN}}+(-1)^{p-1}[\xi,\dif_\pi\eta]_{\textsf{SN}},
	\end{align*}
	for any $\xi\in\Omega^p(M)$ and $\eta\in\Gamma(\Lambda^\bullet TM)$. We obtain in this way a Courant algebroid as the double of the above pair of Lie algebroids.
\end{example}

\begin{example}
	To any Poisson-Nijenhuis manifold (see \cite[definition 4.1]{Kosmann1990}) corresponds a Lie bialgebroid, which is described by \cite[proposition 3.2]{MR1421686}. We can double this Lie bialgebroid to obtain a Courant algebroid, according to the example \ref{ex:double}.
\end{example}

\begin{remark}
	There exists an obstruction for a given Courant algebroid for being the double of some Lie bialgebroid. This obstruction is known as the \emph{modular class} of the Courant algebroid, which appeared in \cite{MR2393640}.
\end{remark}

We finish this section by defining morphisms between Courant algebroids over the \emph{same} base manifold.

\begin{definition}\label{def:Courant-algebroid-morphism}
	Let $\mathcal{E}=(E\to M,\anchor_E,[\cdot,\cdot]_E,\langle\cdot,\cdot\rangle_E)$ and $\mathcal{F}=(F\to M,\anchor_F,\break[\cdot,\cdot]_F)$ be two Courant algebroids, over the \emph{same} base manifold. A \emph{morphism} between $\mathcal{E}$ and $\mathcal{F}$ is a vector bundle map $\Phi$ between $E\to M$ and $F\to M$ such that for all $u$, $v\in\Gamma(E)$
	\begin{enumerate}
		\item $\anchor_E=\anchor_F\circ\Phi$,
		\item $\langle\Phi(u),\Phi(v)\rangle_F=\langle u,v\rangle_E$,
		\item $[\Phi(u),\Phi(v)]_F=\Phi([u,v]_E)$.
	\end{enumerate}
	In the case where $\Phi$ is an isomorphism of vector bundles (respectively automorphism), we say that $\Phi:\mathcal{E}\to\mathcal{F}$ is an \emph{isomorphism of Courant algebroids} (respectively an \emph{automorphism of Courant algebroids}), over the \emph{same} base manifold $M$.
\end{definition}

It is clear that the composition of two Courant algebroid morphisms over a common base manifold yields a Courant algebroid morphism over the same manifold again. Thus, given a \emph{fixed} manifold $M$, we have a category, whose objects are Courant algebroids over the manifold $M$, and whose morphisms are morphisms of Courant algebroids over the manifold $M$, as in the previous definition.

\section{Automorphisms of regular Courant algebroids}

\begin{definition}
	Let $\mathcal{E}=(E\to M,\anchor,[\cdot,\cdot],\langle\cdot,\cdot\rangle)$ be a Courant algebroid. $\mathcal{E}$ is \emph{regular} if the anchor $\anchor:E\to TM$ is a vector bundle morphism of constant rank (see \cite[definition 8.1, chapter 3]{MR1249482}).
\end{definition}

The advantage of working with regular Courant algebroids is that we can consider the kernel and the image of the anchor as vector bundles of constant rank (see \cite[theorem 8.2, chapter 3]{MR1249482}).

In \cite{MR3022918}, Chen, Stiénon and Xu have introduced the appropriate generalization for regular Courant algebroids of the splitting that made possible the study of exact Courant algebroids (see theorem \ref{thm:exact-Courant-algebroid}), called a \emph{dissection}. After recalling the notion of dissection, we use it to find an explicit description of both the global and the infinitesimal automorphisms of a regular Courant algebroid.

\subsection{Dissections}\label{sec:dissection}

\begin{definition}
	Let $\mathcal{E}=(E\to M,\anchor,[\cdot,\cdot],\langle\cdot,\cdot\rangle)$ be a regular Courant algebroid. If the context is clear, we set $F=\Ima\anchor$ and $Q=\Ker\anchor/(\Ker\anchor)^\bot$ without further reference to $\mathcal{E}$.
\end{definition}

Given a regular Courant algebroid, the fiber bundles $F\to M$ and $Q\to M$ defined above are more than just vector bundles, we detail their structures below.

\begin{proposition}\label{pr:Courant-algebroid-foliation}
	Let $\mathcal{E}=(E\to M,\anchor,[\cdot,\cdot],\langle\cdot,\cdot\rangle)$ be a regular Courant algebroid. There is a canonical foliation associated to the vector bundle $F\to M$. We will denote by $\mathcal{F}$ both the foliation and the Lie algebroid structure on $F\to M$ (see example \ref{ex:foliation}) without further reference to $\mathcal{E}$ if the context is clear.
\end{proposition}

\begin{proof}
	From \eqref{pr:anchor-bracket} the vector bundle $F\to M$ is an involutive distribution of $TM\to M$. Therefore, according to the global Frobenius theorem (see \cite[theorem 19.21]{MR2954043}), we obtain a foliation $\mathcal{F}$ associated to $F\to M$.
\end{proof}

The $\smooth(M)$-module $\Gamma(\Lambda^\bullet F^*)$ is the $\smooth(M)$-module of the so-called \emph{tangential} differential forms (\cite[chapter 3]{Moore1988} and \cite[section 1.1.3]{MR2319199}); therefore, it will be denoted by $\Omega^\bullet(\mathcal{F})$, following definition \ref{def:differential-forms}. In what follows, given $X\in\Gamma(F)$, we will write abusively $\opeInsDeRham_X$, $\opeLieDeRham_X$, and $\difDeRham$ for the Cartan triple associated to $\mathcal{F}$ (see theorem \ref{thm:Cartan-triple}), which is essentially a restriction to $\Omega^\bullet(\mathcal{F})$ of the Cartan triple of De Rham. In order to avoid confusions with other brackets, we will also denote by $\{\cdot,\cdot\}$ the Lie algebroid bracket of $\mathcal{F}$, which is just the restriction to $\Gamma(F)$ of the Lie bracket of vector fields of the underlying manifold.

\begin{proposition}
	Let $\mathcal{E}=(E\to M,\anchor,[\cdot,\cdot],\langle\cdot,\cdot\rangle)$ be a regular Courant algebroid. The vector bundle $Q\to M$ is a quadratic Lie algebra bundle (which is a particular case of a Courant algebroid), that we will denote by $\mathcal{Q}=(Q\to M,[\cdot,\cdot]_Q,\langle\cdot,\cdot\rangle_Q)$ without further reference to $\mathcal{E}$ if the context is clear.
\end{proposition}

\begin{proof}
	On the vector bundle $Q\to M$ we have a $\setR$-bilinear map $[\cdot,\cdot]_Q:\Gamma(Q)\times\Gamma(Q)\to\Gamma(Q)$ defined by $[\bar{a},\bar{b}]_Q=\overline{[a,b]}$ for any $\bar{a}$, $\bar{b}\in\Gamma(Q)$. This bracket is well-defined according to \eqref{pr:bracket-D-section} and \eqref{pr:bracket-section-D}. Since the anchor of $\mathcal{E}$ induces the zero map on $Q\to M$, the Leibniz rules \eqref{pr:right-Leibniz-identity} and \eqref{pr:left-Leibniz-identity} imply together that $Q\to M$ is a Lie algebra bundle. The inner product of $\mathcal{E}$ induces a well-defined inner product $\langle\cdot,\cdot\rangle_Q:\Gamma(Q)\times\Gamma(Q)\to\smooth(M)$ thanks to \eqref{pr:anchor-preserves-metric}. It is non-degenerate, indeed, let $\bar{a}\in\Gamma(Q)$ such that $\langle\bar{a},\bar{b}\rangle=0$ for all $\bar{b}\in\Gamma(Q)$, then $\langle a,b\rangle=0$ for all $b\in\Gamma(\Ker\anchor)$, which is generated by $\Ima\opeD$ as a $\smooth(M)$-module according to \eqref{pr:orthogonal-ker-anchor}, hence $\bar{a}=0$. Thanks to \eqref{pr:anchor-preserves-metric}, this inner product also satisfies the property \eqref{def:quadratic-condition}. Therefore, $Q\to M$ is a quadratic Lie algebra bundle.
\end{proof}

\begin{definition}
	Let $\mathcal{E}=(E\to M,\anchor,[\cdot,\cdot],\langle\cdot,\cdot\rangle)$ be a regular Courant algebroid. Elements of the $\smooth(M)$-module $\Omega^\bullet(\mathcal{F},Q)=\Gamma(\Lambda^\bullet F^*\otimes Q)$ will be called \emph{tangential differential forms} on $M$ \emph{with values in} $\mathcal{Q}$. On $\Omega^\bullet(\mathcal{F},Q)$ we define
	\begin{gather*}
		\langle\cdot\wedge\cdot\rangle_Q:\Omega^p(\mathcal{F},Q)\times\Omega^q(\mathcal{F},Q)\to\Omega^{p+q}(\mathcal{F}),\\
		[\cdot\wedge\cdot]_Q:\Omega^p(\mathcal{F},Q)\times\Omega^q(\mathcal{F},Q)\to\Omega^{p+q}(\mathcal{F},Q),
	\end{gather*}
	by the following formulas:
	\begin{gather*}
		\begin{split}
			\langle\omega\wedge\eta\rangle_Q&(X_1,\dots,X_{p+q})\\
			&=\frac{1}{p!q!}\sum_{\sigma\in S_{p+q}}(-1)^{|\sigma|}\big\langle\omega(X_{\sigma(1)},\dots,X_{\sigma(p)}),\eta(X_{\sigma(p+1)},\dots,X_{\sigma(p+q)})\big\rangle_Q,
		\end{split}\\
		\begin{split}
			[\omega\wedge\eta]_Q&(X_1,\dots,X_{p+q})\\
			&=\frac{1}{p!q!}\sum_{\sigma\in S_{p+q}}(-1)^{|\sigma|}\big[\omega(X_{\sigma(1)},\dots,X_{\sigma(p)}),\eta(X_{\sigma(p+1)},\dots,X_{\sigma(p+q)})\big]_Q,
		\end{split}
	\end{gather*}
	for all $X_1,\dots,X_{p+q}\in\Gamma(F)$; or, on elementary tensors, by:
	\begin{gather*}
		\left\langle(\alpha\otimes\bar{a})\wedge(\beta\otimes\bar{b})\right\rangle_Q=\big\langle\bar{a},\bar{b}\big\rangle_Q\alpha\wedge\beta,\\
		\big[(\alpha\otimes\bar{a})\wedge(\beta\otimes\bar{b})\big]_Q=(\alpha\wedge\beta)\otimes\big[\bar{a},\bar{b}\big]_Q,
	\end{gather*}
	for all $\alpha$, $\beta\in\Omega^\bullet(\mathcal{F})$ and $\bar{a}$, $\bar{b}\in\Gamma(Q)$.
\end{definition}

\begin{proposition}
	Let $\mathcal{E}=(E\to M,\anchor,[\cdot,\cdot],\langle\cdot,\cdot\rangle)$ be a regular Courant algebroid. We have
	\begin{enumerate}
		\item $\langle\omega\wedge\eta\rangle_Q=(-1)^{pq}\langle\eta\wedge\omega\rangle_Q$ for all $\omega\in\Omega^p(\mathcal{F},Q)$ and $\eta\in\Omega^q(\mathcal{F},Q)$,
		\item $[\omega\wedge\eta]_Q=-(-1)^{pq}[\eta\wedge\omega]_Q$ for all $\omega\in\Omega^p(\mathcal{F},Q)$ and $\eta\in\Omega^q(\mathcal{F},Q)$, 
		\item Equipped with the operation $[\cdot\wedge\cdot]_Q$, $\Omega^\bullet(\mathcal{F},Q)$ is a $\setZ$-graded Lie algebra (see \cite[definition 1, section 1, chapter 1]{9783540092568}).
	\end{enumerate}
\end{proposition}

\begin{proof}
	The first property is clear from the definition. The second one too, for we have
	\begin{align*}
		\big[(\alpha\otimes\bar{a})\wedge(\beta\otimes\bar{b})\big]_Q&=(\alpha\wedge\beta)\otimes\big[\bar{a},\bar{b}\big]_Q\\
		&=-(-1)^{ij}(\beta\wedge\alpha)\otimes\big[\bar{b},\bar{a}\big]_Q\\
		&=-(-1)^{ij}\big[(\beta\otimes\bar{b})\wedge(\alpha\otimes\bar{a})\big]_Q,
	\end{align*}
	for all $\alpha\otimes\bar{a}\in\Omega^i(\mathcal{F},Q)$ and $\beta\otimes\bar{b}\in\Omega^j(\mathcal{F},Q)$, so the bracket $[\cdot\wedge\cdot]_Q$ is graded skew-symmetric. It remains to show that the graded Jacobi identity holds, but this follows directly from the non-graded Jacobi identity since
	\begin{multline*}
		\begin{aligned}
			(-1)^{ik}\big[&\alpha\otimes\bar{a}\wedge\big[\beta\otimes\bar{b}\wedge\gamma\otimes\bar{c}\big]_Q\big]_Q\\
			\quad+&(-1)^{ij}\big[\beta\otimes\bar{b}\wedge\big[\gamma\otimes\bar{c}\wedge\alpha\otimes\bar{a}\big]_Q\big]_Q\\
			&\quad\quad+(-1)^{jk}\big[\gamma\otimes\bar{c}\wedge\big[\alpha\otimes\bar{a}\wedge\beta\otimes\bar{b}\big]_Q\big]_Q
		\end{aligned}\\
		\quad=(-1)^{ik}(\alpha\wedge\beta\wedge\gamma)\otimes\Big(\big[\bar{a},\big[\bar{b},\bar{c}\big]_Q\big]_Q+\big[\bar{b},[\bar{c},\bar{a}\big]_Q\big]_Q+\big[\bar{c},\big[\bar{a},\bar{b}\big]_Q\big]_Q\Big),
	\end{multline*}
	for any elementary tensors $\alpha\otimes\bar{a}\in\Omega^i(\mathcal{F},Q)$, $\beta\otimes\bar{b}\in\Omega^j(\mathcal{F},Q)$ and $\gamma\otimes\bar{c}\in\Omega^k(\mathcal{F},Q)$.
\end{proof}

The next lemma contains the results of computations that will be used extensively in the following.

\begin{lemma}\label{pr:lemma-forms}
	Let $\mathcal{E}=(E\to M,\anchor,[\cdot,\cdot],\langle\cdot,\cdot\rangle)$ be a regular Courant algebroid.
	\begin{enumerate}
		\item For any $A_1$ and $A_2\in\Omega^1(\mathcal{F},Q)$, we have
		\begin{equation*}
			\langle A_1\wedge A_2\rangle_Q(X,Y)=\big\langle A_1(X),A_2(Y)\big\rangle_Q-\big\langle A_1(Y),A_2(X)\big\rangle_Q,
		\end{equation*}
		for all $X$, $Y\in\Gamma(F)$.
		\item For any $A\in\Omega^1(\mathcal{F},Q)$ and $R\in\Omega^2(\mathcal{F},Q)$, we have
		\begin{align*}
			\big\langle A\wedge R&\big\rangle_Q(X,Y,Z)\\
			&=\big\langle A(X),R(Y,Z)\big\rangle_Q-\big\langle A(Y),R(X,Z)\big\rangle_Q+\big\langle A(Z),R(X,Y)\big\rangle_Q,
		\end{align*}
		for all $X$, $Y$ and $Z\in\Gamma(F)$.
		\item For any $A\in\Omega^1(\mathcal{F},Q)$, we have
		\begin{equation*}
			[A\wedge A]_Q(X,Y)=2\big[A(X),A(Y)\big]_Q,
		\end{equation*}
		for all $X$, $Y\in\Gamma(F)$.
		\item For any $A\in\Omega^1(\mathcal{F},Q)$, we have
		\begin{equation*}
			\big\langle A\wedge[A,A]_Q\big\rangle_Q(X,Y,Z)=6\big\langle A(X),[A(Y),A(Z)]_Q\big\rangle_Q,
		\end{equation*}
		for all $X$, $Y$ and $Z\in\Gamma(F)$.
	\end{enumerate}
\end{lemma}

\begin{definition}
    Let $\omega\in\Omega^k(\mathcal{F},Q)$. Define $\ad_\omega\in\Gamma(\Lambda^k F^*\otimes\End Q)$ by
    \begin{equation*}
	    \ad_\omega(X_1,\dots,X_k)(\bar{a})=\Big[\omega(X_1,\dots,X_k),\bar{a}\Big]_Q,
    \end{equation*}
    for all $X_1,\dots,X_k\in\Gamma(F)$ and $\bar{a}\in\Gamma(Q)$.
\end{definition}

\begin{proposition}[{\cite[lemma 1.2]{MR3022918}}]
	Let $\mathcal{E}=(E\to M,\anchor,[\cdot,\cdot],\langle\cdot,\cdot\rangle)$ be a regular Courant algebroid. Consider the diagram
	\begin{equation}\label{dia:dissection}
		\begin{tikzcd}
			{} & 0\arrow{d}{}\\
			{} & (\Ker\anchor)^\bot\cong F^*\arrow[hook]{d}\\
			0\arrow{r}{} & \Ker\anchor\arrow{d}{\pi}\arrow[hook]{r} & E\arrow{r}{\anchor} & F\arrow{r} & 0\\
			& Q\arrow{d}{}\\
			& 0
		\end{tikzcd},
	\end{equation}
	where $\pi$ is the projection $\Ker\anchor\to Q$. Then,
	\begin{enumerate}
		\item there exists a vector bundle morphism $\lambda:F\to E$ such that $\anchor\circ\lambda=\opeId_F$ and $\Ima\lambda$ is an isotropic vector subbundle of $E\to M$ relatively to $\langle\cdot,\cdot\rangle$,
		\item there exists a vector bundle morphism $\sigma:Q\to\Ker\anchor$ such that $\pi\circ\sigma=\opeId_Q$ and $\Ima\sigma$ is orthogonal to $\Ima\lambda$ in $E\to M$ relatively to $\langle\cdot,\cdot\rangle$.
	\end{enumerate}
\end{proposition}

\begin{remark}\label{rm:dissection-technical}
  In the above diagram, the inclusion on the right is given by \eqref{pr:anchor-circ-D}, the top one by \eqref{pr:coisotropy-ker-anchor} and the isomorphism is given by $\Upsilon^{-1}\circ\anchor^*$ since according to \cite[chapter 2, section 5, proposition 3]{MR0369382} we have
\begin{equation*}
	\Ima(\Upsilon^{-1}\circ\anchor^*)=\Upsilon^{-1}\Ima(\anchor^*)\cong\Upsilon^{-1}(\Ann\Ker\anchor)\cong(\Ker\anchor)^*.
\end{equation*}
\end{remark}

\begin{definition}
	Let $\mathcal{E}=(E\to M,\anchor,[\cdot,\cdot],\langle\cdot,\cdot\rangle)$ be a regular Courant algebroid. A pair of vector bundle morphisms $(\lambda,\sigma)$ resulting from the previous proposition will be called a \emph{dissection} of $\mathcal{E}$.
\end{definition}

The following theorem, which appeared in \cite{MR3022918}, is fundamental: using a dissection of a regular Courant algebroid $\mathcal{E}$, we can decompose the vector bundle underlying $\mathcal{E}$ into a Whitney sum $F^*\oplus Q\oplus F\to M$ and transport the Courant algebroid structure on $\mathcal{E}$ onto this vector bundle, in a similar way to theorem \ref{thm:exact-Courant-algebroid}.

\begin{theorem}[{\cite[section 2]{MR3022918}}]\label{thm:CSX}
	Let $\mathcal{E}=(E\to M,\anchor_E,[\cdot,\cdot]_E,\langle\cdot,\cdot\rangle_E)$ be a regular Courant algebroid. Let $(\lambda,\sigma)$ be a dissection of $\mathcal{E}$. Firstly, the dissection determines a vector bundle isomorphism
	\begin{equation}\label{def:dissection-isomorphism}
		\Delta:
		\left\{
		\begin{array}{l}
		    F^*\oplus Q\oplus F\longrightarrow E\\
		    \alpha\oplus\bar{a}\oplus X\longmapsto\anchor_E^*(\alpha)^\sharp+\sigma(\bar{a})+\lambda(X)
		\end{array}
		\right.
		.
	\end{equation}
	Secondly, using the isomorphism $\Delta$, we transport the Courant algebroid structure on $E\to M$ onto the vector bundle $F^*\oplus Q\oplus F\to M$, whose anchor $\anchor$, inner product $\langle\cdot,\cdot\rangle$, and bracket $[\cdot,\cdot]$ are given for any $\alpha$, $\beta\in\Gamma(F^*)$, $\bar{a}$, $\bar{b}\in\Gamma(Q)$ and $X$, $Y\in\Gamma(F)$ by the following relations:
	\begin{gather*}
		\anchor(\alpha\oplus\bar{a}\oplus X)=X,\\
		\big\langle\alpha\oplus\bar{a}\oplus X,\beta\oplus\bar{b}\oplus Y\big\rangle=\alpha(Y)+\beta(X)+\langle\bar{a},\bar{b}\rangle_Q,\\
		[X,Y]={\opeInsDeRham_Y\opeInsDeRham_X H}\oplus {R(X,Y)}\oplus{\{X,Y\}},\\
		[X,\bar{a}]=-[\bar{a},X]=K(X,\bar{a})\oplus\nabla_X\bar{a},\\
		[X,\alpha]=\opeLieDeRham_X\alpha,\\
		[\alpha,X]=-\opeLieDeRham_X\alpha+\difDeRham\opeInsDeRham_X\alpha=-\opeInsDeRham_X\difDeRham\alpha,\\
		[\bar{a},\bar{b}]={P(\bar{a},\bar{b})}\oplus{[\bar{a},\bar{b}]_Q},\\
		[\alpha,\bar{a}]=[\bar{a},\alpha]=[\alpha,\beta]=0,
	\end{gather*}
	where $H\in\Omega^3(\mathcal{F})$, $R\in\Omega^2(\mathcal{F},Q)$ and $\nabla$ is a $\mathcal{F}$-connection on $Q\to M$, which all explicitly depend on the dissection, and where the intermediary quantities
	$K:\Gamma(F)\otimes\Gamma(Q)\to\Gamma(F^*)$ and $P:\Gamma(Q)\otimes\Gamma(Q)\to\Gamma(F^*)$ are defined by
	\begin{equation*}
		K(X,\bar{a})(Y)=-\big\langle\bar{a},R(X,Y)\big\rangle_Q,\enspace P(\bar{a},\bar{b})(X)=\big\langle\bar{b},\nabla_X\bar{a}\big\rangle_Q.
	\end{equation*}
	Moreover, writing $\difDeRham_\nabla$ for the covariant exterior derivative associated to the $\mathcal{F}$-module $(Q\to M,\nabla)$ (see definition \ref{def:exterior-derivative}), the Courant algebroid axioms for the Courant algebroid structure just introduced on the vector bundle $F^*\oplus Q\oplus F\to M$, require $\nabla$, $R$ and $H$ to satisfy the following \emph{compatibility relations}:
	\begin{gather}
		X\cdot\big\langle\bar{a},\bar{b}\big\rangle_Q=\big\langle\nabla_X\bar{a},\bar{b}\big\rangle_Q+\big\langle\bar{a},\nabla_X\bar{b}\big\rangle_Q,\label{thm:CSX:nabla-metric}\\
		\nabla_X[\bar{a},\bar{b}]_Q=\big[\nabla_X\bar{a},\bar{b}\big]_Q+\big[\bar{a},\nabla_X\bar{b}\big]_Q,\label{thm:CSX:nabla-bracket}\\
		\difDeRham_\nabla R=0,\label{thm:CSX:Bianchi}\\
		\difDeRham_\nabla^2=\ad_R,\label{thm:CSX:curvature-nabla}\\
		\difDeRham H=\frac{1}{2}\langle R\wedge R\rangle_Q,\label{thm:CSX:Pontryagin}
	\end{gather}
	for all $\bar{a}$, $\bar{b}\in\Gamma(Q)$ and $X$, $Y$, $Z\in\Gamma(F)$.
\end{theorem}

\begin{remark}\label{rm:compact-bracket}
	The bracket defined on the vector bundle $F^*\oplus Q\oplus F\to M$ in the previous theorem can be written as
	\begin{align}
		[\alpha\oplus\bar{a}\oplus X&,\beta\oplus\bar{b}\oplus Y]_{\nabla,\,R,\,H}\notag\\
		&=\opeLieDeRham_X\beta-\opeInsDeRham_Y\difDeRham\alpha+\big\langle\nabla\bar{a},\bar{b}\big\rangle_Q-\big\langle\bar{b},{\opeInsDeRham}_X R\big\rangle_Q+\big\langle\bar{a},{\opeInsDeRham}_Y R\big\rangle_Q\notag\\
		&\quad+\opeInsDeRham_Y\opeInsDeRham_X H\oplus{[\bar{a},\bar{b}]_Q+\nabla_X\bar{b}-\nabla_Y\bar{a}+R(X,Y)\oplus{\{X,Y\}}}.\label{eq:full-bracket}
	\end{align}
\end{remark}

The Courant algebroid structure that appears in the previous theorem can be singled out, which is the content of the following proposition (see also \cite[section 2]{MR3022918}).

\begin{proposition}\label{pr:standard-Courant-algebroid}
	Let $M$ be a manifold. Let $F\to M$ be an involutive distribution of $TM\to M$, associated to a foliation $\mathcal{F}$ of $M$, and $\mathcal{Q}=(Q\to M,[\cdot,\cdot]_Q,\langle\cdot,\cdot\rangle_Q)$ a quadratic Lie algebra bundle. Let $\nabla$ be a $\mathcal{F}$-connection on $Q\to M$, $R\in\Omega^2(\mathcal{F},Q)$ and $H\in\Omega^3(\mathcal{F})$ such that the compatibility relations \eqref{thm:CSX:nabla-metric}, \eqref{thm:CSX:nabla-bracket}, \eqref{thm:CSX:Bianchi}, \eqref{thm:CSX:curvature-nabla} and \eqref{thm:CSX:Pontryagin} are satisfied. Then, the vector bundle $F^*\oplus Q\oplus F\to M$ is a Courant algebroid for the anchor, the inner product and the bracket defined in the previous theorem. This regular Courant algebroid is called \emph{standard} and is denoted by $\mathcal{S}_M[\nabla,R,H]$; the inclusions $F\hookrightarrow F^*\oplus Q\oplus F$ and $Q\hookrightarrow Q\oplus F$ constitute a dissection of $\mathcal{S}_M[\nabla,R,H]$.
\end{proposition}

\begin{remark}\label{rmq:dissection}
	Let $\mathcal{E}$ be a regular Courant algebroid. Given a dissection $(\lambda,\sigma)$ of $\mathcal{E}$, we obtain by the previous theorem an isomorphism $\Delta$, a $\mathcal{F}$-connection $\nabla$, a tangential $2$-differential form $R$ with values in $\mathcal{Q}$ and a tangential $3$-differential form $H$. In what follows we will only work with these data, not the splittings $\lambda$, $\sigma$. Therefore, we will refer to the dissection $(\lambda,\sigma)$ by the data $(\Delta,\nabla,R,H)$ associated to the dissection by means of the previous theorem (historically, only the isomorphism $\Delta$ was called a dissection, see \cite[section 1.3]{MR3022918}).
\end{remark}

Now we would like to investigate the change of standard Courant algebroid structure caused by a change of dissection in a regular Courant algebroid. For that matter, consider a regular Courant algebroid $\mathcal{E}=(E\to M,\anchor_E,[\cdot,\cdot]_E,\langle\cdot,\cdot\rangle_E)$, and let $(\Delta,\nabla,R,H)$ and $(\hat{\Delta},\hat{\nabla},\hat{R},\hat{H})$ be two dissections of $\mathcal{E}$ (see remark \ref{rmq:dissection}). We will write $\anchor$ for the anchor, $[\cdot,\cdot]$ for the bracket, and $\langle\cdot,\cdot\rangle$ for the inner product of both Courant algebroid structures on $F^*\oplus Q\oplus F\to M$ associated to these dissections. By definition of a dissection, we get two Courant algebroid isomorphisms $\Delta^{-1}:\mathcal{E}\to\mathcal{S}_M[\nabla,R,H]$ and $\hat{\Delta}^{-1}:\mathcal{E}\to\mathcal{S}_M[\hat{\nabla},\hat{R},\hat{H}]$, over the same manifold $M$. Therefore, we obtain a diagram of vector bundles over $M$
\begin{equation*}
	\begin{tikzcd}[row sep=large]
		& E\arrow[swap]{ld}{\Delta^{-1}}\arrow{rd}{\hat{\Delta}^{-1}} & \\
		F^*\oplus Q\oplus F\arrow{rr}{\delta=\hat{\Delta}^{-1}\circ\Delta} & & F^*\oplus Q\oplus F
	\end{tikzcd},
\end{equation*}
and the \emph{change of dissection} $\delta=\hat{\Delta}^{-1}\circ\Delta:\mathcal{S}_M[\nabla,R,H]\to\mathcal{S}_M[\hat{\nabla},\hat{R},\hat{H}]$ is also an isomorphism of Courant algebroids over $M$ (see also \cite[proposition 2.7]{MR3022918}).

\begin{definition}
	We will write $\groO(Q)$ for the group of orthogonal automorphisms of the vector bundle $Q\to M$ equipped with its inner product $\langle\cdot,\cdot\rangle_Q$, and $\Aut(\mathcal{Q})$ for the group of orthogonal automorphisms of the Lie algebroid $\mathcal{Q}$, that is, elements of $\groO(Q)$ that also preserve the bracket of $\mathcal{Q}$. Furthermore, writing $\tau\in\groO(Q)$ (respectively $\tau\in\Aut(\mathcal{Q})$) will mean that the bundle map $\tau$ is assumed to cover the identity of the base manifold $M$, whereas writing $(\varphi,\tau)\in\groO(\mathcal{Q})$ (respectively $(\varphi,\tau)\in\Aut(\mathcal{Q})$) will mean that the bundle map $\tau$ is assumed to cover the diffeomorphism $\varphi$ of the base manifold $M$.
\end{definition}

\begin{definition}\label{def:adjoint}
	Let $A\in\Omega^1(\mathcal{F},Q)$. Recall that thanks to the inner product on $\mathcal{Q}$, $\Gamma(Q^*)\cong\Gamma(Q)$ as $\smooth(M)$-modules. We will denote by $A^\adjoint:\Gamma(Q)\to\Gamma(F^*)$ the dual map of $A:\Gamma(F)\to\Gamma(Q)$, which is defined by $\big\langle A(X),\bar{a}\big\rangle_Q=\big\langle X,A^\adjoint(\bar{a})\big\rangle_Q$ for all $X\in\Gamma(F)$ and $\bar{a}\in\Gamma(Q)$. 
\end{definition}

\begin{definition}
    Let $B\in\Omega^2(\mathcal{F})$. We define a map $B^\sharp:\Gamma(F)\to\Gamma(F^*)$ by setting $B^\sharp(X)(Y)=B(X,Y)$ for any $X$ and $Y\in\Gamma(F)$.
\end{definition}

The following proposition will be used to get the general form of a change of dissection in theorem \ref{thm:dissection-change}.

\begin{proposition}[{\cite[section 2.4]{MR3022918}}]\label{pr:isomorphism-standard-explicit}
	Any Courant algebroid isomorphism $\Phi:\mathcal{S}_M[\nabla,R,H]\to\mathcal{S}_M[\hat{\nabla},\hat{R},\hat{H}]$, covering the identity, is of the form
	\begin{equation}
		\Phi(\alpha\oplus\bar{a}\oplus X)={\alpha+\opeInsDeRham_X B-\frac{1}{2}A^\adjoint\big(A(X)\big)-A^\adjoint\big(\tau(\bar{a})\big)}\oplus{\tau(\bar{a})+A(X)}\oplus X,\label{eq:isomorphism-standard-explicit}
	\end{equation}
	for some $A\in\Omega^1(\mathcal{F},Q)$, $B\in\Omega^2(\mathcal{F})$ and $\tau\in\Aut(\mathcal{Q})$.
\end{proposition}

\begin{proof}
	Let $f\in\smooth(M)$. Then for any $u,v\in\Gamma(S)$, $S=F^*\oplus Q\oplus F$, we have $\Phi\big([fu,v]\big)=\big[\Phi(fu),\Phi(v)\big]$. We now expand both sides. Firstly,
	\begin{align*}
		\Phi\big([fu,v]\big)&=\Phi\big(f[u,v]-(\anchor(v)\cdot f)u+\langle u,v\rangle\opeD f\big)\\
		&=f\Phi\big([u,v]\big)-(\anchor(v)\cdot f)\Phi(u)+\langle u,v\rangle\Phi(\opeD f),
	\end{align*}
	and secondly,
	\begin{align*}
		\big[\Phi(fu),\Phi(v)\big]&=\big[f\Phi(u),\Phi(v)\big]\\
		&=f\big[\Phi(u),\Phi(v)\big]-\big(\anchor\left(\Phi(v)\right)\cdot f\big)\Phi(u)+\big\langle\Phi(u),\Phi(v)\big\rangle\opeD f,
	\end{align*}
	hence
	\begin{equation}
		-(\anchor(v)\cdot f)\Phi(u)+\langle u,v\rangle\Phi(\opeD f)=-\left(\anchor\left(\Phi(v)\right)\cdot f\right)\Phi(u)+\big\langle\Phi(u),\Phi(v)\big\rangle\opeD f.\label{eq:proof1}
	\end{equation}
	Therefore, taking both $u$ and $v$ in $\Gamma(F)$, we obtain $\langle u,v\rangle=0$ and \eqref{eq:proof1} becomes $(\anchor(v)\cdot f)\Phi(u)=\left(\anchor\left(\Phi(v)\right)\cdot f\right)\Phi(u)$. $\Phi$ being an isomorphism and the relation being valid for any $f$, we obtain that $\anchor\left(\Phi(v)\right)=v$ for any $v\in\Gamma(F)$. Now, still for $v\in\Gamma(F)$, \eqref{eq:proof1} yields $\Phi(\opeD f)=\opeD f$, and since by \eqref{pr:orthogonal-ker-anchor} $(\Ker\anchor)^\bot\cong F^*$ is generated by $\Ima\opeD$, we obtain $\Phi|_{F^*}=\opeId_{F^*}$. We deduce from these observations that for any $\bar{a}\in\Gamma(Q)$ and $\alpha\in\Gamma(F^*)$ we have $\big\langle\Phi(\bar{a}),\alpha\big\rangle=\big\langle\Phi(\bar{a}),\Phi(\alpha)\big\rangle=\langle\bar{a},\alpha\rangle=0$, so finally $\Phi$ can be written in matrix form as
	\begin{equation*}
		\Phi=\begin{bmatrix}
		    \opeId & \gamma & \beta \\
		    0 & \tau & A \\
		    0 & 0 & \opeId
		\end{bmatrix},
	\end{equation*}
	for some maps $\tau:\Gamma(Q)\to\Gamma(Q)$, $A:\Gamma(F)\to\Gamma(Q)$, $\beta:\Gamma(F)\to\Gamma(F^*)$ and $\gamma:\Gamma(Q)\to\Gamma(F^*)$. $\Phi$ has to preserve the inner product so $\big\langle\Phi(X),\Phi(X)\big\rangle=0$ for any $X\in\Gamma(F)$, and then $\beta(X)=-\frac{1}{2}A^\adjoint\circ A(X)$. More generally, we can add to $\beta$ any element in $\Omega^2(\mathcal{F})$ without breaking the orthogonality assumption for $\Phi$, so we can write
	\begin{equation*}
		\beta=B^\sharp-\frac{1}{2}A^\adjoint\circ A,
	\end{equation*}
	for some $B\in\Omega^2(\mathcal{F})$. Now $\Phi$ being an orthogonal transformation, we have $\big\langle\Phi(\bar{a}),\Phi(X)\big\rangle=\langle\bar{a},X\rangle=0$ for all $X\in\Gamma(F)$ and $\bar{a}\in\Gamma(Q)$; hence
	\begin{equation*}
		\gamma=-A^\adjoint\circ\tau.
	\end{equation*}
	We also have that $\tau\in\groO(Q)$. Indeed,
	\begin{equation*}
		\langle\bar{a},\bar{b}\rangle_Q=\big\langle\Phi(\bar{a}),\Phi(\bar{b})\big\rangle=\big\langle{-A^\adjoint(\tau(\bar{a}))}\oplus{\tau(\bar{a})},{-A(\tau(\bar{b}))}\oplus{\tau(\bar{b})}\big\rangle=\big\langle\tau(\bar{a}),\tau(\bar{b})\big\rangle_Q,
	\end{equation*}
	and $\tau$ is invertible since $\Phi$ is. Moreover, since $\Phi$ is an isomorphism of Courant algebroids, $\Phi\left([\bar{a},\bar{b}]_{\nabla,\,R,\,H}\right)=\left[\Phi(\bar{a}),\Phi(\bar{b})\right]_{\hat{\nabla},\,\hat{R},\,\hat{H}}$ holds for all $\bar{a},\bar{b}\in\Gamma(Q)$; expanding this expression on both sides we obtain
	\begin{equation*}
	    \langle\nabla\bar{a},\bar{b}\rangle_Q-A^\adjoint\left(\tau([\bar{a},\bar{b}]_Q)\right)\oplus\tau([\bar{a},\bar{b}]_Q)=\langle\hat{\nabla}\tau(\bar{a}),\tau(\bar{b})\rangle_Q\oplus[\tau(\bar{a}),\tau(\bar{b})]_Q,
	\end{equation*}
	and projecting on $\Gamma(Q)$ this relation implies that $\tau\in\Aut(\mathcal{Q})$ eventually.
\end{proof}

Using the fact that a Courant algebroid morphism has to preserve the brackets (see definition \ref{def:Courant-algebroid-morphism}) as well as the preceding lemma applied to the change of dissection $\delta:\mathcal{S}_M[\nabla,R,H]\to\mathcal{S}_M[\hat{\nabla},\hat{R},\hat{H}]$, we are going to establish relations between $\nabla$, $R$ and $H$ on one side, and $\hat{\nabla}$, $\hat{R}$ and $\hat{H}$ on the other side. To this end, we write the change of dissection $\delta$ as $\delta=\Psi_\tau\circ\Psi_A\circ\Psi_B$ , where $\Psi_\tau$, $\Psi_A$ and $\Psi_B$ are the orthogonal automorphisms (as it is easy to check) of the vector bundle $F^*\oplus Q\oplus F\to M$ respectively associated to $A\in\Omega^1(\mathcal{F},Q)$, $B\in\Omega^2(\mathcal{F})$, $\tau\in\Aut(\mathcal{Q})$, and defined by
\begin{align}
	\Psi_\tau(\alpha\oplus\bar{a}\oplus X) &= \alpha\oplus{\tau(\bar{a})}\oplus X\label{psi-tau},\\
	\Psi_A(\alpha\oplus\bar{a}\oplus X) &= {\alpha-\frac{1}{2}A^\adjoint\big(A(X)\big)-A^\adjoint(\bar{a})}\oplus{\bar{a}+A(X)}\oplus X,\label{psi-A}\\
	\Psi_B(\alpha\oplus\bar{a}\oplus X) &= {\alpha+\opeInsDeRham_X B}\oplus{\bar{a}}\oplus X,\label{psi-B}
\end{align}
for all $\alpha\in\Gamma(F^*)$, $\bar{a}\in\Gamma(Q)$ and $X\in\Gamma(F)$. We can also write the maps $\Psi_\tau$, $\Psi_A$ and $\Psi_B$ in matrix form as
\begin{equation*}
    \Psi_\tau=
    \begin{bmatrix}
        \opeId & 0 & 0\\
        0 & \tau & 0\\
        0 & 0 & \opeId
    \end{bmatrix}
    ,\quad
    \Psi_A=
    \begin{bmatrix}
        \opeId & -A^\adjoint & -\frac{1}{2}A^\adjoint\circ A\\
        0 & \opeId & A\\
        0 & 0 & \opeId
    \end{bmatrix}
    ,\quad
    \Psi_B=
        \begin{bmatrix}
        \opeId & 0 & B^\sharp\\
        0 & \opeId & 0\\
        0 & 0 & \opeId
    \end{bmatrix}.
\end{equation*}
Note that relatively to \eqref{eq:isomorphism-standard-explicit}, we made the bijective change $A\mapsto\tau\circ A$, with $\tau\in\Aut(\mathcal{Q})$.

\begin{lemma}\label{pr:effect-psi-tau}
    Let $M$ be a manifold, $F\to M$ an involutive distribution of $TM\to M$, associated to a foliation $\mathcal{F}$ of $M$, and $\mathcal{Q}=(Q\to M,[\cdot,\cdot]_Q,\langle\cdot,\cdot\rangle_Q)$ a quadratic Lie algebra bundle. Let $\tau\in\Aut(\mathcal{Q})$ and consider the vector bundle endomorphism $\Psi_\tau:F^*\oplus Q\oplus F\to F^*\oplus Q\oplus F$ as defined by \eqref{psi-tau}. Let $\mathcal{S}_M[\nabla,R,H]$ denote a standard Courant algebroid structure on the \emph{source} vector bundle  of $\Psi_\tau$  as described in \ref{pr:standard-Courant-algebroid}. For any $\alpha,\beta\in\Gamma(F^*)$, $\bar{a},\bar{b}\in\Gamma(Q)$ and $X,Y\in\Gamma(F)$, we have (see \ref{rm:compact-bracket} for the notation)
	\begin{equation*}
	    \Psi_\tau[\alpha\oplus\bar{a}\oplus X,\beta\oplus\bar{b}\oplus Y]_{\nabla,\,R,\,H}=\big[\Psi_\tau(\alpha\oplus\bar{a}\oplus X),\Psi_\tau(\beta\oplus\bar{b}\oplus Y)\big]_{\hat{\nabla},\,\hat{R},\,H},
	\end{equation*}
	where $\hat{\nabla}$ and $\hat{R}$ are defined by
	\begin{gather}
	    \hat{\nabla}=\tau\circ\nabla\circ\tau^{-1},\label{eq:proof2}\\
	    \hat{R}=\tau\circ R.\label{eq:proof3}
	\end{gather}
	In other words, $\Psi_\tau$ is a Courant algebroid morphism between $\mathcal{S}_M[\nabla,R,H]$ and $\mathcal{S}_M[\hat{\nabla},\hat{R},H]$, covering the identity.
\end{lemma}

\begin{proof}
	On one side we have
	\begin{align*}
	    \Psi_{\tau}[&\alpha\oplus\bar{a}\oplus X,\beta\oplus\bar{b}\oplus Y]_{\nabla,\,R,\,H}=\\
	    &=\opeLieDeRham_X\beta-\opeInsDeRham_Y\difDeRham\alpha+\langle\nabla\bar{a},\bar{b}\rangle_Q+\langle\bar{a},{\opeInsDeRham}_Y R\rangle_Q-\langle\bar{b},{\opeInsDeRham}_X R\rangle_Q\\
	    &\quad\oplus {\tau\big([\bar{a},\bar{b}]_Q\big)+\tau(\nabla_X\bar{b})-\tau(\nabla_Y\bar{a})+\tau\big(R(X,Y)\big)}\oplus\{X,Y\},
	\end{align*}
	and on the other side
	\begin{align*}
	    \big[\Psi_{\tau}&(\alpha\oplus\bar{a}\oplus X),\Psi_{\tau}(\beta\oplus\bar{b}\oplus Y)\big]_{\hat{\nabla},\,\hat{R},\,\hat{H}}\\
	    &=\big[\alpha\oplus \tau(\bar{a})\oplus X,\beta\oplus \tau(\bar{b})\oplus Y\big]\\
	    &=\opeLieDeRham_X\beta-\opeInsDeRham_Y\difDeRham\alpha+\big\langle\hat{\nabla} \tau(\bar{a}),\tau(\bar{b})\big\rangle_Q+\big\langle \tau(\bar{a}),{\opeInsDeRham}_Y \hat{R}\big\rangle_Q-\big\langle \tau(\bar{b}),{\opeInsDeRham}_X \hat{R}\big\rangle_Q\\
	    &\quad\oplus{\big[\tau(\bar{a}),\tau(\bar{b})\big]_Q+\hat{\nabla}_X \tau(\bar{b})-\hat{\nabla}_Y \tau(\bar{a})+\hat{R}(X,Y)}\oplus{\{X,Y\}}.
	\end{align*}
	Projecting on $\Gamma(Q)$, we obtain the result. Note that the relations \eqref{eq:proof2} and \eqref{eq:proof3} assure in particular that after projecting on $\Gamma(F^*)$ all terms are equal, using furthermore the fact that since $\tau$ is an orthogonal transformation, its adjoint is $\tau^{-1}$, which yields the relations
	\begin{gather*}
	    \left\langle\tau(\bar{a}),\hat{R}(Y,Z)\right\rangle_Q=\left\langle\bar{a},\tau^{-1}\big(\hat{R}(Y,Z)\big)\right\rangle_Q,\\ 
	    \left\langle\hat{\nabla}\tau(\bar{a}),\tau(\bar{b})\right\rangle_Q=\left\langle\tau^{-1}\big(\hat{\nabla} \tau(\bar{a})\big),\bar{b}\right\rangle_Q.
	\end{gather*}
\end{proof}

\begin{lemma}\label{pr:effect-psi-A}
    Let $M$ be a manifold, $F\to M$ an involutive distribution of $TM\to M$, associated to a foliation $\mathcal{F}$ of $M$, and $\mathcal{Q}=(Q\to M,[\cdot,\cdot]_Q,\langle\cdot,\cdot\rangle_Q)$ a quadratic Lie algebra bundle. Let $A\in\Omega^1(\mathcal{F},Q)$ and consider the vector bundle endomorphism $\Psi_A:F^*\oplus Q\oplus F\to F^*\oplus Q\oplus F$ as defined by \eqref{psi-A}. Let $\mathcal{S}_M[\nabla,R,H]$ denote a standard Courant algebroid structure on the \emph{source} vector bundle  of $\Psi_A$  as described in \ref{pr:standard-Courant-algebroid}. For any $\alpha,\beta\in\Gamma(F^*)$, $\bar{a},\bar{b}\in\Gamma(Q)$ and $X,Y\in\Gamma(F)$, we have (see \ref{rm:compact-bracket} for the notation)
    \begin{equation*}
	    \Psi_A[\alpha\oplus\bar{a}\oplus X,\beta\oplus\bar{b}\oplus Y]_{\nabla,\,R,\,H}=\big[\Psi_A(\alpha\oplus\bar{a}\oplus X),\Psi_A(\beta\oplus\bar{b}\oplus Y)\big]_{\hat{\nabla},\,\hat{R},\,\hat{H}},
    \end{equation*}
    where $\hat{\nabla}$, $\hat{R}$ and $\hat{H}$ are defined by
    \begin{gather*}
	    \hat{\nabla}=\nabla-\ad_A,\\
	    \hat{R}=R-\difDeRham_{\hat{\nabla}}A-\frac{1}{2}[A\wedge A]_Q,\\
	    \hat{H}=H-\langle A\wedge\hat{R}\rangle_Q-\frac{1}{2}\langle A\wedge\difDeRham_{\hat{\nabla}}A\rangle_Q-\frac{1}{6}\big\langle A\wedge[A\wedge A]_Q\big\rangle_Q.
    \end{gather*}
    In other words, $\Psi_A$ is a Courant algebroid morphism between $\mathcal{S}_M[\nabla,R,H]$ and $\mathcal{S}_M[\hat{\nabla},\hat{R},\hat{H}]$, covering the identity.
\end{lemma}

\begin{proof}
    We begin by computing $[\Psi_A(\alpha\oplus\bar{a}\oplus X),\Psi_A(\beta\oplus\bar{b}\oplus Y)\big]_{\hat{\nabla},\,\hat{R},\,\hat{H}}$. In what follows, we will purposely omit the references to $\hat{\nabla}$, $\hat{R}$, and $\hat{H}$ in bracket subscripts to lighten the notations somewhat.
    \begin{align*}
	    \big[\Psi_A(\alpha\oplus\bar{a}\oplus X)&,\Psi_A(\beta\oplus\bar{b}\oplus Y)\big]\\
	    &=\big[\alpha-A^\adjoint(\bar{a})-\frac{1}{2}A^\adjoint\circ A(X)\oplus\bar{a}+\opeInsDeRham_X A\oplus X,\\
	    &\quad\quad\beta-A^\adjoint(\bar{b})-\frac{1}{2}A^\adjoint\circ A(Y)\oplus\bar{b}+\opeInsDeRham_Y A\oplus Y\big]\\
	    &=[\alpha\oplus\bar{a}\oplus X,\beta\oplus\bar{b}\oplus Y]-[A^\adjoint(\bar{a}),Y]-\frac{1}{2}[A^\adjoint\circ A(X),Y]\\
	    &\quad\quad+[\bar{a},A(Y)]-[X,A^\adjoint(\bar{b})]-\frac{1}{2}[X,A^\adjoint\circ A(Y)]+[X,A(Y)]\\
	    &\quad\quad+[A(X),A(Y)]+[A(X),\bar{b}]+[A(X),Y].
    \end{align*}
    Then, expanding all the brackets except the first one, we obtain
    \begin{equation}\label{eq:proof4}
        \begin{aligned}
	        \big[\Psi_A(\alpha&\oplus\bar{a}\oplus X),\Psi_A(\beta\oplus\bar{b}\oplus Y)\big]=[\alpha\oplus\bar{a}\oplus X,\beta\oplus\bar{b}\oplus Y]\\
	        &+\opeLieDeRham_Y A^\adjoint(\bar{a})-\difDeRham\opeInsDeRham_Y A^\adjoint(\bar{a})+\frac{1}{2}\opeLieDeRham_Y A^\adjoint(A(X))-\frac{1}{2}\difDeRham\opeInsDeRham_Y A^\adjoint(A(X))\\
	        &+\big\langle A(Y),\hat{\nabla}\bar{a}\big\rangle_Q+\big[\bar{a},A(Y)\big]_Q-\opeLieDeRham_X A^\adjoint(\bar{b})-\frac{1}{2}\opeLieDeRham_X A^\adjoint(A(Y))\\
	        &-\big\langle A(Y),\opeInsDeRham_X \hat{R}\big\rangle_Q+\hat{\nabla}_X A(Y)+\big\langle A(Y),\hat{\nabla} A(X)\big\rangle_Q+\big\langle\bar{b},\hat{\nabla} A(X)\big\rangle_Q\\
	        &+\big[A(X),A(Y)\big]_Q+\big[A(X),\bar{b}\big]_Q+\big\langle A(X),\opeInsDeRham_Y \hat{R}\big\rangle_Q-\hat{\nabla}_Y A(X).
	    \end{aligned}
    \end{equation}
    In order to make the computations easier, let $Z\in\Gamma(F)$, we are going to evaluate $\big[\Psi_A(\alpha\oplus\bar{a}\oplus X),\Psi_A(\beta\oplus\bar{b}\oplus Y)\big]\in\Gamma(F^*)\oplus\Gamma(Q)\oplus\Gamma(F)$ on $Z$; thus we will obtain an element of $\smooth(M)\oplus\Gamma(Q)\oplus\Gamma(F)$. We compute separately the terms
    \begin{align*}
        &\opeInsDeRham_Z\opeLieDeRham_Y A^\adjoint(\bar{a})-\opeInsDeRham_Z\difDeRham\opeInsDeRham_Y A^\adjoint(\bar{a})+\frac{1}{2}\opeInsDeRham_Z\opeLieDeRham_Y A^\adjoint(A(X))\\
        &\quad-\frac{1}{2}\opeInsDeRham_Z\difDeRham\opeInsDeRham_Y A^\adjoint(A(X))-\opeInsDeRham_Z\opeLieDeRham_X A^\adjoint(\bar{b})-\frac{1}{2}\opeInsDeRham_Z\opeLieDeRham_X A^\adjoint(A(Y))
    \end{align*}
    extracted from the previous expression \eqref{eq:proof4}:
    \begin{align*}
	    &\opeInsDeRham_Z\opeLieDeRham_Y A^\adjoint(\bar{a})-\opeInsDeRham_Z\difDeRham\opeInsDeRham_Y A^\adjoint(\bar{a})+\frac{1}{2}\opeInsDeRham_Z\opeLieDeRham_Y A^\adjoint(A(X))\\
	    &\quad-\frac{1}{2}\opeInsDeRham_Z\difDeRham\opeInsDeRham_Y A^\adjoint(A(X))-\opeInsDeRham_Z\opeLieDeRham_X A^\adjoint(\bar{b})-\frac{1}{2}\opeInsDeRham_Z\opeLieDeRham_X A^\adjoint(A(Y))=\\
	    &\big\langle\hat{\nabla}_Y A(Z),\bar{a}\big\rangle_Q+\big\langle A(Z),\hat{\nabla}_Y\bar{a}\big\rangle_Q-\big\langle A(\{Y,Z\}),\bar{a}\big\rangle_Q-\big\langle\hat{\nabla}_Z A(Y),\bar{a}\big\rangle_Q\\
	    &\quad-\big\langle A(Y),\hat{\nabla}_Z\bar{a}\big\rangle_Q+\frac{1}{2}\big\langle\hat{\nabla}_Y A(X),A(Z)\big\rangle_Q+\frac{1}{2}\big\langle A(X),\hat{\nabla}_Y A(Z)\big\rangle_Q\\
	    &\quad-\frac{1}{2}\big\langle\hat{\nabla}_Z A(X),A(Y)\big\rangle_Q-\frac{1}{2}\big\langle A(X),\hat{\nabla}_Z A(Y)\big\rangle_Q-\frac{1}{2}\big\langle A(\{Y,Z\}),A(X)\big\rangle_Q\\
	    &\quad-\big\langle\hat{\nabla}_X A(Z),\bar{b}\big\rangle_Q-\big\langle A(Z),\hat{\nabla}_X\bar{b}\big\rangle_Q+\big\langle A(\{X,Z\}),\bar{b}\big\rangle_Q\\
	    &\quad-\frac{1}{2}\big\langle\hat{\nabla}_X A(Y),A(Z)\big\rangle_Q-\frac{1}{2}\big\langle A(Y),\hat{\nabla}_X A(Z)\big\rangle_Q+\frac{1}{2}\big\langle A(\{X,Z\}),A(Y)\big\rangle_Q,
    \end{align*}
    which yields, after plugging back the result into \eqref{eq:proof4} and evaluating on $Z$:
    \begin{equation}\label{eq:proof5}
    	\begin{aligned}
    		&\opeInsDeRham_Z\big[\Psi_A(\alpha\oplus\bar{a}\oplus X),\Psi_A(\beta\oplus\bar{b}\oplus Y)\big]=\opeInsDeRham_Z[\alpha\oplus\bar{a}\oplus X,\beta\oplus\bar{b}\oplus Y]\\
    		&\quad+\big\langle\hat{\nabla}_Y A(Z),\bar{a}\big\rangle_Q+\big\langle A(Z),\hat{\nabla}_Y\bar{a}\big\rangle_Q-\big\langle A(\{Y,Z\}),\bar{a}\big\rangle_Q-\big\langle\hat{\nabla}_Z A(Y),\bar{a}\big\rangle_Q\\
    		&\quad-\big\langle A(Y),\hat{\nabla}_Z\bar{a}\big\rangle_Q+\frac{1}{2}\big\langle\hat{\nabla}_Y A(X),A(Z)\big\rangle_Q+\frac{1}{2}\big\langle A(X),\hat{\nabla}_Y A(Z)\big\rangle_Q\\
    		&\quad-\frac{1}{2}\big\langle\hat{\nabla}_Z A(X),A(Y)\big\rangle_Q-\frac{1}{2}\big\langle A(X),\hat{\nabla}_Z A(Y)\big\rangle_Q-\frac{1}{2}\big\langle A(\{Y,Z\}),A(X)\big\rangle_Q\\
    		&\quad-\big\langle\hat{\nabla}_X A(Z),\bar{b}\big\rangle_Q-\big\langle A(Z),\hat{\nabla}_X\bar{b}\big\rangle_Q+\big\langle A(\{X,Z\}),\bar{b}\big\rangle_Q\\
    		&\quad-\frac{1}{2}\big\langle\hat{\nabla}_X A(Y),A(Z)\big\rangle_Q-\frac{1}{2}\big\langle A(Y),\hat{\nabla}_X A(Z)\big\rangle_Q+\frac{1}{2}\big\langle A(\{X,Z\}),A(Y)\big\rangle_Q\\
    		&\quad+\big\langle A(Y),\hat{\nabla}_Z\bar{a}\big\rangle_Q+[\bar{a},A(Y)]_Q-\big\langle A(Y),\hat{R}(X,Z)\big\rangle_Q+\hat{\nabla}_X A(Y)\\
    		&\quad+\big\langle A(Y),\hat{\nabla}_Z A(X)\big\rangle_Q+\big\langle\bar{b},\hat{\nabla}_Z A(X)\big\rangle_Q+\big[A(X),A(Y)\big]_Q\\
    		&\quad+\big[A(X),\bar{b}\big]_Q+\big\langle A(X),\hat{R}(Y,Z)\big\rangle_Q-\hat{\nabla}_Y A(X).
    	\end{aligned}
    \end{equation}
    Now we can also compute (the left-hand side bracket subscript being important at this point):
    \begin{align*}
    	\Psi_A[\alpha\oplus\bar{a}\oplus X,&\beta\oplus\bar{b}\oplus Y]_{\hat{\nabla},\,\hat{R},\,\hat{H}}\\
    	&=[\alpha\oplus\bar{a}\oplus X,\beta\oplus\bar{b}\oplus Y]\\
    	&\quad\quad-A^\adjoint\big([\bar{a},\bar{b}]_Q\big)-A^\adjoint\big(\hat{R}(X,Y)\big)-A^\adjoint\big(\hat{\nabla}_X\bar{b}\big)\\
    	&\quad\quad+A^\adjoint\big(\hat{\nabla}_Y\bar{a}\big)-\frac{1}{2}(A^\adjoint\circ A)\big(\{X,Y\}\big)+A\big(\{X,Y\}\big),
    \end{align*}
    where the bracket on the right-and side still refers to the one on $\mathcal{S}_M[\hat{\nabla},\hat{R},\hat{H}]$. Therefore, substituting $[\alpha\oplus\bar{a}\oplus X,\beta\oplus\bar{b}\oplus Y]$ from the above expression to $[\alpha\oplus\bar{a}\oplus X,\beta\oplus\bar{b}\oplus Y]$ in \eqref{eq:proof5}, we obtain
    \begin{equation}\label{eq:proof6}
    	\begin{aligned}
    		&\opeInsDeRham_Z\big[\Psi_A(\alpha\oplus\bar{a}\oplus X),\Psi_A(\beta\oplus\bar{b}\oplus Y)\big]_{\hat{\nabla},\,\hat{R},\,\hat{H}}\\
    		&\quad=\opeInsDeRham_Z\Psi_A[\alpha\oplus\bar{a}\oplus X,\beta\oplus\bar{b}\oplus Y]_{\hat{\nabla},\,\hat{R},\,\hat{H}}\\
    		&\quad\quad+\big\langle A(Z),[\bar{a},\bar{b}]_Q\big\rangle_Q+\big\langle A(Z),\hat{R}(X,Y)\big\rangle_Q+\big\langle A(Z),\hat{\nabla}_X\bar{b}\big\rangle_Q\\
    		&\quad\quad-\big\langle A(Z),\hat{\nabla}_Y\bar{a}\big\rangle_Q+\frac{1}{2}\big\langle A(Z),A(\{X,Y\})\big\rangle_Q-A\big(\{X,Y\}\big)\\
    		&\quad\quad+\text{ the list of terms in \eqref{eq:proof5} except the first one}.
    	\end{aligned}
    \end{equation}
    Now the goal is to identify the supplementary terms in \eqref{eq:proof6}. The terms
    \begin{multline*}
        \big\langle A(Z),\hat{\nabla}_Y\bar{a}\big\rangle_Q-\big\langle A(Y),\hat{\nabla}_Z\bar{a}\big\rangle_Q-\big\langle A(Z),\hat{\nabla}_X\bar{b}\big\rangle_Q+\big\langle A(Y),\hat{\nabla}_Z\bar{a}\big\rangle_Q\\
        -\big\langle A(Z),\hat{\nabla}_Y\bar{a}\big\rangle_Q+\big\langle A(Z),\hat{\nabla}_X\bar{b}\big\rangle_Q
    \end{multline*}
    cancel out. Then the terms
    \begin{equation*}
        -\big\langle A(Y),\hat{R}(X,Z)\big\rangle_Q+\big\langle A(X),\hat{R}(Y,Z)\big\rangle_Q+\big\langle A(Z),\hat{R}(X,Y)\big\rangle_Q
    \end{equation*}
    are collected into the $3$-differential form $\langle A\wedge\hat{R}\rangle_Q$ that we can add to $\hat{H}$, the term $\big[A(X),A(Y)\big]_Q$ is the $\mathcal{Q}$-valued (tangential) $2$-differential form $\frac{1}{2}[A\wedge A]_Q$ and the terms
    \begin{equation*}
        \hat{\nabla}_X A(Y)-\hat{\nabla}_Y A(X)-A\big(\{X,Y\}\big)
    \end{equation*}
    are collected into the $\mathcal{Q}$-valued $2$-differential form $\difDeRham_{\hat{\nabla}} A$ (see lemma \ref{pr:lemma-forms}). 
    
    But if we add $\difDeRham_{\hat{\nabla}} A$ to $\hat{R}$ and consider the associated bracket $[\cdot,\cdot]_{\hat{\nabla},\,\hat{R}+\difDeRham_\nabla A,\,\hat{H}}$, then we have to take into account supplementary terms coming from this addition. In other words:
    \begin{align*}
    	&\Psi_A[\alpha\oplus\bar{a}\oplus X,\beta\oplus\bar{b}\oplus Y]_{\hat{\nabla},\,\hat{R},\,\hat{H}+\langle A\wedge \hat{R}\rangle_Q}\\
    	&\quad\quad+\text{ supplementary terms of \eqref{eq:proof6}}\\
    	&\quad=\Psi_A[\alpha\oplus\bar{a}\oplus X,\beta\oplus\bar{b}\oplus Y]_{\hat{\nabla},\,\hat{R}+\difDeRham_{\hat{\nabla}} A,\,\hat{H}+\langle A\wedge\hat{R}\rangle_Q}-\text{ terms in which}\\
    	&\quad\quad\ \text{$\difDeRham_{\hat{\nabla}} A$ appears in }\Psi_A[\alpha\oplus\bar{a}\oplus X,\beta\oplus\bar{b}\oplus Y]_{\hat{\nabla},\,\hat{R}+\difDeRham_{\hat{\nabla}}A,\,\hat{H}+\langle A\wedge\hat{R}\rangle_Q}\\
    	&\quad\quad+\text{ supplementary terms of \eqref{eq:proof6}.}
    \end{align*}
    Evaluating on $Z\in\Gamma(F)$ we obtain
    \begin{align*}
    	&\opeInsDeRham_Z\big[\Psi_A(\alpha\oplus\bar{a}\oplus X),\Psi_A(\beta\oplus\bar{b}\oplus Y)\big]_{\hat{\nabla},\,\hat{R},\,\hat{H}}\\
    	&\quad=\opeInsDeRham_Z\Psi_A[\alpha\oplus\bar{a}\oplus X,\beta\oplus\bar{b}\oplus Y]_{\hat{\nabla},\,\hat{R}+\difDeRham_{\hat{\nabla}} A,\,\hat{H}+\langle A\wedge\hat{R}\rangle_Q}\\
    	&\quad\quad+\big\langle\bar{a},\difDeRham_{\hat{\nabla}} A(X,Z)\big\rangle_A-\big\langle\bar{a},\difDeRham_{\hat{\nabla}} A(Y,Z)\big\rangle_Q\\
    	&\quad\quad+\opeInsDeRham_Z A^\adjoint\big(\difDeRham_{\hat{\nabla}} A(X,Y)\big)-\difDeRham_{\hat{\nabla}}A(X,Y)+\text{ supplementary terms of \eqref{eq:proof6}}.
    \end{align*}
    The $-\difDeRham_{\hat{\nabla}} A(X,Y)$ term cancels out with the one found previously, and the terms $\big\langle\bar{a},\difDeRham_{\hat{\nabla}} A(X,Z)\big\rangle_Q-\big\langle\bar{a},\difDeRham_{\hat{\nabla}} A(Y,Z)\big\rangle_Q$ cancel out with the terms 
    \begin{multline*}
        \big\langle\hat{\nabla}_Y A(Z),\bar{a}\big\rangle_Q-\big\langle A(\{Y,Z\}),\bar{a}\big\rangle_Q-\big\langle\hat{\nabla}_Z A(Y),\bar{a}\big\rangle_Q\\
        -\big\langle\hat{\nabla}_X A(Z),\bar{b}\big\rangle_Q+\big\langle A(\{X,Z\}),\bar{b}\big\rangle_Q+\big\langle\bar{b},\hat{\nabla}_Z A(X)\big\rangle_Q
    \end{multline*}
    coming from \eqref{eq:proof5}; and finally the term $\opeInsDeRham_Z A^\adjoint\big(\difDeRham_{\hat{\nabla}} A(X,Y)\big)$ combined with the terms
    \begin{multline*}
        \frac{1}{2}\big\langle\hat{\nabla}_Y A(X),A(Z)\big\rangle_Q+\frac{1}{2}\big\langle A(X),\hat{\nabla}_Y A(Z)\big\rangle_Q+\frac{1}{2}\big\langle\hat{\nabla}_Z A(X),A(Y)\big\rangle_Q\\
        -\frac{1}{2}\big\langle A(X),\hat{\nabla}_Z A(Y)\big\rangle_Q-\frac{1}{2}\big\langle A(\{Y,Z\}),A(X)\big\rangle_Q-\frac{1}{2}\big\langle\hat{\nabla}_X A(Y),A(Z)\big\rangle_Q\\
        -\frac{1}{2}\big\langle A(Y),\hat{\nabla}_X A(Z)\big\rangle_Q+\frac{1}{2}\big\langle A(\{X,Z\}),A(Y)\big\rangle_Q+\frac{1}{2}\big\langle A(Z),A(\{X,Y\})\big\rangle_Q
    \end{multline*}
    coming from \eqref{eq:proof5} and \eqref{eq:proof6} are collected into the $3$-differential form $\frac{1}{2}\langle A\wedge\difDeRham_{\hat{\nabla}} A\rangle_Q$ that we will add further to $\hat{H}$. 
    
    But what we have done for $\difDeRham_{\hat{\nabla}} A$, we need to repeat it for the term $\frac{1}{2}[A\wedge A]_Q$ that we have found before. As supplementary terms we obtain
    \begin{equation*}
    	\big\langle\bar{b},[A(X),A(Z)]_Q\big\rangle_Q-\big\langle\bar{a},[A(Y),A(Z)]_Q\big\rangle_Q+\big\langle A(Z),[A(X),A(Y)]_Q\big\rangle_Q.
    \end{equation*}
    The two first terms do not cancel out and the last one corresponds to the $3$-differential form $\frac{1}{6}\big\langle A\wedge[A\wedge A]_Q\big\rangle_Q$ that will be added to $\hat{H}$. Now the two remaining terms in \eqref{eq:proof5} correspond to $\ad_A$ added to $\hat{\nabla}$; however it is necessary to remove the corresponding supplementary terms
    \begin{equation*}
    	-\big\langle[A(Z),\bar{a}]_Q,\bar{b}\big\rangle_Q+\big\langle A(Z),[A(X),\bar{b}]_Q\big\rangle_Q-\big\langle A(Z),[A(Y),\bar{a}]_Q\big\rangle_Q,
    \end{equation*}
    which do not cancel out with other terms. Taking also into account the term $\big\langle A(Z),[\bar{a},\bar{b}]_Q\big\rangle_Q$ remaining in \eqref{eq:proof6}, we obtain at last that
    \begin{align*}
    	&\big[\Psi_A(\alpha\oplus\bar{a}\oplus X),\Psi_A(\beta\oplus\bar{b}\oplus Y)\big]_{\hat{\nabla},\,\hat{R},\,\hat{H}}\\
    	&\quad=\Psi_A[\alpha\oplus\bar{a}\oplus X,\beta\oplus\bar{b}\oplus Y]_{\hat{\nabla}+\bar{\nabla},\,\hat{R}+\bar{R},\,\hat{H}+\bar{H}}\\
    	&\quad\quad+\big\langle A(Z),[\bar{a},\bar{b}]_Q\big\rangle_Q+\big\langle\bar{b},[A(X),A(Z)]_Q\big\rangle_Q+\big\langle\bar{a},[A(Y),A(Z)]_Q\big\rangle_Q\\
    	&\quad\quad-\big\langle[A(Z),\bar{a}]_Q,\bar{b}\big\rangle_Q+\big\langle A(Z),[A(X),\bar{b}]_Q\big\rangle_Q-\big\langle A(Z),[A(Y),\bar{a}]_Q\big\rangle_Q,
    \end{align*}
    where $\bar{\nabla}$, $\bar{R}$ and $\bar{H}$ are defined by
    \begin{align*}
    	\bar{\nabla}&=\ad_A,\\
    	\bar{R}&=\difDeRham_{\hat{\nabla}}A+\frac{1}{2}[A\wedge A]_Q,\\
    	\bar{H}&=\langle A\wedge\hat{R}\rangle_Q+\frac{1}{2}\langle A\wedge\difDeRham_{\hat{\nabla}} A\rangle_Q+\frac{1}{6}\left\langle A\wedge[A\wedge A]_Q\right\rangle_Q.
    \end{align*}
    The six remaining terms all cancel out thanks to \eqref{pr:anchor-preserves-metric}, which ends the proof.
\end{proof}

\begin{lemma}\label{pr:effect-psi-B}
    Let $M$ be a manifold, $F\to M$ an involutive distribution of $TM\to M$, associated to a foliation $\mathcal{F}$ of $M$, and $\mathcal{Q}=(Q\to M,[\cdot,\cdot]_Q,\langle\cdot,\cdot\rangle_Q)$ a quadratic Lie algebra bundle. Let $B\in\Omega^2(\mathcal{F})$ and consider the vector bundle endomorphism $\Psi_B:F^*\oplus Q\oplus F\to F^*\oplus Q\oplus F$ as defined by \eqref{psi-B}. Let $\mathcal{S}_M[\nabla,R,H]$ denote a standard Courant algebroid structure on the \emph{source} vector bundle  of $\Psi_B$  as described in \ref{pr:standard-Courant-algebroid}. For any $\alpha,\beta\in\Gamma(F^*)$, $\bar{a},\bar{b}\in\Gamma(Q)$ and $X,Y\in\Gamma(F)$, we have (see \ref{rm:compact-bracket} for the notation)
    \begin{equation*}
		\Psi_B[\alpha\oplus\bar{a}\oplus X,\beta\oplus\bar{b}\oplus Y]_{\nabla,\,R,\,H}=\big[\Psi_B(\alpha\oplus\bar{a}\oplus X),\Psi_B(\beta\oplus\bar{b}\oplus Y)\big]_{\nabla,\,R,\,\hat{H}},
	\end{equation*}
	where $\hat{H}$ is defined by $\hat{H}=H-\difDeRham B$. In other words, $\Psi_B$ is a Courant algebroid morphism between $\mathcal{S}_M[\nabla,R,H]$ and $\mathcal{S}_M[\nabla,R,\hat{H}]$, covering the identity.
\end{lemma}

\begin{proof}
	We begin by computing $[\Psi_B(\alpha\oplus\bar{a}\oplus X),\Psi_B(\beta\oplus\bar{b}\oplus Y)\big]_{\hat{\nabla},\,\hat{R},\,\hat{H}}$. In what follows, we will purposely omit the references to $\hat{\nabla}$, $\hat{R}$, and $\hat{H}$ in bracket subscripts to lighten the notations somewhat.
	\begin{align*}
		\big[\Psi_B(\alpha\oplus\bar{a}\oplus X),&\Psi_B(\beta\oplus\bar{b}\oplus Y)\big]\\
		&=\big[\alpha+B(X)\oplus\bar{a}\oplus X,\beta+B(Y)\oplus\bar{b}\oplus Y\big]\\
		&=[\alpha\oplus\bar{a}\oplus X,\beta\oplus\bar{b}\oplus Y]+[B(X),Y]+[X,B(Y)]
	\end{align*}
	Using relations concerning the Cartan operations we obtain further
	\begin{align*}
		[B(X),Y]+&[X,B(Y)]\\
		&=-\opeLieDeRham_Y\opeInsDeRham_X B+\difDeRham\opeInsDeRham_Y\opeInsDeRham_X B+\opeLieDeRham_X\opeInsDeRham_Y B\\
		&=-\difDeRham\opeInsDeRham_Y\opeInsDeRham_X B-\opeInsDeRham_Y\difDeRham\opeInsDeRham_X B+ \difDeRham\opeInsDeRham_Y\opeInsDeRham_X B+\opeInsDeRham_X\difDeRham\opeInsDeRham_Y B+\difDeRham\opeInsDeRham_X\opeInsDeRham_Y B\\
		&=-\opeLieDeRham_Y\opeInsDeRham_X B+\opeInsDeRham_X\opeLieDeRham_Y B-\opeInsDeRham_X\opeInsDeRham_Y\difDeRham B\\
		&=\opeInsDeRham_{\{X,Y\}}B-\opeInsDeRham_X\opeInsDeRham_Y\difDeRham B,
	\end{align*}
	which yields
	\begin{equation*}
		\begin{split}
			\big[\Psi_B(\alpha\oplus\bar{a}\oplus X)&,\Psi_B(\beta\oplus\bar{b}\oplus Y)\big]_{\hat{\nabla},\,\hat{R},\,\hat{H}}\\
			&=\Psi_B[\alpha\oplus\bar{a}\oplus X,\beta\oplus\bar{b}\oplus Y]_{\hat{\nabla},\,\hat{R},\,\hat{H}}+\opeInsDeRham_Y\opeInsDeRham_X\difDeRham B\\
			&=\Psi_B[\alpha\oplus\bar{a}\oplus X,\beta\oplus\bar{b}\oplus Y]_{\hat{\nabla},\,\hat{R},\,\hat{H}+\difDeRham B},
		\end{split}
	\end{equation*}
	from which the result is deduced.
\end{proof}

The following theorem describes the effect of a change of dissection on a regular Courant algebroid. It is a direct consequence of the four previous statements. Essentially this result is similar to the one obtained in \cite[proposition 2.7]{MR3022918}, albeit written in another form.

\begin{theorem}\label{thm:dissection-change}
    Let $\mathcal{E}=\left(E\to M,\anchor,[\cdot,\cdot],\langle\cdot,\cdot\rangle\right)$ be a regular Courant algebroid and consider $(\Delta,\nabla,R,H)$ and $(\hat{\Delta},\hat{\nabla},\hat{R},\hat{H})$ two dissections of $\mathcal{E}$ (see remark \ref{rmq:dissection} for the notation). Then the \emph{change of dissection} $\delta=\hat{\Delta}^{-1}\circ\Delta:\mathcal{S}_M[\nabla,R,H]\to\mathcal{S}_M[\hat{\nabla},\hat{R},\hat{H}]$ is a Courant algebroid isomorphism covering the identity of $M$ of the form
	\begin{equation}
		\delta(\alpha\oplus\bar{a}\oplus X)=\alpha+\opeInsDeRham_X B-\frac{1}{2}A^\adjoint\big(A(X)\big)-A^\adjoint(\bar{a})\oplus\tau(\bar{a})+\tau\big(A(X)\big)\oplus X,\label{eq:dissection-change-explicit}
	\end{equation}
	for any $\alpha\in\Gamma(F^*)$, $\bar{a}\in\Gamma(Q)$ and $X\in\Gamma(F)$; and for some $A\in\Omega^1(\mathcal{F},Q)$, $B\in\Omega^2(\mathcal{F})$ and $\tau\in\Aut(\mathcal{Q})$. In matrix form we have
	\begin{equation}
		\delta=
		\begin{bmatrix}
		    \opeId & -A^\adjoint & B^\sharp-\frac{1}{2}A^\adjoint\circ A\\
		    0 & \tau & \tau\circ A\\
		    0 & 0 & \opeId
	    \end{bmatrix}.
	\end{equation}
	Moreover, we have the relations
	\begin{gather}
		\hat{\nabla}=\tau\circ\nabla\circ\tau^{-1}-\ad_{\tau\circ A},\label{delta-trans-1}\\
		\hat{R}=\tau\circ R-\tau\circ\difDeRham_{\nabla}A+\frac{1}{2}[\tau\circ A\wedge\tau\circ A]_Q,\label{delta-trans-2}\\
		\hat{H}=H-\difDeRham B-\langle A\wedge R\rangle_Q+\frac{1}{2}\langle A\wedge\difDeRham_{\nabla}A\rangle_Q-\frac{1}{6}\big\langle A\wedge[A\wedge A]_Q\big\rangle_Q.\label{delta-trans-3}
	\end{gather}
\end{theorem}

\begin{proof}
    The expression \eqref{eq:dissection-change-explicit} is a direct application of proposition \ref{pr:isomorphism-standard-explicit}. The relations \eqref{delta-trans-1}, \eqref{delta-trans-2} and \eqref{delta-trans-3} are obtained from the successive application of lemmas \ref{pr:effect-psi-B}, \ref{pr:effect-psi-A} and \ref{pr:effect-psi-tau}, as well as the identities $\tau\circ\ad_A\circ\tau^{-1}=\ad_{\tau\circ A}$ and $\difDeRham_{\nabla-\ad_A}=\difDeRham_\nabla A-[A\wedge A]_Q$.
\end{proof}

\subsection{Global automorphisms}\label{sec:global-automorphism}

In this section we are interested in the study of the group of automorphisms of a regular Courant algebroid by means of a dissection. What is interesting is the appearance of automorphisms of a new kind, the so-called $A$-fields, in addition to the $B$-fields already known from \emph{generalized complex geometry}, which is partly based on the structure of \emph{exact} Courant algebroids (see \cite{MR2811595}). These new automorphisms made their appearance in \cite{MR3090107}, \cite{BarHek} and \cite{CMTW} as particular cases of Courant algebroids that we will recover in \ref{sec:automorphism-examples}.

\begin{definition}\label{def:automorphism}
    Let $\mathcal{E}=(E\to M,\anchor,[\cdot,\cdot],\langle\cdot,\cdot\rangle)$ be a Courant algebroid. An \emph{automorphism} of $\mathcal{E}$ is an automorphism $(\varphi,\Phi):E\to E$ of the vector bundle $E\to M$ (the notation $(\varphi,\Phi)$ meaning that the map $\Phi$ covers the diffeomorphism $\varphi$) satisfying the following relations for all $u,v\in\Gamma(E)$:
    \begin{gather}
	    \difDeRham\varphi\circ\anchor=\anchor\circ\Phi,\label{eq:auto-1}\\
	    (\varphi^{-1})^*\langle u,v\rangle=\big\langle\Phi(u),\Phi(v)\big\rangle,\label{eq:auto-2}\\
	    \Phi\big([u,v]\big)=\big[\Phi(u),\Phi(v)\big],\label{eq:auto-3}
    \end{gather}
    where in \eqref{eq:auto-2} $(\varphi^{-1})^*$ denotes the pullback of a smooth function by $\varphi^{-1}$, and in \eqref{eq:auto-3} $\Phi$ denotes the induced map on sections of the vector bundle $E\to M$ defined by $\Phi(u)_x=\Phi(u_{\varphi^{-1}(x)})$ for any $u\in\Gamma(E)$ and $x\in M$ (this makes sense because $\varphi$ is a diffeomorphism of $M$).
\end{definition}

The next proposition will show that the condition \eqref{eq:auto-1} is actually unnecessary and follows from both \eqref{eq:auto-2} and \eqref{eq:auto-3} together with the special relations available in any Courant algebroid (see proposition \ref{pr:Courant-algebroid-special-properties}). But first we remark that given an automorphism $(\varphi,\Phi):E\to E$ of some vector bundle $E\to M$, the induced map $\Phi:\Gamma(E)\to\Gamma(E)$ is not $\smooth(M)$-linear since $\Phi:E\to E$ does not cover the identity of $M$; however we have a relation that replaces the $\smooth(M)$-linearity given by the following lemma.

\begin{lemma}\label{pr:non-linear-map}
	Let $(\varphi,\Phi):E\to E$ be an automorphism of a vector bundle $E\to M$ on a manifold $M$. For all $f\in\smooth(M)$ and $s\in\Gamma(E)$ we have the relation $\Phi(fs)=\big((\varphi^{-1})^*f\big)\Phi(s)$.
\end{lemma}

\begin{proof}
	For any point $y\in M$ we have successively
	\begin{align*}
	    \Phi(fs)_y&=\Phi\big((fs)_{\varphi^{-1}(y)}\big)\\
	    &=\Phi\big((f\circ\varphi^{-1})(y)s_{\varphi^{-1}(y)}\big)\\
	    &=(f\circ\varphi^{-1})(y)\Phi(s_{\varphi^{-1}(y)})\\
	    &=\big((\varphi^{-1})^*f\big)(y)\Phi(s)_y.
	\end{align*}
\end{proof}

Hereafter is another lemma that we will need for the proposition. To this end, we recall that given $M$ and $N$ two manifolds and $\varphi:M\to N$ a smooth map, $X\in\mathfrak{X}(M)$ and $Y\in\mathfrak{X}(N)$ are said to be $\varphi$-\emph{compatible} if $\varphi_*(X)=\difDeRham\varphi(X)=\varphi^*(Y)$ where $\varphi^*:\mathfrak{X}(N)\to\Gamma(\varphi^*TM)$ (the target being the sections of $TM$ pullbacked along $\varphi$) is the map defined by $\varphi^*(Y)_x=Y_{\varphi(x)}$ for any $x\in M$ (see \cite[proposition 3.1, chapter 3]{MR1249482} for more details).

\begin{lemma}\label{pr:interior-product-Lie-derivative}
	Let $M$ and $N$ be two manifolds and $\varphi:M\to N$ be a smooth map. Let $X\in\mathfrak{X}(M)$ and $Y\in\mathfrak{X}(N)$ be two $\varphi$-compatible vector fields. For any $\omega\in\Omega^\bullet(N)$ we have
	\begin{equation*}
		\opeInsDeRham_X(\varphi^*\omega)=\varphi^*\opeInsDeRham_Y\omega,\quad\opeLieDeRham_X(\varphi^*\omega)=\varphi^*\opeLieDeRham_Y\omega.
	\end{equation*}
\end{lemma}

\begin{proof}
	Let $\omega\in\Omega^k(N)$ and $X_1,\dots,X_{k-1}\in\mathfrak{X}(M)$. Then
	\begin{align*}
		\big[\opeInsDeRham_X(\varphi^*\omega)\big](X_1,\dots,X_k)&=(\varphi^*\omega)(X,X_1,\dots,X_{k-1})\\
		&=\omega(\varphi_*X,\varphi_*X_1,\dots,\varphi_*X_{k-1})\\
		&=(\opeInsDeRham_Y\omega)(\varphi_*X_1,\dots,\varphi_*X_{k-1})\\
		&=\big[\varphi^*(\opeInsDeRham_Y\omega)\big](X_1,\dots,X_{k-1}).
	\end{align*}
	Now, the second formula results from the application of the first one, a Cartan's formula (\cite[theorem 14.35]{MR2954043}), and the fact that the pullback operation $\varphi^*$ (on differential forms) commutes with the De Rham differential $\difDeRham$ (see \cite[proposition 11.25]{MR2954043}):
	\begin{align*}
		\opeLieDeRham_X(\varphi^*\omega)&=\opeInsDeRham_X\difDeRham\varphi^*\omega+\difDeRham\opeInsDeRham_X\varphi^*\omega\\
		&=\opeInsDeRham_X\varphi^*\difDeRham\omega+\difDeRham\varphi^*\opeInsDeRham_Y\omega\\
		&=\varphi^*\opeInsDeRham_Y\difDeRham\omega+\varphi^*\difDeRham\opeInsDeRham_Y\omega\\
		&=\varphi^*\opeLieDeRham_Y\omega.
	\end{align*}
\end{proof}

Using both lemmas we can now prove the announced proposition.

\begin{proposition}\label{pr:unnecessary}
	Let $\mathcal{E}=\left(E\to M,\anchor,[\cdot,\cdot],\langle\cdot,\cdot\rangle\right)$ be a Courant algebroid and $(\varphi,\Phi):E\to E$ be an automorphism of a vector bundle $E\to M$ on a manifold $M$, that satisfies both conditions \eqref{eq:auto-2} and \eqref{eq:auto-3}. Then $(\varphi,\Phi):E\to E$ satisfies \eqref{eq:auto-1} and is an automorphism of the Courant algebroid $\mathcal{E}$.
\end{proposition}

\begin{proof}
	Let $f\in\smooth(M)$. According to lemma \ref{pr:non-linear-map} and relations \ref{pr:right-Leibniz-identity} and \ref{eq:auto-3}, we have on one side that
	\begin{align*}
	    \big[\Phi(u),\Phi(fv)\big]&=\Phi\big([u,fv]\big)\\
	    &=\Phi\big(f[u,v]+(\anchor(u)\cdot f)v\big)\\
	    &=(\varphi^{-1})^*f\Phi\big([u,v]\big)+(\varphi^{-1})^*\big(\anchor(u)\cdot f\big)\Phi(v),
	    \end{align*}
	and on the other side we have
	\begin{align*}
	    \big[\Phi(u),\Phi(fv)\big]&=\left[\Phi(u),(\varphi^{-1})^*f\Phi(v)\right]\\
	    &=(\varphi^{-1})^*f\big[\Phi(u),\Phi(v)\big]+\left(\anchor(\Phi(u))\cdot(\varphi^{-1})^*f\right)\Phi(v),
	\end{align*}
	which implies
	\begin{equation*}
	    (\varphi^{-1})^*\big(\anchor(u)\cdot f\big)=\anchor(\Phi(u))\cdot(\varphi^{-1})^*f.
	\end{equation*}
	According to lemma \ref{pr:interior-product-Lie-derivative}, it follows that
	\begin{align*}
	    (\varphi^{-1})^*\opeInsDeRham_{\anchor(u)}\difDeRham f&=\opeInsDeRham_{\varphi_*\anchor(u)}(\varphi^{-1})^*\difDeRham f\\
	    &=\opeInsDeRham_{\varphi_*\anchor(u)}\difDeRham(\varphi^{-1})^*f\\
	    &=\big[\varphi_*\anchor(u)\big]\cdot\big[(\varphi^{-1})^*f\big],
	\end{align*}
	which results in \eqref{eq:auto-1}.
\end{proof}

From now on we will consider a regular Courant algebroid $\mathcal{E}=(E\to M,\anchor,[\cdot,\cdot],\langle\cdot,\cdot\rangle)$ and use the notations of the previous section. Let $(\Delta,\nabla,R,H)$ denote a dissection of $\mathcal{E}$. The next proposition states that having chosen a dissection of $\mathcal{E}$, to study the group of automorphisms of $\mathcal{E}$ is equivalent to study the group of automorphisms of the standard Courant algebroid $\mathcal{S}=\mathcal{S}_M[\nabla,R,H]$ obtained from the dissection (see theorem \ref{thm:CSX}).

\begin{proposition}
	There is a group isomorphism $\Aut(\mathcal{E})\cong\Aut(\mathcal{S})$.
\end{proposition}

\begin{proof}
	The group isomorphism $\Aut(\mathcal{E})\to\Aut(\mathcal{S})$ is defined by conjugation: $\Phi\mapsto\Delta^{-1}\circ\Phi\circ\Delta$. Note that $\Delta^{-1}\circ\Phi\circ\Delta$ is actually an automorphism of $\mathcal{S}=\mathcal{S}_M[\nabla,R,H]$ because $\Delta$ is by definition an isomorphism of Courant algebroids covering the identity of $M$ (see definition \ref{def:Courant-algebroid-morphism}).
\end{proof}

\begin{definition}
	We will denote by $\groO(S)$ the group of orthogonal automorphisms of the vector bundle $S=F^*\oplus Q\oplus F\to M$, where the orthogonality requirement is given by condition \eqref{eq:auto-2} and the inner product is $\langle\cdot,\cdot\rangle$ of $\mathcal{S}$ given by 
	\begin{equation*}
	    \big\langle\alpha\oplus\bar{a}\oplus X,\beta\oplus\bar{b}\oplus Y\big\rangle=\alpha(Y)+\beta(X)+\langle\bar{a},\bar{b}\rangle_Q,
	\end{equation*}
	as in \eqref{thm:CSX}, for any $\alpha,\beta\in F^*$, $\bar{a},\bar{b}\in\Gamma(Q)$ and $X,Y\in\Gamma(F)$.
\end{definition}

Although the condition \eqref{eq:auto-1} is unnecessary in the definition of an automorphism of a Courant algebroid, it is still important for it ensures that given $(\varphi,\Phi):\mathcal{S}\to\mathcal{S}$ an automorphism of $\mathcal{S}$, the pushforward operation $\varphi_*=\difDeRham\varphi$ maps $\Gamma(F)$ into itself. In other words $\varphi$ is a \emph{foliated} diffeomorphism of $M$ (see \cite[section 1.1.3]{MR2319199}), that is, preserves leaves of the foliation $\mathcal{F}$ of $M$ (see proposition \ref{pr:Courant-algebroid-foliation}).

\begin{definition}
    Let $\Dif_{\mathcal{F}}(M)$ denote the group of foliated diffeomorphisms of $M$ (relatively to the foliation $\mathcal{F}$ and let $\varphi\in\Dif_{\mathcal{F}}(M)$. We extend the pullback operation $\varphi^*$ to $\Omega^\bullet(\mathcal{F},Q)$ by setting $\varphi^*(\alpha\otimes\bar{a})=\left(\varphi^*\alpha\right)\otimes\bar{a}$ for any $\alpha\in\Omega^k(\mathcal{F})$ and $\bar{a}\in\Gamma(Q)$.
\end{definition}

\begin{definition}
    Let $(\varphi,\tau)\in\groO(Q)$ such that $\varphi\in\Dif_{\mathcal{F}}(M)$. We will denote by $(\varphi,\tau)^\natural:S\to S$ (or simply $\tau^\natural$) the bundle map covering $\varphi$ and defined by
    \begin{equation*}
        \tau^\natural=
	    \begin{bmatrix}
	        (\varphi^{-1})^* & 0 & 0\\
	        0 & \tau & 0\\
	        0 & 0 & \varphi_*
	    \end{bmatrix}.
    \end{equation*}
    We also set $\groO_{\mathcal{F}}^\natural(Q)=\left\{(\varphi,\tau)^\natural:\varphi\in\Dif_{\mathcal{F}}(M)\text{ and }(\varphi,\tau)\in\groO(Q)\right\}$.
\end{definition}

\begin{proposition}
    $\groO_{\mathcal{F}}^\natural(Q)$ is a subgroup of $\groO(S)$.
\end{proposition}

\begin{proof}
    First of all, the map $\tau^\natural$ maps $S$ into $S$ because $\varphi$ is foliated, and is invertible since $\varphi$ is. Thus we obtain an automorphism of the vector bundle $S=F^*\oplus Q\oplus F\to M$. It is moreover orthogonal in the sense of \eqref{eq:auto-2} since for all $\alpha,\beta\in\Gamma(F^*)$, $\bar{a},\bar{b}\in\Gamma(Q)$, $X,Y\in\Gamma(F)$, and $x\in M$, we have
    \begin{align*}
	    \big\langle\varphi^\natural(&\alpha\oplus\bar{a}\oplus X),\varphi^\natural(\beta\oplus\bar{b}\oplus Y)\big\rangle_x\\
	    &=\left\langle\left[(\varphi^{-1})^*\alpha\right]_x\big|(\varphi_*Y)_x\right\rangle+\left\langle\left[(\varphi^{-1})^*\beta\right]_x\big|(\varphi_*X)_x\right\rangle+\left\langle\tau(\bar{a})_x,\tau(\bar{b})_x\right\rangle_Q\\
	    &=\alpha_{\varphi^{-1}(x)}\difDeRham_{\varphi^{-1}(x)}\varphi^{-1}\left[\difDeRham_{\varphi^{-1}(x)}\varphi(Y_{\varphi^{-1}(x)})\right]\\
	    &\quad+\beta_{\varphi^{-1}(x)}\difDeRham_{\varphi^{-1}(x)}\varphi^{-1}\left[\difDeRham_{\varphi^{-1}(x)}\varphi(X_{\varphi^{-1}(x)})\right]+\left\langle\tau(\bar{a}_{\varphi^{-1}(x)}),\tau(\bar{b}_{\varphi^{-1}(x)})\right\rangle_Q\\
	    &=\langle\alpha\oplus\bar{a}\oplus X,\beta\oplus\bar{b}\oplus Y\rangle_{\varphi^{-1}(x)},
	\end{align*}
	where we have used the fact that $\tau$ is orthogonal for $\langle\cdot,\cdot\rangle_Q$ to obtain the last line above. Therefore $\groO_{\mathcal{F}}^\natural(Q)$ is a subset of $\groO(S)$. It is actually a subgroup of $\groO(S)$, essentially because we have the properties $(\varphi\circ\psi)_*=\varphi_*\circ\psi_*$ and $(\varphi\circ\psi)^*=\psi^*\circ\varphi^*$ (see \cite[propositions 3.6 and 12.25]{MR2954043}).
\end{proof}

\begin{proposition}\label{pr:effect-phi-tau-natural}
    Let $(\varphi,\tau)\in\Aut(\mathcal{Q})$ such that $\varphi\in\Dif_{\mathcal{F}}(M)$. For any $\alpha,\beta\in\Gamma(F^*)$, $\bar{a},\bar{b}\in\Gamma(Q)$ and $X,Y\in\Gamma(F)$, we have (see \ref{rm:compact-bracket} and the previous proposition for the notations)
    \begin{equation*}
	    \tau^\natural[\alpha\oplus\bar{a}\oplus X,\beta\oplus\bar{b}\oplus Y]_{\nabla,\,R,\,H}=\big[\tau^\natural(\alpha\oplus\bar{a}\oplus X),\tau^\natural(\beta\oplus\bar{b}\oplus Y)\big]_{\hat{\nabla},\,\hat{R},\,\hat{H}},
    \end{equation*}
    with $\hat{\nabla}$, $\hat{R}$ and $\hat{H}$ defined by
    \begin{equation*}
        \hat{\nabla}=(\varphi^{-1})^*(\tau\circ\nabla\circ\tau^{-1}),\enspace\hat{R}=(\varphi^{-1})^*(\tau\circ R)\text{\enspace and\enspace}\hat{H}=(\varphi^{-1})^*H.
    \end{equation*}
\end{proposition}

\begin{proof}
    According to \eqref{rm:compact-bracket}, we have on one side that
    \begin{align*}
	    \big[\tau^\natural(\alpha&\oplus\bar{a}\oplus X),\tau^\natural(\beta\oplus\bar{b}\oplus Y)\big]_{\hat{\nabla},\hat{R},\hat{H}}\\
	    &=\big[(\varphi^{-1})^*\alpha\oplus\tau(\bar{a})\oplus\varphi_*X,(\varphi^{-1})^*\beta\oplus\tau(\bar{b})\oplus\varphi_*Y\big]_{\hat{\nabla},\hat{R},\hat{H}}\\
	    &=\opeLieDeRham_{\varphi_*X}(\varphi^{-1})^*\beta-\opeInsDeRham_{\varphi_*Y}\difDeRham(\varphi^{-1})^*\alpha+\big\langle\hat{\nabla}\tau(\bar{a}),\tau(\bar{b})\big\rangle_Q\\
	    &\quad\quad+\big\langle\tau(\bar{a}),{\opeInsDeRham}_{\varphi_*Y}\hat{R}\big\rangle_Q-\big\langle\tau(\bar{b}),{\opeInsDeRham}_{\varphi_*X}\hat{R}\big\rangle_Q+\opeInsDeRham_{\varphi_*Y}\opeInsDeRham_{\varphi_*X}\hat{H}\\
	    &\quad\oplus[\tau(\bar{a}),\tau(\bar{b})]_Q+\hat{\nabla}_{\varphi_*X}\tau(\bar{b})-\hat{\nabla}_{\varphi_*Y}\tau(\bar{a})\\
	    &\quad\quad+\hat{R}(\varphi_*X,\varphi_*Y)\\
	    &\quad\oplus\{\varphi_*X,\varphi_*Y\},
	\end{align*}
	whereas on the other side we have
	\begin{align*}
	    \tau^\natural[\alpha&\oplus\bar{a}\oplus X,\beta\oplus\bar{b}\oplus Y\big]_{\nabla,R,H}\\
	    &=(\varphi^{-1})^*\opeLieDeRham_X\beta-(\varphi^{-1})^*\opeInsDeRham_Y\difDeRham\alpha+(\varphi^{-1})^*\langle\nabla\bar{a},\bar{b}\rangle_Q\\
	    &\quad+(\varphi^{-1})^*\langle\bar{a},{\opeInsDeRham}_{Y}R\rangle_Q-(\varphi^{-1})^*\langle\bar{b},{\opeInsDeRham}_{X}R\rangle_Q+(\varphi^{-1})^*\left(\opeInsDeRham_{Y}\opeInsDeRham_{X}H\right)\\
	    &\quad\oplus{\tau([\bar{a},\bar{b}]_Q)+\tau(\nabla_{X}\bar{b})-(\tau\nabla_{Y}\bar{a})+\tau\big(R(X,Y)\big)}\\
	    &\quad\oplus{\varphi_*\{X,Y\}}.
	\end{align*}
	First of all, we remark that these computations make sense since $\varphi$ is foliated: $\varphi_*$ maps $\Gamma(F)$ into $\Gamma(F)$ so outputs of Lie derivatives and interior products are still in $\Omega^\bullet(\mathcal{F},Q)$. Then, projecting on $\Gamma(F)$ does not yield any new relation since we already know that $\varphi_*$ preserves the Lie bracket (in other words, $\varphi_*$ is a Lie algebroid endomorphism of $\mathcal{T}_M$, see \cite[corollary 8.31]{MR2954043}). Projecting on $\Gamma(Q)$ we obtain the first two conditions: $\hat{R}=(\varphi^{-1})^*(\tau\circ R)$ and $\tau\circ\hat{\nabla}_{\varphi_*X}\circ\tau^{-1}=\nabla_X$. Finally, projecting on $\Gamma(F^*)$ and using lemma \ref{pr:interior-product-Lie-derivative} as well as the first two conditions, we do not obtain any new condition except relatively to $H$ and $\hat{H}$: for all $Z\in\Gamma(F)$ we have on one side that
	\begin{equation*}
	    (\opeInsDeRham_{\varphi_*Y}\opeInsDeRham_{\varphi_*X}\hat{H})(Z)=\hat{H}(\varphi_*X,\varphi_*Y,Z)=(\varphi^*\hat{H})(X,Y,\varphi_*^{-1}Z),
	\end{equation*} 
	and on the other side
	\begin{equation*}
	    \big[(\varphi^{-1})^*\opeInsDeRham_{Y}\opeInsDeRham_{X}H\big](Z)=(\opeInsDeRham_{Y}\opeInsDeRham_{X}H)(\varphi_*^{-1}Z)=H(X,Y,\varphi_*^{-1}Z),
	\end{equation*} 
	hence the last condition, $\varphi^*\hat{H}=H$.
\end{proof}

\begin{remark}\label{rm:phi-natural-isomorphism}
    In other words, under the assumptions stated in the proposition, one can say that $\tau^\natural$ is a Courant algebroid isomorphism between $\mathcal{S}_M[\nabla,R,H]$ and $\mathcal{S}_M[\hat{\nabla},\hat{R},\hat{H}]$ (where $\hat{\nabla}$, $\hat{R}$ and $\hat{H}$ are defined in the above proposition), covering $\varphi$. Although we did not define explicitly the notion of isomorphism between Courant algebroids covering a diffeomorphism, it is an easy adaption from definition \ref{def:automorphism}. In this particular case where both source and target Courant algebroids are standard ones, only the relation \eqref{eq:auto-3} needs to be changed into
    \begin{equation*}
        \tau^\natural\left([\alpha\oplus\bar{a}\oplus X,\beta\oplus\bar{b}\oplus Y]_{\nabla,R,H}\right)=\left[\tau^\natural(\alpha\oplus\bar{a}\oplus X),\tau^\natural(\beta\oplus\bar{b}\oplus Y)\right]_{\hat{\nabla},\hat{R},\hat{H}},
    \end{equation*}
    since anchors and inner products remain the same on both the source and the target of $\tau^\natural$.
\end{remark}

\begin{lemma}\label{pr:B-sharp-technical}
    Let $B\in\Omega^2(\mathcal{F})$ and $\varphi\in\Dif_{\mathcal{F}}(M)$. We have the formula $B^\sharp\circ\varphi_*=(\varphi^{-1})^*\big(\varphi^*B\big)^\sharp$.
\end{lemma}

\begin{proof}
    For all $X$, $Y\in\Gamma(F)$ we have successively
	\begin{align*}
	    (B^\sharp\circ\varphi_*)(X)(Y)&=B(\varphi_*X,Y)\\
	    &=(\varphi^*B)(X,\varphi_*^{-1}Y)\\
	    &=\big(\varphi^*B\big)^\sharp(X)(\varphi_*^{-1}Y)\\
	    &=(\varphi^{-1})^*\left(\big(\varphi^*B\big)^\sharp(X)\right)(Y).
	\end{align*}
\end{proof}

\begin{definition}
	Let $\varphi\in\Dif_{\mathcal{F}}(M)$. For any $A\in\Omega^1(\mathcal{F},Q)$ we have $A^\adjoint:\Gamma(Q)\to\Gamma(F^*)$ (see definition \ref{def:adjoint}). We define $\varphi^*A^\adjoint$ by setting $(\varphi^*A^\adjoint)(\bar{a})=\varphi^*\left[A^\adjoint(\bar{a})\right]$ for all $\bar{a}\in\Gamma(Q)$. Similarly, for any $B\in\Omega^2(\mathcal{F})$ we have $B^\sharp:\Gamma(F)\to\Gamma(F^*)$ (see the discussion just before \ref{pr:effect-psi-tau}). We define $\varphi^*B^\sharp$ by $(\varphi^*B^\sharp)(X)=\varphi^*\left[B^\sharp(X)\right]$ for all $X\in\Gamma(F)$.
\end{definition}

The following theorem describes explicitly the group of automorphisms of a standard Courant algebroid.

\begin{theorem}\label{thm:automorphism-group}
    Let $\mathcal{S}=\mathcal{S}_M[\nabla,R,H]$ denote a standard Courant algebroid. Let $(\varphi,\Phi)\in\Aut(\mathcal{S})$. There exists $(\varphi,\tau)\in\Aut(\mathcal{Q})$, $A\in\Omega^1(\mathcal{F},Q)$ and $B\in\Omega^2(\mathcal{F})$ satisfying the conditions
    \begin{gather}
	    \nabla-\varphi^*(\tau^{-1}\circ\nabla\circ\tau)=\ad_A,\label{eq:global-automorphism-1}\\
	    R-\varphi^*(\tau^{-1}\circ R)=\difDeRham_{\nabla} A-\frac{1}{2}[A\wedge A]_Q,\label{eq:global-automorphism-2}\\
	    H-\varphi^*H=\difDeRham B+\big\langle A\wedge R\big\rangle_Q-\frac{1}{2}\left\langle\difDeRham_{\nabla} A\wedge A\right\rangle_Q+\frac{1}{6}\big\langle A\wedge[A\wedge A]_Q\big\rangle_Q\label{eq:global-automorphism-3},
	\end{gather}
    and such that $\Phi=(\varphi,\tau)^\natural\circ\Psi_A\circ\Psi_B$, that is,
    \begin{align*}
        \Phi(\alpha\oplus\bar{a}\oplus X)=(\varphi^{-1})^*\alpha&-(\varphi^{-1})^*A^\adjoint(\bar{a})+(\varphi^{-1})^*B^\sharp(X)-\frac{1}{2}(\varphi^{-1})^*A^\adjoint(A(X))\\
        &\oplus\tau(\bar{a})+\tau(A(X))\oplus\varphi_*(X),
    \end{align*}
    for all $\alpha\in\Gamma(F^*)$, $\bar{a}\in\Gamma(Q)$ and $X\in\Gamma(F)$. We will denote $\Phi$ by $(\varphi,\tau,A,B)$. In matrix form we have
    \begin{equation*}
        \Phi=
	    \begin{bmatrix}
	        (\varphi^{-1})^* & -(\varphi^{-1})^*A^\adjoint & (\varphi^{-1})^*B^\sharp-\frac{1}{2}(\varphi^{-1})^*A^\adjoint\circ A\\
	        0 & \tau & \tau\circ A\\
	        0 & 0 & \varphi_*
	    \end{bmatrix}.
	\end{equation*}
	Moreover, the composition law of the group $\Aut(\mathcal{S})$ is given by
	\begin{align}
		(\varphi,\tau,&A,B)\circ(\psi,\sigma,C,D)\notag\\
		&=\left(\varphi\circ\psi,\tau\circ\sigma,\psi^*(\sigma^{-1}\circ A)+C,\psi^*B+D+\frac{1}{2}\big\langle\psi^*(\sigma^{-1}\circ A)\wedge C\big\rangle_Q\right),\label{eq:automorphism-composition}
	\end{align}
	the identity element is $(\opeId_M,\opeId_Q,0,0)$ and the inversion is given by
	\begin{equation}
		(\varphi,\tau,A,B)^{-1}=\left(\varphi^{-1},\tau^{-1},-(\varphi^{-1})^*(\tau\circ A),-(\varphi^{-1})^* B\right).\label{eq:automorphism-inversion}
	\end{equation}
\end{theorem}

\begin{proof}
    First of all, using a similar technique that the one used at the beginning of the proof of \ref{pr:isomorphism-standard-explicit}, and also using lemma \ref{pr:non-linear-map}, we show that $\Phi(\alpha)=(\varphi^{-1})^*\alpha$ for all $\alpha\in\Gamma(F^*)$, but what is important is only the fact that $\Phi(F^*)\subset F^*$. Using this fact, we deduce, in a similar way to the end of the proof of \ref{pr:isomorphism-standard-explicit}, that $\Phi$ both restricted and corestricted to $Q$ is an automorphism of $\mathcal{Q}$, that we will denote by $\tau$. Now define $\Psi=(\varphi^{-1},\tau^{-1})^\natural\circ\Phi$. Then $\Psi$ is an \emph{isomorphism} of standard Courant algebroids covering the identity of $M$, and such that $\Psi$ both restricted and corestricted to $Q$ is the identity map. Therefore, according to proposition \ref{pr:isomorphism-standard-explicit}, there exists $A\in\Omega^1(\mathcal{F},Q)$ and $B\in\Omega^2(\mathcal{F})$ such that $\Psi=\Psi_A\circ\Psi_B$, and the successive application of \ref{pr:effect-phi-tau-natural}, \ref{pr:effect-psi-A} and \ref{pr:effect-psi-B} yields the conditions appearing in the statement, taking into account that $\Phi$ is an \emph{automorphism} of $\mathcal{S}$, that is, $\nabla$, $R$ and $H$ must be used in both source and target standard Courant algebroids. The composition law is found multiplying two automorphisms in matrix form, and with the help of lemma \ref{pr:B-sharp-technical} to identify the right terms.
\end{proof}

\begin{remark}
    Define a group $\Aut_\mathcal{F}(\mathcal{Q})$ as the collection of elements $(\varphi,\tau)\in\Aut(\mathcal{Q})$ such that $\varphi\in\Dif_{\mathcal{F}}(M)$. Also, define another group $\Gau(\mathcal{S})=\Omega^1(\mathcal{F},Q)\times\Omega^2(\mathcal{F})$ for the (non abelian) composition law
    \[
        (A,B)\circ(C,D)=\left(A+C,B+D+\frac{1}{2}\langle A\wedge C\rangle_Q\right). 
    \]
    We call this group the \emph{gauge group} of $\mathcal{S}$. Then from the previous theorem we can say that $\Aut(\mathcal{S})$ is the subgroup of elements $(\varphi,\tau,A,B)$ of $\Aut_{\mathcal{F}}(\mathcal{Q})\ltimes_\theta\Gau(\mathcal{S})$ satisfying the relations \eqref{eq:global-automorphism-1}, \eqref{eq:global-automorphism-2} and \eqref{eq:global-automorphism-3}, where the action on the right $\theta:\Aut_{\mathcal{F}}(\mathcal{Q})\to\Aut\left[\Gau(\mathcal{S})\right]$ is given by
    \begin{equation*}
        \theta\big((\varphi,\tau)\big)(A,B)=\big(\varphi^*(\tau^{-1}\circ A),\varphi^*B\big),
    \end{equation*}
    for all $(\varphi,\tau)\in\Aut_{\mathcal{F}}(\mathcal{Q})$, $A\in\Omega^1(\mathcal{F},Q)$ and $B\in\Omega^2(\mathcal{F})$.
\end{remark}

\subsection{Infinitesimal automorphisms}\label{sec:infinitesimal-automorphism}

In this section we are interested in an infinitesimal version of the group of automorphisms of a regular Courant algebroid that we described in the previous section. To begin with, we will define the notion of an infinitesimal automorphism without referring to a dissection. Once this task is done, we will use a dissection to obtain an explicit description of these infinitesimal automorphisms.

In what follows, we will denote by $\Aut(E)$ the group of automorphisms of a vector bundle $E\to M$ over a manifold $M$.

\begin{definition}\label{def:infinitesimal-automorphism}
	Let $E\to M$ be a vector bundle over a manifold $M$, and let $(\varphi_t,\Phi_t)_{t\in\setR}$ be a \emph{smooth} one-parameter subgroup of $\Aut(E)$, that is, a group homomorphism $t\in\setR\mapsto(\varphi_t,\Phi_t)\in\Aut(E)$ such that we can define
	\begin{equation*}
		X\cdot f=\frac{\dif}{\dif t}(\varphi_t^{-1})^*f\Big|_{t=0}\text{ and }\mathcal{D}(s)=\frac{\dif}{\dif t}\Phi_t(s)\Big|_{t=0}
	\end{equation*}
 	for any $f\in\smooth(M)$ and $s\in\Gamma(E)$, thus giving rise to a vector field $X\in\mathfrak{X}(M)$ and an endomorphism $\mathcal{D}\in\End\Gamma(E)$. We may also denote by $X_\mathcal{D}$ the vector field $X$ as soon as $D$ is given explicitly, and call the pair $(X,\mathcal{D})$  the \emph{infinitesimal generator} of $(\varphi_t,\Phi_t)_{t\in\setR}$.
\end{definition}

\begin{proposition}\label{pr:infinitesimal-generator-properties}
	Let $\mathcal{E}=\left(E\to M,\anchor,[\cdot,\cdot],\langle\cdot,\cdot\rangle\right)$ be a regular Courant algebroid, and $(\varphi_t,\Phi_t)_{t\in\setR}$ be a smooth one-parameter subgroup of $\Aut(\mathcal{E})$. In particular, $(\varphi_t,\Phi_t)_{t\in\setR}$ is a one-parameter subgroup of $\Aut(E)$ for which we will denote by $(X,\mathcal{D})$ the associated infinitesimal generator. Then $X$ is $\mathcal{F}$-projectable (see \cite[chapter 1]{MR1456994}, and where $\mathcal{F}$ still denotes the natural foliation associated to $\mathcal{E}$, see proposition \ref{pr:Courant-algebroid-foliation}) and the following properties are satisfied:
	\begin{gather}
		\mathcal{D}(fu)=f\mathcal{D}(u)+(X\cdot f)u,\label{eq:auto-inf-1}\\
		X\cdot\langle u,v\rangle=\big\langle \mathcal{D}(u),v\big\rangle+\big\langle u,\mathcal{D}(v)\big\rangle,\label{eq:auto-inf-3}\\
		\mathcal{D}\big([u,v]\big)=\big[\mathcal{D}(u),v\big]+\big[u,\mathcal{D}(v)\big],\label{eq:auto-inf-2}
	\end{gather}
	for any $f\in\smooth(M)$, and $u$, $v\in\Gamma(E)$.
\end{proposition}

\begin{proof}
    Each diffeomorphism $\varphi_t$ is foliated, so for any $V\in\Gamma(F)$ we have
    \begin{equation*}
		[X,V]=\opeLieDeRham_X V=\frac{\dif}{\dif t}(\varphi_{-t}^{-1})_*(V)\Big|_{t=0}=\frac{\dif}{\dif t}(\varphi_{t})_*(V)\Big|_{t=0}\in\Gamma(F),
	\end{equation*}
	that is, $X$ is $\mathcal{F}$-projectable. Now according to lemma \ref{pr:non-linear-map} we have
	\begin{align*}
		\mathcal{D}(fu)&=\frac{\dif}{\dif t}\Phi(fu)\Big|_{t=0}\\
		&=\frac{\dif}{\dif t}\left((\varphi_t^{-1})^*f\Phi_t(u)\right)\bigg|_{t=0}\\
		&=\left[\left(\frac{\dif}{\dif t}(\varphi_t^{-1})^*f\right)\Phi_t(u)\right]\bigg|_{t=0}+\left[(\varphi_t^{-1})^*f\left(\frac{\dif}{\dif t}\Phi_t(u)\right)\right]\bigg|_{t=0}\\
		&=(X\cdot f)u+\mathcal{D}(u),
	\end{align*}
	which gives the first property. The second property comes from the $\setR$-bilinearity of the inner product $\langle\cdot,\cdot\rangle$ \eqref{eq:auto-2}, since by definition, we have on one side that
	\begin{equation*}
	    \frac{\dif}{\dif t}(\varphi_t^{-1})^*\langle u,v\rangle\Big|_{t=0}=X\cdot\langle u,v\rangle,
	\end{equation*}
	and on the other side that
	\begin{align*}
		\frac{\dif}{\dif t}\big\langle\Phi_t(u),\Phi_t(v)\big\rangle\Big|_{t=0}&=\left\langle\frac{\dif}{\dif t}\Phi_t(u)\Big|_{t=0},v\right\rangle+\left\langle u,\frac{\dif}{\dif t}\Phi_t(v)\Big|_{t=0}\right\rangle\\
		&=\big\langle \mathcal{D}(u),v\big\rangle+\big\langle u,\mathcal{D}(v)\big\rangle.
	\end{align*}
	The third property follows from the $\setR$-bilinearity of the bracket $[\cdot,\cdot]$ and \eqref{eq:auto-3} since by definition, we have on one side that
	\begin{equation*}
	    \frac{\dif}{\dif t}\Phi_t\big([u,v]\big)\Big|_{t=0}=\mathcal{D}\big([u,v]\big),
	\end{equation*}
	and on the other side that
	\begin{align*}
		\frac{\dif}{\dif t}\Phi_t\big([u,v]\big)\Big|_{t=0}&=\frac{\dif}{\dif t}\big[\Phi_t(u),\Phi_t(v)\big]\Big|_{t=0}\\
		&=\left[\frac{\dif}{\dif t}\Phi_t(u)\Big|_{t=0},v\right]+\left[u,\frac{\dif}{\dif t}\Phi_t(v)\Big|_{t=0}\right]\\
		&=\big[\mathcal{D}(u),v\big]+\big[u,\mathcal{D}(v)\big].
	\end{align*}
\end{proof}

\begin{definition}{\cite[chapter 1]{MR1456994}}
    Let $M$ be a manifold and $\mathcal{F}$ be a foliation of $M$. We will denote by $\mathfrak{X}_{\mathcal{F}}(M)$ the Lie algebra of $\mathcal{F}$-projectable vector fields of $M$, that is, vector fields $X\in\mathfrak{X}(M)$ such that for any $V\in\Gamma(F)$, $[X,V]\in\Gamma(F)$, where $F\to M$ is the vector bundle associated to foliation $\mathcal{F}$.
\end{definition}

\begin{definition}
    Let $\mathcal{E}=\left(E\to M,\anchor,[\cdot,\cdot],\langle\cdot,\cdot\rangle\right)$ be a regular Courant algebroid. Any pair $(X,\mathcal{D})\in\mathfrak{X}_\mathcal{F}(M)\times\End\Gamma(E)$ such that relations \eqref{eq:auto-inf-1}, \eqref{eq:auto-inf-2} and \eqref{eq:auto-inf-3} hold is called an \emph{infinitesimal automorphism} of $\mathcal{E}$. We will denote by $\aut(\mathcal{E})$ the set of all infinitesimal automorphisms of $\mathcal{E}$.
\end{definition}

\begin{proposition}
	Let $\mathcal{E}=\left(E\to M,\anchor,[\cdot,\cdot],\langle\cdot,\cdot\rangle\right)$ be a regular Courant algebroid. The set $\aut(\mathcal{E})$ is a Lie algebra for the bracket $\llbracket\cdot,\cdot\rrbracket$ defined by 
	\begin{equation*}
	    \big\llbracket(X_1,\mathcal{D}_1),(X_2,\mathcal{D}_2)\big\rrbracket=\big(\{X_1,X_2\},[\mathcal{D}_1,\mathcal{D}_2]\big),
	\end{equation*}
	with $\{\cdot,\cdot\}$ denoting the Lie bracket of vector fields and $[\cdot,\cdot]$ denoting the commutator on $\End\Gamma(E)$. In other words, $X_{[D_1,D_2]}=\{X_1,X_2\}$ for all $(X_1,D_1)$, $(X_2,D_2)\in\aut(\mathcal{E})$.
\end{proposition}

\begin{proof}
    We have to check that the bracket is well-defined. Using the Jacobi identity for the Lie bracket of vector fields, we obtain that $\{X_1,X_2\}$ is again a $\mathcal{F}$-projectable vector field. Then \eqref{eq:auto-inf-2} tells that $\mathcal{D}_1$ and $\mathcal{D}_2$ are derivations of $(\Gamma(E),[\cdot,\cdot])$ so their commutator is a derivation again (see \cite[section 5.6]{MR0369382}) and \eqref{eq:auto-inf-2} holds for $[\mathcal{D}_1,\mathcal{D}_2]$. Concerning \eqref{eq:auto-inf-1}, we have for all $f\in\smooth(M)$ and $u\in\Gamma(E)$ that
    \begin{align*}
		(\mathcal{D}_1\circ \mathcal{D}_2&-\mathcal{D}_2\circ \mathcal{D}_1)(fu)\\
		&=\mathcal{D}_1\big((X_2\cdot f)u+f\mathcal{D}_2(u)\big)-\mathcal{D}_2\big((X_1\cdot f)u+f\mathcal{D}_1(u)\big)\\
		&=(X\cdot X_2\cdot f)u+(X_2\cdot f)\mathcal{D}_1(u)+f\mathcal{D}_1\circ \mathcal{D}_2(u)+(X_1\cdot f)\mathcal{D}_2(u)\\
		&\quad-f\mathcal{D}_2\circ \mathcal{D}_1(u)-(X_2\cdot f)\mathcal{D}_1(u)-(X_2\cdot X_1\cdot f)u-(X_1\cdot f)\mathcal{D}_2(u)\\
		&=(\{X_1,X_2\}\cdot f)u+f[\mathcal{D}_1,\mathcal{D}_2](u).
	\end{align*}
	Now concerning \eqref{eq:auto-inf-3}, we have for all $u$, $v\in\Gamma(E)$ that
	\begin{align*}
		\{&X_1,X_2\}\cdot\langle u,v\rangle\\
		&=X_1\cdot X_2\cdot\langle u,v\rangle-X_2\cdot X_1\cdot\langle u,v\rangle\\
		&=X_1\cdot\Big(\big\langle \mathcal{D}_2(u),v\big\rangle+\big\langle u,\mathcal{D}_2(v)\big\rangle\Big)-X_2\cdot\Big(\big\langle \mathcal{D}_1(u),v\rangle+\langle u,\mathcal{D}_1(v)\big\rangle\Big)\\
		&=\big\langle \mathcal{D}_1\circ \mathcal{D}_2(u),v\big\rangle+\big\langle\mathcal{D}_2(u),\mathcal{D}_1(v)\big\rangle+\big\langle \mathcal{D}_1(u),\mathcal{D}_2(v)\big\rangle+\big\langle u,\mathcal{D}_1\circ\mathcal{D}_2(v)\big\rangle\\
		&\quad\quad-\big\langle\mathcal{D}_2\circ \mathcal{D}_1(u),v\big\rangle-\big\langle\mathcal{D}_1(u),\mathcal{D}_2(v)\big\rangle-\big\langle\mathcal{D}_2(u),\mathcal{D}_1(v)\big\rangle-\langle u,\mathcal{D}_2\circ\mathcal{D}_1(v)\big\rangle\\
		&=\big\langle(\mathcal{D}_1\circ\mathcal{D}_2-\mathcal{D}_2\circ\mathcal{D}_1)(u),v\big\rangle+\big\langle u,(\mathcal{D}_1\circ \mathcal{D}_2-\mathcal{D}_2\circ \mathcal{D}_1)(v)\big\rangle\\
		&=\big\langle[\mathcal{D}_1,\mathcal{D}_2](u),v\big\rangle+\big\langle u,[\mathcal{D}_1,\mathcal{D}_2](v)\big\rangle.
	\end{align*}
	So $\big(\{X_1,X_2\},[\mathcal{D}_1,\mathcal{D}_2]\big)$ is an infinitesimal automorphism of $\mathcal{E}$, and $X_{[\mathcal{D}_1,\mathcal{D}_2]}=\{X_1,X_2\}$. Finally, it is clear that $\llbracket\cdot,\cdot\rrbracket$ defines a Lie bracket as its two components are both skew-symmetric and satisfy the Jacobi identity.
\end{proof}

From now on we will consider a regular Courant algebroid $\mathcal{E}=(E\to M,\anchor,[\cdot,\cdot],\langle\cdot,\cdot\rangle)$ and the associated notations. Let $(\Delta,\nabla,R,H)$ denote a dissection of $\mathcal{E}$. The next proposition states that having chosen a dissection of $\mathcal{E}$, to study the Lie algebra of infinitesimal automorphisms of $\mathcal{E}$ is equivalent to study the Lie algebra of infinitesimal automorphisms of the standard Courant algebroid $\mathcal{S}=\mathcal{S}_M[\nabla,R,H]$ obtained from the dissection (see theorem \ref{thm:CSX}).

\begin{proposition}
	There is a Lie algebra isomorphism $\aut(\mathcal{E})\cong\aut(\mathcal{S})$.
\end{proposition}

\begin{proof}
	The Lie algebra isomorphism $\aut(\mathcal{E})\to\aut(\mathcal{S})$ is defined by conjugation: $(X,\mathcal{D})\mapsto(X,\Delta^{-1}\circ \mathcal{D}\circ\Delta)$. Note that $(X,\Delta^{-1}\circ \mathcal{D}\circ\Delta)$ is actually an infinitesimal automorphism of $\mathcal{S}=\mathcal{S}_M[\nabla,R,H]$ because $\Delta$ is by definition an isomorphism of Courant algebroids covering the identity of $M$ (see definition \ref{def:Courant-algebroid-morphism}).
\end{proof}

\begin{theorem}\label{thm:infinitesimal-automorphism-algebra}
    Let $\mathcal{S}=\mathcal{S}_M[\nabla,R,H]$ denote a standard Courant algebroid. Let $(X,\mathcal{D})\in\aut(\mathcal{S})$. Then there exists $\Theta\in\End\Gamma(Q)$, $a\in\Omega^1(\mathcal{F},Q)$ and $b\in\Omega^2(\mathcal{F})$ satisfying the conditions
    \begin{gather}
        \Theta(f\bar{a})=f\Theta(\bar{a})+(X\cdot f)\bar{a},\label{eq:infinitesimal-automorphism-1}\\
        X\cdot\langle\bar{a},\bar{b}\rangle_Q=\big\langle\Theta(\bar{a}),\bar{b}\big\rangle_Q+\big\langle\bar{a},\Theta(\bar{b})\big\rangle_Q,\label{eq:infinitesimal-automorphism-2}\\
		\Theta\big([\bar{a},\bar{b}]_Q\big)=\big[\Theta(\bar{a}),\bar{b}\big]_Q+\big[\bar{a},\Theta(\bar{b})\big]_Q,\label{eq:infinitesimal-automorphism-3}\\
		[\Theta,\nabla_U]-\nabla_{\{X,U\}}=\ad_a(U),\label{eq:infinitesimal-automorphism-4}\\ 
		{\opeLieDeRham}_{X,\Theta}\,R=\difDeRham_\nabla a,\label{eq:infinitesimal-automorphism-5}\\
		\opeLieDeRham_X H+\difDeRham b+\langle a\wedge R\rangle_Q=0,\label{eq:infinitesimal-automorphism-6}
	\end{gather}
	for all $f\in\smooth(M)$ and $\bar{a}$, $\bar{b}\in\Gamma(Q)$; and such that $\mathcal{D}$ acts on $\Gamma(F^*\oplus Q\oplus F)$ as
	\begin{equation}
	    \mathcal{D}(\xi\oplus\bar{x}\oplus U)={\opeLieDeRham_X\xi-a^\adjoint(\bar{x})+\opeInsDeRham_U b}\oplus{\Theta(\bar{x})+a(U)}\oplus{\opeLieDeRham_X U}.\label{eq:action-infinitesimal-automorphism}
	\end{equation}
	We will denote $(X,\mathcal{D})$ by $(X,\Theta,a,b)$. Moreover the Lie bracket on $\aut(\mathcal{S})$ is given by
	\begin{align*}
		\big\llbracket(X,\Theta,a,b),&(Y,\Sigma,c,d)\big\rrbracket\\
		&=\big(\{X,Y\},[\Theta,\Sigma],{\opeLieDeRham}_{X,\Theta}\,c-{\opeLieDeRham}_{Y,\Sigma}\,a,\opeLieDeRham_{X}d-\opeLieDeRham_{Y}b+\langle a\wedge c\rangle_Q\big),
	\end{align*}
	where for any $X\in\Gamma(F)$ and $\Theta\in\End\Gamma(Q)$ the operation ${\opeLieDeRham}_{X,\Theta}:\Omega^k(\mathcal{F},Q)\to\Omega^k(\mathcal{F},Q)$ is defined for any $\omega\in\Omega^k(\mathcal{F},Q)$ and $X_1,\dots,X_k\in\Gamma(F)$ by
	\begin{equation*}
	    ({\opeLieDeRham}_{X,\Theta}\,\omega)(X_1,\dots,X_k)=\Theta\left[\omega(X_1,\dots,X_k)\right]-\sum_{i=1}^k\omega\left(X_1,\dots,\{X,X_i\},\dots,X_k\right).
	\end{equation*}
\end{theorem}

\begin{proof}
    The proof is similar to the proof of \ref{pr:isomorphism-standard-explicit}. Let $(X,\mathcal{D})\in\aut(\mathcal{S})$. By definition, $(X,\mathcal{D})$ satisfies \eqref{eq:auto-inf-2} so for any $u$, $v\in\Gamma(S)$ we have $\mathcal{D}\left([fu,v]\right)=\left[\mathcal{D}(fu),v\right]+\left[fu,\mathcal{D}(v)\right]$. After expanding on both sides, it remains:
    \begin{equation*}
        -\big(\{X,\anchor(v)\}\cdot f\big)u+\langle u,v\rangle\mathcal{D}(\opeD f)=\langle u,v\rangle\opeD(X\cdot f)-\big(\anchor(\mathcal{D}(v))\cdot f\big)u.
    \end{equation*}
    Taking both $u$ and $v\in\Gamma(F)$, we obtain $\anchor(\mathcal{D}(v))=\opeLieDeRham_Xv$. Now taking just $v\in\Gamma(F)$, we obtain $\mathcal{D}(\opeD f)=\opeD(X\cdot f)=\opeD(\opeLieDeRham_X f)=\opeLieDeRham_X(\opeD f)$, and since by \eqref{pr:orthogonal-ker-anchor} $(\Ker\anchor)^\bot\cong F^*$ is generated by $\Ima\opeD$, we obtain $\mathcal{D}(\xi)=\opeLie_X\xi$ for any $\xi\in\Gamma(F^*)$. Therefore, in matrix form we have
    \begin{equation*}
        \mathcal{D}=
        \begin{bmatrix}
            \opeLieDeRham_X & \gamma & \beta\\
            0 & \Theta & a\\
            0 & Z & \opeLieDeRham_X
        \end{bmatrix},
    \end{equation*}
    for some maps $\gamma:\Gamma(Q)\to\Gamma(F^*)$, $\Theta:\Gamma(Q)\to\Gamma(Q)$, $Z:\Gamma(Q)\to\Gamma(F)$, $\beta:\Gamma(F)\to\Gamma(F^*)$ and $a:\Gamma(F)\to\Gamma(Q)$. Since $(X,\mathcal{D})$ satisfies \eqref{eq:auto-inf-3}, we have for $u=\bar{x}\in\Gamma(Q)$ and $v=\xi\in\Gamma(F^*)$ that $\langle\mathcal{D}(\bar{x}),\xi\rangle+\langle\bar{x},\mathcal{D}(\xi)\rangle=0$, which yields $\xi(Z(\bar{x}))=0$ so $Z\equiv 0$. Now, restricting \eqref{eq:auto-inf-1}, \eqref{eq:auto-inf-2} and \eqref{eq:auto-inf-3} to $\Gamma(Q)$, we get the relations \eqref{eq:infinitesimal-automorphism-1}, \eqref{eq:infinitesimal-automorphism-2} and \eqref{eq:infinitesimal-automorphism-3} relatively to $\Theta$ (one could say that $(X,\Theta)$ is an infinitesimal automorphism of the quadratic Lie algebroid $\mathcal{Q}$). Next, using \eqref{eq:auto-inf-3} restricted to $\Gamma(F)$, we obtain that $\beta$ comes from a tangential $2$-form, $\beta=B^\sharp$ for some $B\in\Omega^2(\mathcal{F})$; and using \eqref{eq:auto-inf-3} again for $u=\bar{a}\in\Gamma(Q)$ and $v\in\Gamma(F)$ we obtain that $\gamma=-a^\adjoint$. Finally, we obtain conditions \eqref{eq:infinitesimal-automorphism-4}, \eqref{eq:infinitesimal-automorphism-5} and \eqref{eq:infinitesimal-automorphism-6} from \eqref{eq:auto-inf-2} after a long but straightforward computation, using lemma \ref{pr:interior-product-Lie-derivative} and the relations given in theorem \ref{thm:Cartan-triple}.
    We now turn to the Lie bracket of $\aut(\mathcal{S})$. By definition, for all $\xi\in\Gamma(F^*)$, $\bar{x}\in\Gamma(Q)$ and $U\in\Gamma(F)$ we have
	\begin{align*}
		\big\llbracket(&X,\Theta,a,b),(Y,\Sigma,c,d)\big\rrbracket(\xi\oplus\bar{x}\oplus U)\\
		&=(X,a,b,\Theta)\circ(Y,c,d,\Sigma)(\xi\oplus\bar{x}\oplus U)-(Y,c,d,\Sigma)\circ(X,a,b,\Theta)(\xi\oplus\bar{x}\oplus U),
	\end{align*}
	and after expanding the right hand side equals
	\begin{align*}
		&\opeLieDeRham_{\{X,Y\}}\xi-\opeLieDeRham_{X}c^\adjoint(\bar{x})+\opeLieDeRham_{Y}a^\adjoint(\bar{x})+\opeLieDeRham_{X}\opeInsDeRham_U d-\opeLieDeRham_{Y}\opeInsDeRham_U b\\
		&\quad-a^\adjoint\big(\Sigma(\bar{x})\big)-a^\adjoint\big(c(U)\big)+c^\adjoint\big(\Theta(\bar{x})\big)+c^\adjoint\big(a(U)\big)+\opeInsDeRham_{\{Y,U\}}b-\opeInsDeRham_{\{X,U\}}d\\
        &\quad\quad\quad\oplus{[\Theta,\Sigma](\bar{x})+\Theta\big(c(U)\big)-\Sigma\big(a(U)\big)+{\opeInsDeRham}_{\{Y,U\}}a-{\opeInsDeRham}_{\{X,U\}}c}\\
        &\quad\quad\quad\oplus{\opeLieDeRham_{\{X,Y\}}U}.
	\end{align*}
    The terms
    \begin{equation*}
        \opeLieDeRham_{X}\opeInsDeRham_U d-\opeLieDeRham_{Y}\opeInsDeRham_U b-a^\adjoint\big(c(U)\big)+c^\adjoint\big(a(U)\big)+\opeInsDeRham_{\{Y,U\}}b-\opeInsDeRham_{\{X,U\}}d
    \end{equation*}
    correspond to the $\Omega^2(\mathcal{F})$ component of the bracket we want to compute, it reads $\opeLieDeRham_{X}d-\opeLieDeRham_{Y} b+\langle a\wedge c\rangle_Q$. The other terms are easier to identify, for instance the $\Omega^1(\mathcal{F},Q)$ component of the bracket comes from the terms
    \begin{equation*}
        \Theta\big(c(U)\big)-\Sigma\big(a(U)\big)+{\opeInsDeRham}_{\{Y,U\}}a-{\opeInsDeRham}_{\{X,U\}}c=(\opeLieDeRham_{X,\Theta}\,c)(U)-(\opeLieDeRham_{Y,\Sigma}\,a)(U),
    \end{equation*}
    so ${\opeLieDeRham}_{X,\Theta}\,c-{\opeLieDeRham}_{Y,\Sigma}\,a$ is the $\Omega^1(\mathcal{F},Q)$ component of the bracket.
\end{proof}

\begin{remark}
    Define a Lie algebra $\aut_\mathcal{F}(\mathcal{Q})$ as the collection of pairs $(X,\Theta)\in\aut(\mathcal{Q})$ (that is, infinitesimal generators $(X,\Theta)$ of $Q\to M$ satisfying \eqref{eq:infinitesimal-automorphism-1}, \eqref{eq:infinitesimal-automorphism-2} and \eqref{eq:infinitesimal-automorphism-3}) such that $X\in\mathfrak{X}_\mathcal{F}(M)$. Define also another Lie algebra $\gau(\mathcal{S})=\Omega^1(\mathcal{F},Q)\times\Omega^2(\mathcal{F})$ with bracket given by
    \begin{equation*}
        \big\llbracket(a,b),(c,d)\big\rrbracket=(0,\langle a\wedge c\rangle_Q).
    \end{equation*}
    Then from the previous theorem we can say that $\aut(\mathcal{S})$ is the sub-Lie algebra of elements $(X,\Theta,a,b)$ of $\aut_\mathcal{F}(\mathcal{Q})\ltimes_\vartheta\gau(\mathcal{S})$ (see \cite[section 1.10]{MR1743970} for the definition of the semidirect sum of Lie algebras) \eqref{eq:infinitesimal-automorphism-4}, \eqref{eq:infinitesimal-automorphism-5} and \eqref{eq:infinitesimal-automorphism-6}, where the (infinitesimal) action on the right $\vartheta:\aut_\mathcal{F}(\mathcal{Q})\to\Der\left[\gau(\mathcal{S})\right]$ is given by
    \begin{equation*}
        \vartheta\big((X,\Theta)\big)(a,b)=({\opeLieDeRham}_{X,\Theta}\,a,\opeLieDeRham_X b),
    \end{equation*}
    for all $(X,\Theta)\in\aut_\mathcal{F}(\mathcal{Q})$, $a\in\Omega^1(\mathcal{F},Q)$ and $b\in\Omega^2(\mathcal{F})$.
\end{remark}

\subsection{Examples}\label{sec:automorphism-examples}

In this last section we detail three examples to illustrate the results obtained in the two previous sections. For an approach based on dg-manifolds, see \cite{Actions} and \cite{Duality}.

\begin{example}[Courant algebroids of type $D_n$]
    Let $\mathcal{E}=(E\to M,\anchor,[\cdot,\cdot],\langle\cdot,\cdot\rangle)$ be an exact Courant algebroid (see definition \ref{def:exact-Courant-algebroid}). Such a Courant algebroid is also known as a Courant algebroid of type $D_n$, with $n=\dim M$ (\cite[section 2]{MR3090107}). These Courant algebroids are essential to generalized complex geometry as well as many models found in theoretical Physics (see for instance \cite{MR2811595}). 

     In this case, the Lie algebroid $\mathcal{F}$ (proposition \ref{pr:Courant-algebroid-foliation}) is $\mathcal{T}_M$ (see \ref{ex:canonical}) and the quadratic Lie algebra bundle $\mathcal{Q}$ is the null one. A dissection corresponds to an isotropic splitting $\varepsilon:TM\to E$ of the short exact sequence of vector bundles
    \begin{equation*}
	    0\longrightarrow T^* M\overset{\anchor^*}{\longrightarrow}E^*\cong E\overset{\anchor}{\longrightarrow}TM\longrightarrow 0,
    \end{equation*}
    and gives a vector bundle isomorphism $\Delta:TM\oplus T^* M\to E$ as well as a $3$-form $H$ on $M$, then a Courant algebroid isomorphism $\Delta:\mathcal{E}_M[H]\to\mathcal{E}$. A change of splitting corresponds to the transformation $H\mapsto H-\difDeRham B$ for some $B\in\Omega^2(M)$, as stated in theorem \ref{thm:dissection-change}. With respect to this dissection we have $\nabla=0$ and $R=0$ (this comes from the proof of theorem \ref{thm:CSX}).
    
    The group $\Dif_\mathcal{F}(M)$ is actually the group of diffeomorphisms $\Dif(M)$ of $M$ and the group $\Gau(\mathcal{E}_M[H])$ is the abelian group $\Omega^2(M)$ (for the addition of forms). The group of automorphisms of $\mathcal{E}_M[H]$ consists of pairs $(\varphi,B)\in\Dif(M)\ltimes_\theta\Omega^2(M)$ satisfying the condition $H-\varphi^*H=\difDeRham B$. Such an automorphism $(\varphi,B)$ acts on $X\oplus\xi\in\mathfrak{X}(M)\oplus\Omega^1(M)$ as
\begin{equation*}
	(\varphi,B)(X\oplus\xi)={\varphi_*(X)}\oplus{(\varphi^{-1})^*\big(\xi+\opeInsDeRham_X B\big)},
\end{equation*}
and the composition law of the group is
\begin{equation*}
    (\varphi,B)\circ(\psi,D)=(\varphi\circ\psi,\psi^*B+D),
\end{equation*}
for all $\varphi$, $\psi\in\Dif(M)$ and $B$, $D\in\Omega^2(M)$. This result has been obtained for the first time in \cite{MR2811595}. We can also consider $\Aut(\mathcal{E}_M[H])$ as the group extension
\begin{equation*}
	0\longrightarrow\Omega^2(M)\longrightarrow\Aut(\mathcal{E}_M[H])\longrightarrow\Dif_{[H]}(M)\longrightarrow 1,
\end{equation*}
for the natural injection and surjection, where $\Dif_{[H]}(M)$ denotes the group of automorphisms $\varphi$ of $M$ that preserve the cohomology class $[H]\in\mathbf{H}^3(M)$, that is $\varphi^*[H]-[H]=0$ in $\mathbf{H}^3(M)$.

    On the infinitesimal side, $\aut(\mathcal{E}_M[H])$ consists of pairs $(X,b)\in\mathfrak{X}(M)\ltimes_\vartheta\Omega^2(M)$ satisfying the condition $\opeLieDeRham_X H+\difDeRham b=0$. Such an infinitesimal automorphism $(X,b)$ acts on $U\oplus\xi\in\mathfrak{X}(M)\oplus\Omega^1(M)$ as
\begin{equation*}
    (X,b)(U\oplus\xi)={\{X,U\}}\oplus{\opeLieDeRham_X\xi+\opeInsDeRham_U b},
\end{equation*}
and the Lie bracket is given by
\begin{equation*}
    \big\llbracket(X,b),(Y,c)\big\rrbracket=\big(\{X,Y\},\opeLieDeRham_X c-\opeLieDeRham_{Y} b\big),
\end{equation*}
for all $X$, $Y\in\mathfrak{X}(M)$ and $b$, $c\in\Omega^2(M)$. This result has been obtained for the first time in \cite{MR2534281} and \cite[section 4]{MR2479266}. We can also consider $\aut(\mathcal{E}_M[H])$ as a Lie algebra extension
\begin{equation*}
    0\longrightarrow\Omega^2(M)\longrightarrow\aut(\mathcal{E}_M[H])\longrightarrow\mathfrak{X}_{[H]}(M)\longrightarrow 0,
\end{equation*}
for the natural injection and surjection, where $\mathfrak{X}_{[H]}(M)$ denotes the Lie algebra of vector fields $X$ on $M$ that preserve the cohomology class $[H]\in\mathbf{H}^3(M)$, that is $\opeLieDeRham_X[H]=0$ in $\mathbf{H}^3(M)$.
\end{example}

\begin{example}[Courant algebroids of type $B_n$]
    Let $M$ be a manifold. We are interested in a particular standard Courant algebroid on $M$ (see \ref{pr:standard-Courant-algebroid}). Consider the vector bundle  $T^*M\oplus Q\oplus TM\to M$, with $Q=M\times\setR\to M$ the trivial vector bundle of rank $1$, equipped with the null bracket and with inner product $\langle\lambda,\mu\rangle_Q=\lambda\mu$ for all $\lambda$, $\mu\in\Gamma(Q)\cong\smooth(M)$. Moreover, we consider for the connection $\nabla$ on $Q\to M$ the trivial one (see \ref{ex:trivial-representation}), set $R=0$ and let $H\in\mathbf{H}^3(M)$ be a $\difDeRham$-closed $3$-form on $M$. We will denote by $\mathcal{R}_M[H]$ the standard Courant algebroid associated to these data and called it a \emph{Courant algebroid of type $B_n$}, with $n=\dim M$. According to \ref{rm:compact-bracket} its bracket reads
    \begin{align*}
        [\alpha\oplus\lambda\oplus X,\beta\oplus&\mu\oplus Y]\\
        &={\opeLieDeRham_X\beta-\opeInsDeRham_Y\difDeRham\alpha+\mu\difDeRham\lambda+\opeInsDeRham_Y\opeInsDeRham_X H}\oplus{\opeInsDeRham_X\difDeRham\mu-\opeInsDeRham_Y\difDeRham\lambda}\oplus{\{X,Y\}},
    \end{align*} 
    for all $\alpha$, $\beta\in\Omega^1(M)$, $\lambda$, $\mu\in\smooth(M)$ and $X$, $Y\in\mathfrak{X}(M)$. This Courant algebroid appeared for the first time in \cite{MR2901838} and \cite[section 2]{MR3090107} (with $H=0$ though), where it was shown they are useful for the description of a certain geometric structure on a orientable $3$-manifold.  
    
    In this case the Lie algebroid $\mathcal{F}$ is $\mathcal{T}_ M$ and $\mathcal{Q}$ is the quadratic Lie algebra bundle $(Q\to M,[\cdot,\cdot]_Q\equiv 0,\langle\cdot,\cdot\rangle_Q)$. We compute that $\Aut(\mathcal{Q})=\groO(Q)\cong\setZ_2$. Therefore, automorphisms of $\mathcal{R}_M[H]$ are elements $(\varphi,\varepsilon,A,B)$ (with $\varepsilon=\pm 1$) of $\Dif(M)\ltimes_\theta(\setZ_2\times\Omega^1(M)\times\Omega^2(M))$ such that $\difDeRham A=0$ and $H-\varphi^*H=\difDeRham B$ (after computing that on sections of $Q\to M$, elements of $\groO(Q)$ act as $\pm(\varphi^{-1})^*$, the condition \eqref{eq:global-automorphism-1} is automatically satisfied thanks to lemma \ref{pr:interior-product-Lie-derivative}). Such automorphisms act on $\Omega^1(M)\oplus\smooth(M)\oplus\mathfrak{X}(M)$ as
    \begin{equation*}
 	    (\varphi,\varepsilon,A,B)(\xi\oplus\lambda\oplus X)={(\varphi^{-1})^*\big(\xi-\lambda A+\opeInsDeRham_X B\big)}\oplus{\varepsilon\varphi^*(\lambda+A(X))}\oplus{\varphi_*(X)},
 	\end{equation*}
    and the composition law is given by
    \begin{equation*}
 	    (\varphi,\varepsilon,A,B)\circ(\psi,\varsigma,C,D)=\left(\varphi\circ\psi,\varepsilon\varsigma,\varsigma\psi^*A+C,\psi^*B+C+\frac{\varsigma}{2}\psi^*A\wedge C\right),
	\end{equation*}
	for all $\varphi$, $\psi\in\Dif(M)$, $\varepsilon$, $\varsigma\in\setZ_2$, $A$, $C\in\Omega^1(M)$ and $B$, $D\in\Omega^2(M)$. We can also consider $\Aut(\mathcal{R}_M[H])$ as the group extension
	\begin{equation*}
	    0\longrightarrow\setZ_2\times Z^1(M)\times\Omega^2(M)\longrightarrow\Aut(\mathcal{R}_M[H])\longrightarrow\Dif_{[H]}(M)\longrightarrow 1,
	\end{equation*}
	for the natural injection and surjection, where $Z^1(M)$ denotes $\difDeRham$-closed $1$-forms on $M$. Therefore, we recover the result of \cite[proposition 2.2]{MR3090107} as soon as $H=0$. However, note that in \cite{MR3090107}, we always have $\varepsilon=+1$ because $\groSO(Q)$ is considered instead of $\groO(Q)$.
	
	On the infinitesimal side, conditions \eqref{eq:infinitesimal-automorphism-1} and \eqref{eq:infinitesimal-automorphism-2} are equivalent and show that $\Theta=X$ as a derivation of $\smooth(M)$, so $\aut_\mathcal{F}(\mathcal{Q})$ is reduced to $\mathfrak{X}(M)$ and $\opeLieDeRham_{X,\Theta}=\opeLieDeRham_X:\Omega^k(M)\to\Omega^k(M)$. Then $\aut(\mathcal{R}_M[H])$ is the sub-Lie algebra of $\mathfrak{X}(M)\ltimes_\vartheta(\Omega^1(M)\times\Omega^2(M))$ of elements $(X,a,b)$ satisfying the relations $\difDeRham a=0$ and $\opeLieDeRham_X H+\difDeRham b=0$ (the condition \eqref{eq:infinitesimal-automorphism-4} is automatically satisfied thanks to the Jacobi identity for vector fields). Such infinitesimal automorphisms act on $\Omega^1(M)\oplus\smooth(M)\oplus\mathfrak{X}(M)$ as
	\begin{equation*}
	    (X,a,b)(\xi\oplus\lambda\oplus U)={\opeLieDeRham_X\xi-\lambda a+\opeInsDeRham_U b}\oplus{\opeLieDeRham_X\lambda+\opeInsDeRham_U a}\oplus{\{X,U\}},
	\end{equation*}
	and the Lie bracket is given by
	\begin{equation*}
	    \big\llbracket(X,a,b),(Y,c,d)\big\rrbracket=\big(\{X,Y\},\opeLieDeRham_X c-\opeLieDeRham_{Y}a,\opeLieDeRham_X d-\opeLieDeRham_{Y}b+a\wedge c\big),
	\end{equation*}
	for all $X$, $Y\in\mathfrak{X}(M)$, $a$, $c\in\Omega^1(M)$ and $b$, $d\in\Omega^2(M)$. We can also consider $\aut(\mathcal{R}_M[H])$ as the Lie algebra extension
	\begin{equation*}
	    0\longrightarrow Z^1(M)\times\Omega^2(M)\longrightarrow\aut(\mathcal{R}_M[H])\longrightarrow\mathfrak{X}_{[H]}(M)\longrightarrow 0,
	\end{equation*}
	for the natural injection and surjection. Therefore, we recover the result of \cite[section 2]{MR3090107} as soon as $H=0$.
\end{example}

\begin{example}[Heterotic Courant algebroids]
    Let $\mathcal{H}=(E\to M,\anchor,[\cdot,\cdot],\langle\cdot,\cdot\rangle)$ be a \emph{transitive} Courant algebroid, that is the anchor $\anchor$ is assumed to be surjective. We will say that $\mathcal{H}$ is \emph{heterotic} if its associated Lie algebroid $\overline{\mathcal{H}}=(E/(\Ker\anchor)^\bot\to M,\anchor,[\cdot,\cdot])$ is isomorphic to the Atiyah Lie algebroid of some $G$-principal bundle (see example \ref{ex:Atiyah}). Such Courant algebroids appeared for the first time in \cite[section 3.3]{BarHek}. Note that $E/(\Ker\anchor)^\bot\cong E/T^*M$.
    
    Let $G$ be a semisimple Lie group with Lie algebra $(\mathfrak{g},[\cdot,\cdot]_\mathfrak{g})$, equipped with its Killing form $\langle\cdot,\cdot\rangle_\mathfrak{g}$ (see example \ref{ex:semisimple}). Let $\pi:P\to M$ be a $G$-principal bundle and denote by $\Ad P=P\times_{\Ad}\mathfrak{g}\to M$ the adjoint bundle of $P\to M$ (see \cite[proposition 5.1.6]{Neeb}). The fibers of $\Ad P\to M$ are isomorphic to $\mathfrak{g}$ and we can extend $[\cdot,\cdot]_{\mathfrak{g}}$ to the whole $\Gamma(\Ad P)$ by setting $[u,v]_\mathfrak{g}\big|_x=[u_x,v_x]_\mathfrak{g}$ for all $u$, $v\in\Gamma(\Ad P)$ and $x\in M$, promoting in this way $\Ad P\to M$ to a Lie algebra bundle. Also, the Lie algebra $\mathfrak{g}$ being quadratic, we can extend $\langle\cdot,\cdot\rangle_\mathfrak{g}$ to the whole $\Gamma(\Ad P)$ by setting  $\langle u,v\rangle_\mathfrak{g}\big|_x=\langle u_x,v_x\rangle_\mathfrak{g}$ for all $u$, $v\in\Gamma(\Ad P)$ and  $x\in M$, promoting in this way $\Ad P\to M$ to a quadratic Lie algebra bundle. We also note that $\Gamma(\Ad P)\cong\smooth(P,\mathfrak{g})^G$ (see for instance \cite[proposition 1.6.3]{Neeb}).

    Let $\omega\in\Omega^1(M,\mathfrak{g})$ be a (principal) connection on $P\to M$, with curvature $R=\difDeRham\omega+\frac{1}{2}[\omega\wedge\omega]_{\mathfrak{g}}\in\Omega^2(M,\Ad P)$ (see \cite[sections 5.4 and 6.1]{Neeb}). Denote by $\nabla$ the (linear) connection on $\Ad P\to M$ associated to $\omega$ (see \cite[chapter 3, section 1]{MR1393940}). Moreover suppose that there exists a $H\in\Omega^3(M)$ such that $\difDeRham H=\frac{1}{2}\langle R\wedge R\rangle_\mathfrak{g}$. With these notations, we consider the vector bundle $T^*M\oplus Q\oplus TM\to M$ where $Q=\Ad P$ (equipped with the bracket $[\cdot,\cdot]_\mathfrak{g}$ and the inner product $\langle\cdot,\cdot\rangle_\mathfrak{g}$), the connection $\nabla$, the curvature $2$-form $R$ and the $3$-form $H$. According to \cite[proposition 3.2]{BarHek} the standard Courant algebroid associated to these data is heterotic for the Atiyah Lie algebroid of $\pi:P\to M$, and conversely any heterotic Courant algebroid is of this form after a choice of dissection has been made. We will denote this standard Courant algebroid simply by $\mathcal{H}$. In this case $\mathcal{F}=\mathcal{T}_M$ and $\mathcal{Q}$ is $(Q\to M,\anchor=0,[\cdot,\cdot]_Q=[\cdot,\cdot]_\mathfrak{g},\langle\cdot,\cdot\rangle_Q=\langle\cdot,\cdot\rangle_\mathfrak{g})$; $\Omega^\bullet(\mathcal{F})\cong\Omega^\bullet(M)$, and according to \cite[proposition 5.3.4]{Neeb} we have an isomorphism of $\smooth(M)$-modules $\Omega(\mathcal{F},Q)\cong\Omega^\bullet(P,\mathfrak{g})_\text{bas}$ where the second $\smooth(M)$-module corresponds to the module of differential forms on $P$ taking values in $\mathfrak{g}$ which are \emph{basic}, that is, those which are both zero on $\Ker(\difDeRham\pi)$ (\emph{horizontal} vector fields) and $G$-invariant.

    The results of the two previous sections can be applied to this example, the formulas are essentially the same so we will not repeat them here. In particular, concerning the group $\Aut(\mathcal{H})$, we recover the result obtained in \cite[proposition 4.7]{GFRYT}.
\end{example}

\def\cprime{$'$}

\end{document}